\documentclass[leqno,11pt]{amsart}
\usepackage{amssymb, amsmath}
\usepackage[usenames,dvipsnames]{xcolor}

\usepackage{pst-node}
\usepackage{tikz-cd}

\usepackage{color,ulem,textcomp}

\numberwithin{equation}{section}
\usepackage{enumerate} 
\usepackage{graphicx}
\usepackage{cancel}

\theoremstyle{definition}
\theoremstyle{plain}
\newtheorem{thm}{Theorem}[section]
\newtheorem{Prop}[thm]{Proposition}
\newtheorem{lem}[thm]{Lemma}
\newtheorem{cor}[thm]{Corollary}

\newtheorem{Thm}{Theorem}[section]
\newtheorem{Rmk}[thm]{Remark}
\newtheorem{Lem}[thm]{Lemma}
\newtheorem{Def}[thm]{Definition}

\def\N {\mathbf{N}}

\def\T {\mathbb{T}}

\def\Z {\mathbf{Z}}
\def\R {\mathbf{R}}

\def\cA {\mathcal{A}}
\def\cB {\mathcal{B}}
\def\cC {\mathcal{C}}

\def\cD {\mathcal{D}}

\def\cG {\mathcal{G}}
\def\cL {\mathcal{L}}

\def\cM {\mathcal{M}}
\def\cN {\mathcal{N}}

\def\cZ {\mathcal{Z}}

\def\cR {\mathcal{R}}
\def\cT {\mathcal{T}}
\def\cW {\mathcal{W}}

\def\eps {{\varepsilon}}
\def\e {{\varepsilon}}

\def\k {{\kappa}}

\def\indc {{\bf 1}}

\def\d {{\partial}}

\def\cN {\mathcal{N}}
\def\cK {\mathcal{K}}
\def\cT {\mathcal{T}}

\newcommand{\ba}{\begin{aligned}}
\newcommand{\ea}{\end{aligned}}

\newcommand{\be}{\begin{equation}}
\newcommand{\ee}{\end{equation}}

\newcommand{\petit}{\xi}
\newcommand{\cO}{\mathcal{O}}
\newcommand{\cV}{\mathcal{V}}
\newcommand{\brown}{\Xi}
\def\R {\mathbb{R}}
\def\N {\mathbb{N}}
\def\Z {\mathbb{Z}}

\numberwithin{equation}{section}

\begin{document}


\title[]{Derivation of an Ornstein-Uhlenbeck process for a massive particle in a rarified gas of   particles}
\author{Thierry Bodineau, Isabelle Gallagher and Laure Saint-Raymond}
\address{
CMAP, Ecole Polytechnique, CNRS, Universit\'e Paris-Saclay, Route de Saclay, 91128 Palaiseau, France
}
\date{\today}

\begin{abstract}
We consider the statistical motion of a  convex rigid body in a gas 
 of $N$ smaller (spherical) atoms close to thermodynamic equilibrium. 
Because the rigid body is much bigger and heavier, it undergoes a lot of collisions leading to small deflections.
We prove that its velocity is described, in a suitable limit, by an Ornstein-Uhlenbeck process.
 
The strategy of proof relies on Lanford's arguments \cite{lanford} together with the pruning procedure from \cite{BGSR1} to reach diffusive times, much larger than the mean free time. 
 Furthermore, we need to introduce a modified dynamics to avoid pathological collisions of atoms with the rigid body: these  collisions, due to the geometry of the rigid body, require developing a new type of trajectory analysis.

\end{abstract}

\maketitle

\section{Introduction}

The first observation of the erratic motion of fragments of  pollen particles in a liquid   is attributed to the botanist Brown.
Following this observation, a lot of attention was devoted to understanding the physical mechanisms behind these fluctuations leading ultimately to the mathematical theory of Brownian motion. We refer to the review paper \cite{Dup} for a historical overview. 
The macroscopic motion of the massive particle is due to the fact that it undergoes many collisions with the atoms of the fluid and even though the microscopic dynamics is deterministic the motion observed on a macroscopic scale appears to be stochastic.
In a seminal paper, Holley \cite{Holley} studied a one-dimensional deterministic dynamics of a large particle
interacting with a bath of atoms represented by independent particles with a small mass. 
Each collision with an atom leads to a small deflection of the large particle and as the atoms are initially randomly distributed the successive collisions lead ultimately to a Brownian motion for the large particle displacement. This result was generalized to higher dimensions by D\"urr, Goldstein,  Lebowitz in  \cite{DGL_1} and to a particle which has a convex body in \cite{DGL_2}. The latter model follows asymptotically a
generalized Ornstein-Uhlenbeck diffusion jointly on the velocity and on the angular momentum.
Even though the atoms do not interact  one with the other, recollisions may occur between the large particle and some atoms,      leading to a memory effect. Asymptotically when the mass of the atoms vanishes, this effect was shown to be irrelevant in 
\cite{Holley, DGL_1, DGL_2} and the limiting dynamics is a Markov process.
Similar results were derived in \cite{Dobson_legoll_lelievre} when the gas is not at equilibrium. 
Note that in some different regimes or in presence of boundaries,   recollisions may have a macroscopic impact even when the body is in contact with an ideal gas.
This is for example the case in one-dimensional models where the correlations can be important 
\cite{DT} or in models of friction
\cite{Cavallaro_Marchioro_Pulvirenti, Cavallaro_Marchioro, KL}.

In this paper, we extend the framework studied in \cite{DGL_2}  to the case of a large particle with convex shape in contact with a gas of interacting atoms modelled by hard spheres in the Boltzmann-Grad scaling.  We prove  that the distribution of this particle is close to the solution of    a linear Boltzmann equation whose underlying process is asymptotically an Ornstein-Uhlenbeck process.

\subsection{The microscopic model}
\label{microscopicmodel}

We consider, in~$d = 2$ space dimensions,~$N$ spherical particles of mass~$m \ll 1$ and diameter~$\eps$ (from now on called atoms), and one  massive particle (the rigid body) of mass~$M=1$ and size~$\eps/\alpha$ with $\eps \ll \alpha \ll 1$. More precisely the rigid body is a strictly convex body~$\Sigma$, which is rescaled by a factor~$\eps/\alpha$, and which is allowed to translate and rotate.
The dynamics takes place in the periodic domain ${\mathbb T}^2 = [0,1]^2$.

We denote by~$\hat V_N:=(\hat v_1,\dots ,\hat v_N) \in {\mathbb R}^{2 N}$ the collection of velocities of all the atoms, and by~$X_N:=(x_1,\dots ,x_N) \in{\mathbb T}^{2 N}$ the positions of their centers.
Without loss of generality, we assume that the atoms have no angular momentum, as spherical particles do not exchange any angular momentum. 

The rigid body is described by the position and velocity~$(X,V) \in{\mathbb T}^2 \times {\mathbb R}^2$  of its center of mass~$G$, and by its orientation and its angular velocity $(\Theta, \hat \Omega) \in  {\mathbb S} \times {\mathbb R}$.
 If~$P$ is a point on the boundary of the rigid body, we denote by  $n$ the unit outward normal vector to the rigid body at point~$P$ and we 
 locate~$P$ by a vector~$  \displaystyle  r :=\frac\eps\alpha  {GP} \in {\mathbb R}^2$.
 Since the rigid body is not deformable, the position and the normal are obtained by 
applying a rotation~$R_\Theta$ of angle~$\Theta$ (see Figure \ref{fig: rigid body})
\begin{equation}
\label{defRtheta}
r_\Theta = R_\Theta   r \quad  \hbox{ with } \quad   r  \in \d   \Sigma \, , \quad n_\Theta=R_\Theta   n  \, .
\end{equation}
In particular, we write
 \begin{equation}\label{parametrizationboundary}
X_P:= X +\frac\eps\alpha r_\Theta \, .
\end{equation}
The velocity of~$P$ is
\begin{equation}
\label{vitesse locale}
V_P:=  V + \frac\eps\alpha \hat\Omega   r_\Theta ^\perp\, , 
\,  \mbox{ with } \, r^\perp:=(- r_2,   r_1) \, .
\end{equation}
 The  boundary $\d \Sigma$ of the body~$  \Sigma$  is described from now on by its arc-length which we denote by~$  \sigma \in [0,L]$ where~$L $ is the perimeter of~$\d  \Sigma$.  
We further assume that the  curvature~$\sigma \mapsto   \kappa (\sigma) $ of $\d \Sigma $ never  vanishes and we denote  
\begin{equation}\label{defkappa}
  \kappa_{min} := \min_{\sigma \in [0,L]} \;  \kappa (\sigma) \, .
\end{equation}
Finally we assume that $\d \Sigma$
 is included in a   ball of radius
\begin{equation}\label{defrmax}
r_{max} := \max_{\sigma \in  [0,L]} \; r(\sigma)\,.
\end{equation}

\begin{figure} [h] 
\centering
\includegraphics[width=4cm]{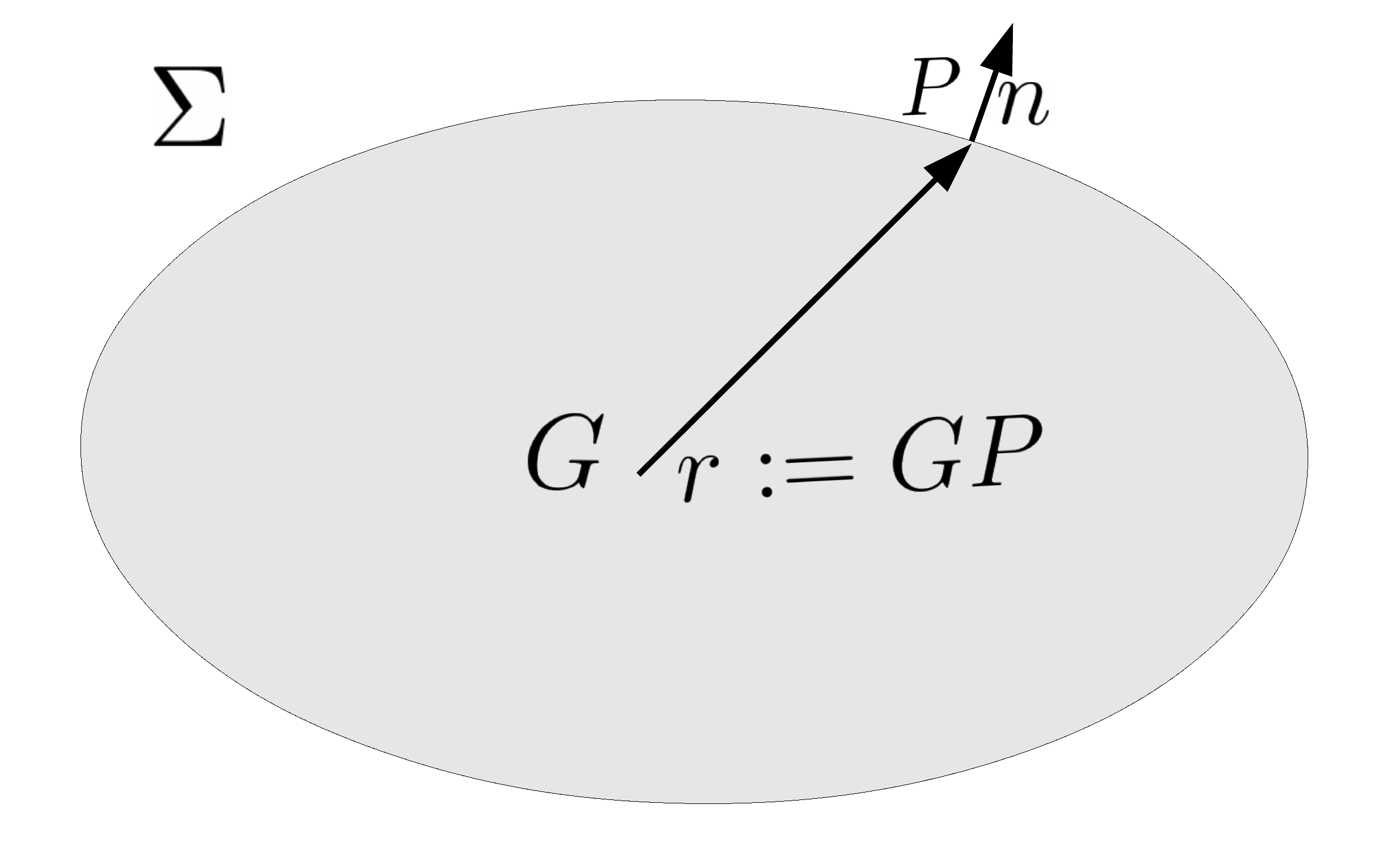}
  \qquad\qquad\qquad  \includegraphics[width=4.8cm]{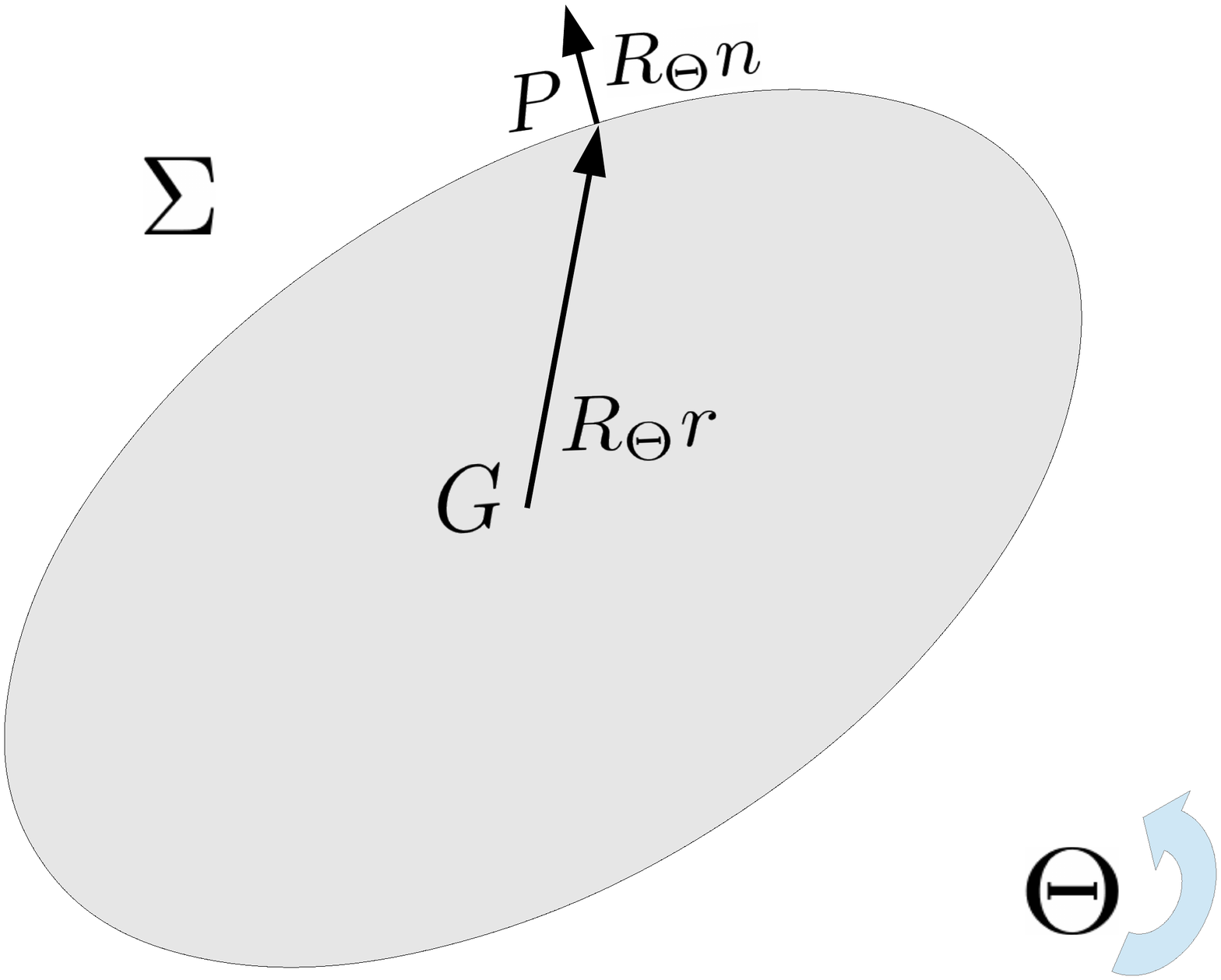} 
\caption{On the left, the reference rigid body~$  \Sigma$, on the right the rotated rigid body~$R_\Theta   \Sigma$}
\label{fig: rigid body}
\end{figure}

\subsection{Laws of motion}

Now let us describe the dynamics of the rigid body-atom system. As long as there is no collision, the centers of mass of the atoms and of the rigid body move in straight lines  and the rigid body rotates, according to the equations
\begin{equation}
\label{hard-spheres1}
\begin{aligned}
{dx_i\over dt} &= \hat  v_i \,,\quad {d\hat v_i\over dt} =0\, , \qquad  \forall i \leq N \, ,   \\
{dX\over dt} &=  V \,,\quad {dV\over dt} =0 \, , \quad {d \Theta \over dt} =\hat \Omega\, , \quad {d\hat\Omega \over dt} =0 \, ,
\end{aligned}
\end{equation}
since the moment of inertia of  a two dimensional body is constant.

Binary collisions are of two types.
If the atoms $i,j$ encounter, then their velocities $\hat v_i, \hat v_j$ are modified through the usual laws of specular reflection
\begin{equation}
\label{hard-spheres2'}
\begin{aligned}
\left. \begin{aligned}
\hat v_i'& :=\hat  v_i - \frac1{\eps^2} (\hat v_i-\hat v_j)\cdot (x_i-x_j) \, (x_i-x_j)   \\
\hat v_j'& := \hat v_j + \frac1{\eps^2} (\hat v_i-\hat v_j)\cdot (x_i-x_j) \, (x_i-x_j)
\end{aligned}\right\}
\quad  \hbox{ if } |x_i(t)-x_j(t)|=\eps\,.
\end{aligned}
\end{equation}
Atoms such that~$(\hat v_i-\hat v_j)\cdot (x_i-x_j)<0$ are said to be incoming, and after collision they are outgoing since~$(\hat v_i'- \hat v_j')\cdot (x_i-x_j)>0$.

\smallskip
If the rigid body has a collision with the atom $i$ at the point $\frac\eps\alpha r_\Theta$, meaning
$$
x_i(t)-X(t)-\frac\eps\alpha r_{\Theta(t)}=\frac\eps2  n_{\Theta(t)} \, ,
$$
the velocities $V,\hat v_i$ and the momentum~$\hat\Omega$ become (see Appendix~\ref{prooflawsmotion} for a proof) 
\begin{equation}\label{eq: collision laws 0}
\begin{aligned}
 \begin{aligned}
\hat v_i'&=\hat v_i + {2\over A+1} (V+ \frac\eps\alpha \hat \Omega \, r_\Theta^\perp - \hat v_i) \cdot n_\Theta \, n_\Theta   \\
 V'&=V  -{2m\over M(A+1)} (V+ \frac\eps\alpha \hat \Omega \, r_\Theta^\perp - \hat v_i) \cdot n_\Theta\,  n_\Theta \\
\hat\Omega'&=\hat\Omega  -{2m\over (A+1)\hat I}  \frac\eps \alpha   ( n_\Theta \cdot r^\perp_\Theta  )  \;  (V+ \frac\eps\alpha \hat \Omega \,  r_\Theta^\perp - \hat v_i) \cdot n_\Theta 
\end{aligned}
\end{aligned}
\end{equation}
with $\hat  I >0$ the moment of inertia and 
\begin{equation}
\label{eq: A 0}
A:= \frac m M + \frac m{\hat I}  \left(\frac\eps\alpha\right)^2  \big (  n \cdot r^\perp)^2  \, .
\end{equation}
The mass $M$ of the rigid body has been kept to stress the homogeneity of the coefficients, but later on we shall replace it by 1.

As in the atom-atom case, the atom and the rigid body are incoming if
$$(\hat v_i - V+ \frac\eps\alpha \hat \Omega \, r_\Theta^\perp ) \cdot n_\Theta<0$$ 
and after  scattering  one checks easily that
$$(\hat v_i' - V'+ \frac\eps\alpha \hat \Omega' \, r_\Theta^\perp ) \cdot n_\Theta>0$$ so the particles are outgoing. Recall that~$V+ \frac\eps\alpha \hat \Omega \, r_\Theta^\perp - \hat v_i$ is the relative velocity at  the impact point.

The following quantities are conserved  when an atom and the rigid body collide:
\begin{equation}
\label{eq: conservation laws}
\begin{cases}
\mathcal P:=m \hat v+MV&  \mbox{(total  momentum)} \\
\mathcal E  := \displaystyle\frac12 \big (\displaystyle   (m |\hat v|^2 +M |V|^2) +  \hat I \hat \Omega^2 \big) &  \mbox{(total energy)} \\
\hat {\mathcal I}  := \hat  I  \hat\Omega - \frac\eps\alpha M  V  \cdot r_\Theta^\perp &  \mbox{(angular momentum at   contact point)\,.} 
\end{cases}
\end{equation}

\subsection{Scalings}

The parameters $N$ and $\eps$ are related by the Boltzmann-Grad scaling~$N \eps = 1$ in dimension $d=2$. 
With this scaling, the large $N$ asymptotics describes a rarefied gas. Even if the density of the gas 
is asymptotically zero, the mean free path for the atoms is 1. 

The rigid body being larger, it will encounter 
roughly $\alpha^{-1} \hat v_{typ} $ collisions per unit time with~$\alpha \ll 1$ and $ \hat v_{typ}$ the typical relative velocity of the atoms. In the following, we   consider the joint asymptotics~$N \to \infty, \eps \to 0$ and~$\alpha \to 0$ with $\alpha$ vanishing not faster than~$1/ ({\log \log N})^\frac14$ (this restriction will be clear later on in the computations).

\medskip

The rigid body has mass $M=1$, the atoms are much lighter $m\ll1$.
As a consequence of the equipartition of energy (cf (\ref{eq: conservation laws})), 
the typical  atom velocities $\hat v_{typ}  = O(m^{-1/2}) $ are expected to be much larger than the rigid body velocity which is of order 1.
 Each collision with an atom deflects very little the rigid body. We expect to get asymptotically a diffusion with respect to the velocity variable provided that 
 $$m = \alpha^2\,.$$
From now on we therefore  rescale the atom velocities by setting
\begin{equation}
\label{scaling1}
v := m^\frac12 \hat v =\alpha \hat v \, .
\end{equation}

Similarly due to the small size $\eps/\alpha$ of the rigid body, the moment of inertia~$\hat I$ is very small, namely of the order of~$(\eps/\alpha)^2$. We therefore rescale the  moment of inertia and the angular velocity by defining
\begin{equation}
\label{eq: rescaled momentum}
I:= \left(\frac \alpha \eps\right)^2 \hat I \quad \mbox{and} \quad \Omega := \frac  \eps\alpha \hat \Omega \, ,
\end{equation}
so that both~$I$ and~$\Omega$ are quantities of order one. We accordingly set
$$
\mathcal I  :=    I   \Omega - M V\cdot r^\perp_\Theta \,  .
$$
After rescaling the collision laws~(\ref{hard-spheres2'}) and~\eqref{eq: collision laws 0}  
become
\begin{equation}
\label{hard-spheres2}
\begin{aligned}
\left. \begin{aligned}
  v_i'& :=   v_i - (  v_i-  v_j)\cdot n \, n   \\
  v_j'& :=   v_j +(  v_i-  v_j)\cdot n \, n
\end{aligned}\right\}
\quad  \hbox{ if } x_i(t)-x_j(t) =\eps n
\end{aligned}
\end{equation}
and  if~$x_i(t)-X(t)-\frac\eps\alpha r_{\Theta(t)}=\frac\eps2 n_{\Theta(t)}$,
\begin{equation}
\label{eq: collision laws}
\begin{aligned}
\begin{aligned}
v_i'&=v_i + {2 \over A+1} (\alpha V+\alpha  \Omega   r_\Theta^\perp - v_i) \cdot n_\Theta  \, n_\Theta\\
 V'&=V -{2\over A+1 } (\alpha^2 V+ \alpha^2 \Omega r_\Theta^\perp - \alpha v_i) \cdot n_\Theta\,  n_\Theta  \\
\Omega'&= \Omega  -{2\over (A+1)I}     ( n \cdot r^\perp  )   (\alpha^2 V+ \alpha^2 \Omega  r_\Theta^\perp - \alpha v_i) \cdot n_\Theta   
\end{aligned}
\end{aligned}
\end{equation}
with 
\begin{equation}
\label{eq: A}
A:= \alpha^2 \Big (1  + \frac 1I     (  n \cdot r^\perp)^2  \Big)  \, .
\end{equation}

 \subsection{Initial data and the Liouville equation}

To simplify notation, we   use throughout the paper
$$
Y:= (X,V,\Theta,\Omega)
\quad \text{and} \quad
Z_N:=(X_N,V_N)\,.
$$
 We denote by $f_{N+1}(t, Y, Z_N)$
the distribution of the $N+1$ particles at time $t \geq 0$. 
This function is symmetric   with respect to the variables~$Z_N$ since we assume that the atoms are undistinguishable.
It satisfies the Liouville equation, recalling that rescaled velocities are defined by~(\ref{scaling1}),(\ref{eq: rescaled momentum}),
$$
\partial_t   f_{N+1} + V\cdot \nabla_X   f_{N+1} +\frac1 \alpha \sum_{i=1}^N v_i \cdot \nabla_{x_i}   f_{N+1} +\frac \alpha\eps \Omega  \,  \partial_{\Theta}   f_{N+1} = 0 \,  ,
$$
in the domain
$$
{\mathcal D}_\eps^{N+1}:=  \Big\{ (Y,Z_N)     \, / \, \forall i \neq j \, ,\quad |x_i - x_j| > \eps \quad \mbox{and} \quad   d\big(  x_i , X+ \frac \eps\alpha R_\Theta    \Sigma\big)  > \frac\eps{2} \Big\} \, .
$$
Following the strategy in \cite{alexanderthesis,GSRT}, one can prove that the dynamics is well defined for almost all initial data~: the main difference with the system of hard spheres is that here, because of the geometry of the rigid body, the interaction of a single atom with the rigid body could involve many collisions. 
Using an argument similar to the one which will be developed in Section \ref{conditioning-geometry}, one can show that the collisions between the rigid body and the atoms can be controlled   for almost all initial data.

We introduce two types of Gaussian measures 
\begin{equation}
\label{eq: Gauss}
\begin{aligned}
\forall v \in \mathbb{R}^2, \qquad  
M_{\beta  } (v) &: = \frac{\beta  }{2\pi} \exp \Big( -\frac{  \beta}{ 2}|v|^2  \Big)\, , \\
\forall (V,\Omega) \in \mathbb{R}^2 \times \mathbb{R} \, , \qquad  
M_{\beta, I} (V,\Omega) &: = \frac \beta{2\pi} \left( \frac {\beta I}{2\pi} \right) ^{\frac 1 2}  \exp \Big(
-\frac \beta 2 \big(| V|^2 +  I  \Omega^2 \big) \Big)\,  .
\end{aligned}\end{equation}
We introduce the Gibbs measure on the $N+1$ particle system 
\begin{equation}
\label{eq: Gibbs measure}
M_{\beta, I, N} (Y,Z_N) := \bar M_{\beta,I} ( Y)  \left( \prod_{i =1}^N M_{\beta  } (v_i) \right)
{{\mathbf 1}_{{  \mathcal D}_\eps^{N+1}} (Y,Z_N) \over  {\mathcal Z}_N} \, ,
\end{equation}
 with
$$
\bar M_{\beta,I} ( Y) := 
\frac{1}{2 \pi} M_{\beta,I} (V, \Omega) 
$$ 
 and
where  the normalisation factor 
$$
{\mathcal Z}_N:=\int {\mathbf 1}_{{\mathcal D}_\eps^{N+1}} (Y,Z_N) \, dXdX_N 
$$ 
is computed by using the rotation and translation invariance of the system so that 
the only relevant part of the integral involves the spatial exclusion.  
The measure $M_{\beta, I, N}$
is a stationary solution for the Liouville equation, i.e. a thermal  equilibrium of the system.
Here,  we choose for initial data a small perturbation around this equilibrium, namely  
\begin{equation}
\label{defdata}
\begin{aligned}
f_{N+1,0} (Y,Z_N) :=  g_0 (Y) \; M_{\beta, I,N}  (Y,Z_N)    \, , 
\end{aligned}
\end{equation}
with 
\begin{equation}
\label{g0 borne}
\begin{aligned}
\| g_0 \|_{L^{\infty}} \leq C \, ,\quad  \|\nabla g_0 \|_{L^{\infty}} \leq C
\quad \text{and} \quad
\int  \bar M_{\beta, I} (Y) g_0(Y) \, dY = 1\,.
\end{aligned}
\end{equation}
This perturbation modifies only the distribution of the rigid body, however this initial modification will drive
 the  whole system out of equilibrium at later times.
 Note  that the uniform bounds on $g_0$ could also be slowly diverging with $\alpha$ to allow for   the limiting distribution to be a Dirac mass (see \cite{BGSR1}).

\subsection{Main result}

Our goal is   to describe the evolution of the rigid body distribution in a rarified gas starting from the measure $f_{N+1,0}$ defined in~(\ref{defdata}), i.e. close   to equilibrium.
The distribution of the rigid body is given by 
the first marginal 
$$
 f_{N+1}^{(1)} (t,Y):= \int   f_{N+1} (t,Y, Z_N) \, dZ_N \, .
$$
Our key result is a quantitative approximation of the distribution of the mechanical process by a linear Boltzmann equation. We define the operator 
{\color{black}
\begin{align}
\label{defLalpha}
{\mathcal L}_\alpha g (Y)& := 
\frac1\alpha \int _{[0,L_{ \alpha}] \times {\mathbb R}^2}  M_{\beta } (v)  
 \Big( g(Y') \big(  (\frac1{\alpha}v' -V'- \Omega'      r_{ \Theta}^\perp
) \cdot    n_{ \Theta}\big)_-    \\
  & \qquad \qquad \qquad \qquad \qquad -  g(Y) \big(  (\frac1{\alpha}v
-V- \Omega    r_{ \Theta}^\perp
\big ) \cdot    n_{ \Theta} \big)_-\Big)\, d  \sigma_\alpha dv	\, ,	 \nonumber 
\end{align}
}with $Y= (X,V,\Theta,\Omega)$ and $Y' = (X,V',\Theta,\Omega')$ as defined in~(\ref{eq: collision laws}) and where~$L_\alpha$ is the perimeter of the enlarged body
\begin{equation}
\label{enlargedbody}
  \Sigma_\alpha:= \big \{y  \; \big| \,  d(y,   \Sigma) \leq  \alpha/2  \big\} \, ,
\end{equation}
and~$ \sigma_\alpha$ is the arc-length on~$\d   \Sigma_\alpha$.

\begin{Thm}
\label{prop-approximate tensorization}
Assume that the particles are initially distributed according to $f_{N+1,0}$ defined in~{\rm(\ref{defdata})-(\ref{g0 borne})} and consider 
 the joint limit  $N \to \infty, \eps \to 0$ and $\alpha \to 0$ with 
$$
N \eps =1\, , \qquad \alpha^4  {\log \log N}  \gg 1 \, .
$$
Then for any time $T\geq 1$, the distribution $  f_{N+1}^{(1)}$ of the rigid body satisfies
 \begin{equation}
\label{mainpart-est}
\lim_{N\to \infty}\big\|   f^{(1)}_{N+1}(t)- \bar M_{\beta, I}    g_\eps(t) \big\|_{L^\infty([0,T];L^1(\T^2\times \R^2\times {\mathbb S}\times \R))} = 0
\end{equation}
where $g_\eps$ satisfies the linear Boltzmann equation
\begin{equation}
\label{eq: linear boltz}
\partial_t g_\eps + V \cdot \nabla_X g_\eps + \frac\alpha \eps \Omega  \, \partial_\Theta g_\eps \\
=  {\mathcal L}_\alpha g_\eps\,.
\end{equation}
\end{Thm}
When $\eps$ and $\alpha$ tend to 0, the solution of the linear Boltzmann equation \eqref{eq: linear boltz}
converges to the solution of a hypoelliptic equation  combining the transport  with the diffusion operator
\begin{align}
\label{eq: diffusion operator}
\cL =   \frac1\beta  \big(\frac{L }2  \Delta_V +  \frac{\cK}{I^2}  \partial^2_\Omega \big)    -\frac{L }2V  \cdot \nabla_V    -  \frac{\cK}{I}  \Omega \partial_\Omega \,  ,
\end{align}
where recall that $L $ stands for the perimeter of $\Sigma$ and  where
$$
\cK  := \int _{0}^{{L }}   (r \cdot n^\perp) ^2 d  \sigma \, .
$$
Note that~$\cK = 0$ in the case when the rigid body is a disk. 

We can therefore also deduce the following behavior of the rigid body distribution for all times. 
\begin{Thm}
\label{thm: convergence density}
Assume that the particles are initially distributed according to $f_{N+1,0}$ defined in~{\rm(\ref{defdata})}-{\rm(\ref{g0 borne})}
 and consider 
 the joint limit  $N \to \infty, \eps \to 0$ and $\alpha \to 0$ with 
$$
N \eps =1\, , \qquad \alpha^4  {\log \log N}  \gg 1 \, .
$$
Then for any time $T\geq 1$, the distribution $  f_{N+1}^{(1)}$ of the rigid body converges to $\bar M_{\beta,I} \, g$,  weak-$\star$ in~$L^\infty ([0,T]\times \T^2 \times \R^2\times {\mathbb S}\times \R)$,  where $g$ is the solution of 
\begin{align}
\label{eq: distribution limite}
\d_t g +V\cdot \nabla_X g = a  \cL g
\quad \text{with} \quad
a := \left( \frac{8}{\pi \beta}\right)^{1/2},
\end{align}
starting from $g_0$. 
\end{Thm}

The fluctuations of the whole path of the rigid body can also be controlled and the limiting process will be  the  Ornstein-Uhlenbeck process  $\cW(t) =(\cV (t) , \cO (t))$  with generator $a \cL$ given in~(\ref{eq: diffusion operator}) 
\begin{equation}
\label{eq: OU process}
\begin{split}
d X(t) &= \cV(t) dt \\
d \cV(t) & =  -   aL  \cV  (t) dt + \sqrt{\frac{2  aL }{\beta}} \; dB_1(t) \\
d \cO (t) & = -  a \frac{\cK}{I} \cO (t) dt + \sqrt{\frac{2 a \cK}{\beta I}} \; dB_2(t) 
\end{split}
\end{equation}
where $B_1 \in {\mathbb R}^2, B_2 \in \mathbb R$ are two independent Brownian motions.
Initially $\cW(0) =(\cV (0) , \cO (0))$ is distributed according to $M_{\beta, I}$ defined in~\eqref{eq: Gauss}. {\color{black}There is no limiting process for the angles which are rotating too fast as the angular momentum $\hat \Omega$ has been rescaled by a factor $ {\varepsilon}/{\alpha}$.}

\medskip

In the joint Boltzmann-Grad limit and $\alpha \to 0$, the velocity and the angular momentum of the rigid body
converge to the diffusive process $\cW$. 
\begin{Thm}
\label{thm: convergence process}
Let $\brown( t) = ( V (t) , \Omega(t) )$ be  the microscopic process associated with the rigid body and 
starting from the equilibrium measure $M_{\beta,I,N}$ defined in~{\rm\eqref{eq: Gibbs measure}}.
For any time $T>0$, the process $\brown$  converges in law in $[0,T]$ to the Ornstein-Uhlenbeck process  $\cW$ defined by {\rm(\ref{eq: OU process})}  in the joint limit 
$N \to \infty, \eps \to 0$ and $\alpha \to 0$ with 
$$
N \eps =1\, , \qquad \alpha^4  {\log \log N}  \gg 1 .
$$
\end{Thm}
Compared to  \cite{DGL_2}, the limiting process \eqref{eq: OU process} is somewhat simpler 
as the  velocity and the angular momentum fluctuations of the rigid body decouple.
This comes from the fact that the  size of the rigid body is scaled with $\eps/\alpha$ and when $\eps/\alpha $ tends to 0, this induces a very fast rotation  (see~\eqref{eq: rescaled momentum}) which averages out the cross correlations between the velocity and the angular momentum.
%
%

\section{Formal asymptotics and structure of the proof}

\subsection{The BBGKY hierarchy}

To prove  Theorem \ref{prop-approximate tensorization}, we need to write down the equation on~$  f_{N+1}^{(1)} $, which  involves the second marginal, so we are led as usual in this context to studying the full BBGKY hierarchy on the marginals (denoting~$z_i:=(x_i,v_i)$)
$$
\forall s\leq N+1\, , \qquad
  f_{N+1}^{(s)} (t,Y,Z_{s-1}):= \int   f_{N+1} (t,Y, Z_N) \, dz_{s} \dots dz_N \, .
$$
Recall that $  f_{N+1}$ is the distribution over the $N+1$ particles and we have~$  f_{N+1}^{(N+1)} =   f_{N+1}$.
 Applying Green's formula leads to the scaled equation
\begin{equation}
\label{eq: BBGKY}
\begin{aligned} \partial_t   f_{N+1}^{(s)}  + V \cdot \nabla_X   f_{N+1}^{(s)} + \frac 1\alpha \sum_{i=1}^{s-1} v_i \cdot \nabla_{x_i}   f_{N+1}^{(s)}+ \frac \alpha\eps \Omega \,   \partial_{\Theta}    f_{N+1}^{(s)} \\=  C_{s,s+1}   f_{N+1}^{(s+1)} +D_{s,s+1}  f_{N+1}^{(s+1)},
\end{aligned}
\end{equation}
where~$C_{s,s+1}$ is the usual collision operator related to collisions between two atoms
$$
\begin{aligned}
(C_{s,s+1}   f_{N+1}^{(s+1)}) (Y, Z_{s-1})& := (N-s+1)  \eps 	\\
& \times\sum_{i=1}^{s-1} \int_{{\mathbb S} \times {\mathbb R}^2}    f_{N+1}^{(s+1)} ( Y, Z_{s-1}, x_i + \eps \nu_s,v_{s})  \frac1\alpha (v_s-v_{i} ) \cdot \nu_s \, d\nu_s dv_{s} \, ,
\end{aligned}
$$
while~$D_{s,s+1}$ takes into account collisions between the rigid body and the atoms 
$$
\begin{aligned}
& (D_{s,s+1}   f_{N+1}^{(s+1)})  (Y, Z_{s-1}):=   (N-s+1)  \frac \eps \alpha \\
& \qquad \times\int_{[0,L_\alpha] \times {\mathbb R}^2}    f_{N+1}^{(s+1)} ( Y, Z_{s-1}, X  + \frac \eps \alpha \,   r_{\alpha, \Theta} ,v_{s})\big(  \frac1\alpha v_{s}-V
- \Omega      r_{\alpha, \Theta}^\perp
\big)
\cdot    n_{\alpha, \Theta}  \, d  \sigma_\alpha dv_{s} \, ,
\end{aligned}
$$
where we recall that the subscript $\Theta$ denotes the rotation of angle~$\Theta$ as defined in~(\ref{defRtheta}), and~$r_\alpha,n_\alpha $ are functions of the arc-length $  \sigma_\alpha$    on $\d   \Sigma_\alpha$. 

Note that the set~$ \Sigma_\alpha$ defined in~(\ref{enlargedbody})  is introduced to take into account the radius of the atoms. Indeed at the collision, the center of the atom is at distance $\eps/2$ of the body $ ({\eps}/{\alpha} )   \Sigma$, thus after rescaling by~$ {\alpha}/{\eps} $, the center of the atom is at distance~$\alpha/2$ of $   \Sigma$.
Ultimately $\alpha$ will tend to~0 and $\d   \Sigma_\alpha$  to $\d   \Sigma$ since $\Sigma$ is assumed to be smooth and convex.

The structure of the collision kernel in $D_{s,s+1}$ can be understood as follows: such a collision occurs when an atom among those labeled from~$s$ to~$N$ (say~$s$) has its center~$x_{s}$ such that~$\frac\alpha \eps R_{-\Theta} (x_s - X)$ belongs to~$\d   \Sigma_\alpha$~:
$$
R_{-\Theta} (x_s - X) -  \frac \eps\alpha r_\alpha = 0\,.
$$
 The 
    normal to the corresponding surface $\d \cD_\alpha$  is   given by~$ (\displaystyle R_\Theta   n_\alpha, -\displaystyle R_\Theta   n_\alpha , -  \frac\eps\alpha  r _\alpha^\perp \cdot   n_\alpha) $ in the~$(x_s, X, \Theta)$ space.  Applying the Stokes theorem, we   obtain that for any function $\varphi$
  {$$
\begin{aligned}
& \int _{\cD_\alpha\times \R^5} (\frac1\alpha v_s\cdot \nabla_{x_s} +V\cdot \nabla_X + \frac \alpha \eps \Omega  \, \partial_\Theta) \varphi (x_s,v_s,X,V,\Theta,\Omega) dx_sdv_sdY\\
& =  \int _{\cD_\alpha\times \R^5} \nabla_{x_s,X,\Theta}\cdot \big ( (\frac1\alpha v_s,V,  \frac \alpha \eps\Omega) \varphi \big)(x_s,v_s,X,V,\Theta,\Omega) dx_sdv_sdY \\
& =  \int _{\d \cD_\alpha\times \R^5} \big( \frac1\alpha v_{s}-V-
 \Omega  R_\Theta    r_\alpha ^\perp
\big)
\cdot  \frac{ R_\Theta  n_\alpha}{\sqrt {2+ (   \frac\eps\alpha  r _\alpha ^\perp \cdot   n_\alpha )^2 }} \varphi  (x_s,v_s,X,V,\Theta,\Omega) d\nu_\alpha  dv_sdV d\Omega
\end{aligned}
$$
where~$\nu_\alpha $ is the four-dimensional unit surface measure on the set~$\partial \cD_\alpha$. Parametrizing this set by~$   \sigma_\alpha,X,\Theta$ with~$d \sigma_\alpha$ the elementary arc-length on~$\d   \Sigma_\alpha $  we find that 
$$
 d\nu _\alpha  =  \frac\eps\alpha   \sqrt {2+ (    \frac\eps\alpha   r _\alpha ^\perp \cdot   n_\alpha )^2 } \, d  \sigma_\alpha d X d \Theta
$$
so finally we   obtain that for any function $\varphi$
$$
\begin{aligned}
& \int _{\cD_\alpha\times \R^5} (\frac1\alpha v_s\cdot \nabla_{x_s} +V\cdot \nabla_X + \frac \alpha \eps  \Omega  \, \partial_\Theta) \varphi (x_s,v_s,X,V,\Theta,\Omega) dx_sdv_sdY\\
&= \frac\eps\alpha \int \! \! \varphi ( X + \frac \eps \alpha \,r_{\alpha, \Theta} ,v_{s}, X,V,\Theta,\Omega)\big(  \frac1\alpha v_{s}-V
- \Omega  r_{\alpha, \Theta}^\perp
\big)
\cdot n_{\alpha, \Theta} \, d  \sigma _\alpha dv_sdY \, .
\end{aligned}$$
This enables us to identify the contribution of the boundary term at a rigid body-atom collision.

\begin{Rmk}
Note that the integral could be reparametrized by the arc-length on $\d \Sigma$. We indeed have that
$$ r_\alpha = r +\frac\alpha2 n $$
which leads to 
$$ 
{d r_\alpha \over d\sigma} = - n^\perp + \frac\alpha 2 \kappa n^\perp,
$$
where $-n^\perp$ stands for the tangent.
In particular, both curves have the same tangent and the same normal, and we have the identity
$$\big(  \frac1\alpha v_{s}-V
- \Omega      r_{\alpha,\Theta}^\perp
\big)
\cdot    n_{\alpha,\Theta}= \big(  \frac1\alpha v_{s}-V
- \Omega      r_\Theta^\perp
\big)
\cdot    n_\Theta\,,$$
which implies that the cross section depends only on the normal relative velocity at the contact point.
\end{Rmk}

\medskip

As usual in this context, we now separate the collision operators according to post and pre-collisional configurations, using the collision laws (\ref{hard-spheres2})\eqref{eq: collision laws}. 
This is classical as far as~$C_{s,s+1}$ is concerned: we write, thanks to the boundary condition when two atoms collide,
$$
C_{s,s+1} = C_{s,s+1}^+- C_{s,s+1}^-
$$
with
$$
   \begin{aligned}
  & \big(  C_{s,s+1}^+   f_{N+1}^{(s+1)}\big) (Y,Z_{s-1})  :=  (N-s+1)  \frac \eps \alpha\\
  & \qquad \qquad  \times  \sum_{i=1}^{s-1} \int_{{\mathbb S} \times {\mathbb R}^2}    f_{N+1}^{(s+1)}(\dots, x_i, v_i',\dots , x_i+\eps \nu_s, v'_{s})   \big((v'_{s}-v'_i) \cdot \nu_s\big)_- d\nu_s dv_{s} \\
&\big(  C_{s,s+1}^-   f_{N+1}^{(s+1)}\big) (Y,Z_{s-1})  :=(N-s+1)   \frac \eps \alpha
  \\
  &\qquad \qquad    \times \sum_{i=1}^{s-1}\int_{{\mathbb S} \times {\mathbb R}^2}   f_{N+1}^{(s+1)}(\dots, x_i, v_i,\dots , x_i+\eps \nu_s, v_{s}) \big((v_{s}-v_i) \cdot \nu_s\big)_- d\nu_s dv_{s} \, .
 \end{aligned}
$$
Note that $ \big((v_{s}-v_i) \cdot \nu_s\big)_+  =  \big((v_{s}'-v_i') \cdot \nu_s\big)_- $.

\smallskip

In the case of the operator~$D_{s,s+1}$,
we also use the collision laws which provide the decomposition
\begin{equation}
\label{eq: collision operator rigid body}
D_{s,s+1} = D_{s,s+1}^+- D_{s,s+1}^-,
\end{equation}
with
\begin{equation}
\label{eq: operator D}
\begin{aligned}
&\big(  D_{s,s+1}^+   f_{N+1}^{(s+1)}\big) (Y,Z_{s-1})    :=   (N-s+1)  \frac \eps \alpha  \\
&\,   \times \int_{[0,L_\alpha] \times {\mathbb R}^2}  \!  \!  f_{N+1}^{(s+1)} ( Y', Z_{s-1}, X +   \frac \eps \alpha \,    r_{\alpha,\Theta} ,v'_{s}) \Big( \big (  \frac 1 \alpha v'_{s}-V'-
 \Omega'     r_{ \Theta} ^\perp  \big ) \cdot     n_{ \Theta} \Big)_-  \!  d  \sigma_\alpha dv_{s} \, ,  \\
&\big(  D_{s,s+1}^-  f_{N+1}^{(s+1)}\big) (Y,Z_{s-1})  :=(N-s+1) \frac \eps \alpha   \\
  &\,   \times\int_{[0,L_\alpha] \times {\mathbb R}^2}    f_{N+1}^{(s+1)} ( Y, Z_{s-1}, X  + \frac \eps \alpha \,    r_{\alpha, \Theta} ,v_{s}) \Big( \big (  \frac 1 \alpha v_{s}-V
- \Omega     r_{ \Theta} ^\perp  \big) \cdot     n_{\Theta} \Big)_-  d  \sigma_\alpha dv_{s} \, ,
\end{aligned}
\end{equation}
where we used that $\big (  \frac 1 \alpha v'_{s} -V'- \Omega'      r_{ \Theta} ^\perp 
\big ) \cdot     n_{\Theta}   = -  \big ( \frac 1 \alpha v_{s} -V- \Omega      r_{\Theta} ^\perp   \big ) \cdot   n_{ \Theta} $ and where we have written~$Y' = (X,V',\Theta,\Omega')$ and $(v'_{s},V',\Omega' )$ is the post-collisional configuration  defined by~(\ref{eq: collision laws}).

\subsection{Iterated Duhamel formula and continuity estimates}
\label{subsec: Iterated Duhamel formula and continuity estimates}

Using the hierarchy \eqref{eq: BBGKY}, the first marginal can be represented in terms
  of  the iterated Duhamel formula
\begin{equation}\label{iterated-Duhamel 00}
\begin{aligned}
f_{N+1}^{(1)}  (t) 
= {\bf S}_1(t) f_{N+1,0}^{(1)} +
\sum_{n=1}^{N}    \int_0^t \int_0^{t_1}\dots  \int_0^{t_{n-1}}  {\bf S}_1(t-t_1) ( C_{1,2} +  D_{1,2}) 
{\bf S}_2 (t_1-t_2) \dots  \\
\dots  {\bf S}_{1+n}(t_n)  f_{N+1,0}^{(1+n)}\: dt_n \dots   {\color{black}  dt_1} \, ,
\end{aligned}
\end{equation}

where  ${\bf S}_s$ denotes the group associated with~$\displaystyle V \cdot \nabla_X + \frac1\alpha\sum_{i=1}^{s-1} v_i \cdot \nabla_{x_i} + \frac \alpha \eps \Omega \, \partial_\Theta$ in $\cD_\eps^{s}$ with specular reflection on the boundary.	
To simplify notation, we define the operators $Q_{1,1} (t) = {\bf S}_1 (t)$
and for $s,n \geq 1$
\begin{equation}
\label{eq: operator Q}
\begin{aligned}
Q_{n,n+s} (t) :=    \int_0^t \int_0^{{\color{black} t_n}}\dots  \int_0^{{\color{black} t_{n+s-2}}} 
{\bf S}_n (t-t_{n}) ( C_{n,n+1} +  D_{n,n+1}) 
{\bf S}_{n+1} (t_{n}-t_{n+1})\\ \dots  {\bf S}_{n+s}(t_{n+s-1})     {\color{black} dt_{n+s-1} \dots dt_n}  \, ,
\end{aligned}
\end{equation}
so that
\begin{equation}
\label{eq: iterated-Duhamel}
f_{N+1}^{(1)} (t) =  \sum_{s=0}^N  Q_{1,1+s} (t)    f_{N+1,0}^{(1+s)} \,.
\end{equation}

\medskip

To establish uniform bounds on the iterated Duhamel formula~(\ref{eq: iterated-Duhamel}), we   use the estimates on  the initial data~\eqref{defdata} and the maximum principle for the Liouville equation to get 
\begin{equation}
\label{maxprinciple}
f_{N+1} (t) \leq  \| g_0\|_{L^\infty} M_{\beta,I,N}\,.
\end{equation}
Thus,
\begin{equation}
\label{estimatedata 0}
\begin{aligned}
f_{N+1} ^{(s)} (t,Y, Z_{s-1})&\leq  \| g_0\|_{L^\infty} M_{\beta,I,N}^{(s)}  (X,V,\Omega, Z_{s-1} )  \\
& \leq C^s \| g_0\|_{L^\infty} \; M_{\beta,I} (V,\Omega)  M_{\beta }^{\otimes (s-1)} (V_{s-1} ) \, ,
\end{aligned}
\end{equation}
where from now on~$C$ is a constant  which may change from line to line, and the upper bound in terms of the Gaussian measures \eqref{eq: Gauss} is uniform with respect to the positions. The factor~$C^s$ is due to the exclusion~${\mathcal Z}_s^{-1}$ in~$ M_{\beta,I,N}^{(s)}$.
This estimate can be combined with   
  continuity estimates  on the collision operators (see~\cite{GSRT, BGSR1}). 
As usual we     overestimate all contributions by considering rather the operators~$|C_{s,s+1}|$  and~$|D_{s,s+1}|$ defined by
$$
| C_{s,s+1} |f_{s+1}:=  \sum_{i=1}^s  (C_{s,s+1}^{+} +C_{s,s+1}^{-})f_{s+1}\, ,\quad | D_{s,s+1} |f_{s+1}:=  \sum_{i=1}^s  (D_{s,s+1}^{+} +D_{s,s+1}^{-})f_{s+1}    \, ,
$$
and the corresponding series operators $|Q_{s,s+n}|$.
Thanks to \eqref{estimatedata 0}, it is enough to estimate 
the norm of collision operators when applied to the reference Gaussian measures introduced in~\eqref{eq: Gauss}. The following result holds.
\begin{Prop}
\label{estimatelemmacontinuity 0} 
There is a constant~$C_1= C_1(\beta,I)$  such that for all $s,n\in \N^*$ and all~$t\geq 0$, the operator~$|Q|$   satisfies the following continuity estimate:  
$$
	\begin{aligned}
	|Q_{1,1+s}| (t)  \big(
	 M_{\beta,I,N}^{(s)} \big)  \leq \Big({C_1 t\over \alpha^2} \Big) ^{s} M_{3\beta/4,I} \,  .
	\end{aligned}
$$
\end{Prop}
The proof is standard and sketched in  Appendix~\ref{subsec : continuity estimates}.
 Note that this   estimate  is the key to the local wellposedness of the hierarchy (see~\cite{GSRT}) : it implies indeed that
the series expansion~(\ref{iterated-Duhamel 00}) converges (uniformly in $N$)
  on any time such that~$t  \ll \alpha^2 $.

\subsection{Probability of trajectories and the Duhamel series}
\label{appendixprobaduhamel}

We start by recalling how the series \eqref{eq: iterated-Duhamel}
can be interpreted in terms of a branching process.
This plays a key role in the analysis of the series as explained in~\cite{lanford,CIP,GSRT}.
A more detailed presentation will be given in Section \ref{subsec: iterated Duhamel formula and collision trees}.
The operator $Q_{1,1+s}$ defined in \eqref{eq: operator Q} 
can be described by collision trees with a root indexed by  the coordinates at time $t$ of the rigid body  to which we assign the label~0. 
\begin{Def}[Collision trees]
\label{trees-def}
Let $s \geq 1$ be fixed. An (ordered) collision tree  $a \in \cA_s$ is defined by a family $(a_i ) _{1\leq i \leq s}$ with $a_i \in \{0,\dots, i-1\}$.
 \end{Def}

We first describe the adjunction of new particles in the backward dynamics starting at time $t_0=t$
{\color{black} from the rigid body which can be seen as the root of the tree}.
Fix~$s \geq 1$,  a collision tree~$a \in \cA_s$ and~$Y= (X,V,\Theta,\Omega)$, and consider a collection
 of decreasing times, impact parameters and velocities
$$
T_{1,s} = \{ t_1 > \dots > t_s \} \, , 
\quad \cN_{1,s} =  \{ \nu_1 , \dots ,\nu_s \}  \,, 
\quad V_{1,s} = \{ v_1 , \dots , v_s \} \,.
$$ 
Pseudo-trajectories are defined in terms of the  backward BBGKY dynamics as follows :
\begin{itemize}
\item in between the  collision times~$t_i$ and~$t_{i+1}$   the particles follow the~$i+1$-particle backward flow with elastic reflection;
\item  at time~$t_i^+$, if $a_i \neq 0$,   the atom labeled~$i$ is adjoined to atom $a_i$ at position~$x_{a_i} + \eps \nu_{i}$ and 
with velocity~$v_{i}$  provided this does not cause an overlap of particles. 

If $(v_{i} - v_{a_i} (t_i^+)) \cdot \nu_i >0$, velocities at time $t_i^-$ are given by the scattering laws
\begin{equation}
\label{V-def}
\begin{aligned}
v_{a_i}(t^-_i) &= v_{a_i}(t_i^+) - (v_{a_i}(t_i^+)-v_{i}) \cdot \nu_{i} \; \nu_{i} \, ,\\
v_{i} (t^-_i) &= v_{i} + (v_{a_i}(t_i^+)-v_{i}) \cdot \nu_{i} \;  \nu_{i} \, .
\end{aligned}
\end{equation}
\item at time~$t_i^+$, if $a_i = 0$ and  provided this does not cause an overlap of particles,  the atom labeled~$i$    is adjoined to the rigid body at position~$X+{\eps \over \alpha}R_{\Theta(t_i^+)}   r_{\alpha, i}  $ and 
with velocity~$v_{i}$.

If $(\alpha^{-1}  v_{i} - V (t_i^+)-\Omega(t_i^+)  R_{\Theta(t_i^+)}   r_{i}^\perp  ) \cdot R_{\Theta(t_i^+)}   \nu_{i} >0$, velocities at time $t_i^-$ are given by the scattering laws
\begin{equation}
\label{VOmega-def}
\begin{aligned}
&v_{i} (t^-_i)-v_{i} = {2\over A+1} \big(\alpha V(t_i^+)+\alpha  \Omega (t_i^+)   R_{\Theta(t_i^+)}     r_{i} ^\perp  - v_{i}  \big) \cdot R_{\Theta(t_i^+)}    \nu_{i}  \, R_{\Theta(t_i^+)}   \nu_{i}  \, ,\\
 &V(t_i^-) - V(t_i^+)  = -{2\alpha\over  A+1 }  (\alpha V(t_i^+)+\alpha  \Omega (t_i^+)  R_{\Theta(t_i^+)}     r_{i} ^\perp   -  v_{i})  \cdot R_{\Theta(t_i^+)}   \nu_{i} \,  R_{\Theta(t_i^+)}   \nu_{i} \, ,\\
 &\Omega(t_i^-) - \Omega(t_i^+) =-{2\alpha\over (A+1)I}  (\alpha V(t_i^+)+ \alpha \Omega (t_i^+)  R_{\Theta(t_i^+)}     r_{i} ^\perp -   v_{i}) \cdot  R_{\Theta(t_i^+)}   \nu_{i}  \, (  r_{i}^\perp \cdot   \nu_i )    \,. 
\end{aligned}
\end{equation}
\end{itemize}
At each time $\tau \in [0,t]$, we denote  by $Y(a, T_{1,s},\cN_{1,s}, V_{1,s},\tau)$ the position, velocity, orientation and angular velocity of the rigid body and by~$z_i(a, T_{1,s},\cN_{1,s}, V_{1,s},\tau)$ the position and velocity of the atom labeled~$i$ (provided~$\tau< t_i$). The  configuration obtained at the end of the tree, i.e. at time 0, is~$ \big(Y, Z_s\big)(a, T_{1,s}, \cN_{1,s}, V_{1,s},0)$. 
 The term  $Q_{1,1+s} (t) f_{N+1,0}^{(s+1)}$ in the series \eqref{eq: iterated-Duhamel}
 is evaluated by integrating the initial data  $f_{N+1,0}^{(s+1)}$ over the  values 
 $\big(Y, Z_s \big)(a, T_{1,s}, \cN_{1,s}, V_{1,s},0)$ of the pseudo-trajectories
  at time 0.

\medskip

Pseudo-trajectories provide a geometric representation of the iterated Duhamel series~\eqref{eq: iterated-Duhamel}, but they are not physical trajectories of the particle system.
Nevertheless,  the probability on the trajectories of the rigid body can be derived from the Duhamel series,
as we are  going to explain now.
For a given time $T >0$, the sample path of the rigid body  is denoted by~$Y_T := (X(t),V(t),\Theta(t),\Omega(t))_{t \leq T}$ and the corresponding probability by $\mathbb{P}_{f_{N+1,0}}$, where the subscript 
 stands for the initial data of the particle system.
As $Y_T$  has jumps in the velocity and angular momentum, it is   convenient to work 
 in the space $D([0,T])$ of functions that are right-continuous with left-hand limits in ${\mathbb R}^6$. 
 This space is endowed with the Skorohod topology (see \cite{billingsley} page 121).

The following proposition allows us to rephrase the probability of trajectory events in terms of the Duhamel series. \begin{Prop}
\label{prop: identification proba Duhamel}
For any measurable event $\cC$  in the space $D([0,T])$, the probability that the path~$Y_T = \{ Y(t) \}_{t \leq T}$ under the initial distribution~$ f_{N+1,0}$ belongs to $\cC$ is given by
\begin{align*}
\mathbb{P}_{f_{N+1,0}}  \Big(  \{ Y_T \in \cC \} \Big) 
= 
\int dY \sum_{s =0}^{N} 
Q_{1,1+ s} (T) \; \indc_{\{  Y_T \in \cC \}} \;  f_{N+1,0}^{(1 + s)}  \,,
\end{align*}
where the notation $Q_{1,1+s} (T) \; \indc_{\{  Y_T \in \cC \}}$ means that only the pseudo-trajectories such that 
$ Y_T$ belongs to $\cC$ are integrated over the initial data. The other pseudo-trajectories are discarded.
The integral is over  the coordinates $Y$ of the rigid body at time $T$.
\end{Prop}

\begin{proof}
In \cite{BGSR1}, the iterated Duhamel formula was adapted to control the process at different times.
Let $\tau_1 < \dots < \tau_\ell$ be an increasing collection of times and~$H_\ell = \{ h_1,\dots,h_\ell \}$ a collection of $\ell$ smooth functions.
Define the biased distribution  at time $\tau_\ell$ as follows 
\begin{equation}\label{eq: fHk 2}
\begin{aligned}
\mathbb{E}_N  \Big(  h_1 \big(Y  (\tau_1) \big) \dots  h_\ell \big( Y (\tau_\ell) \big) \Big) 
: = 
\int dY \; h_\ell ( Y ) f_{N+1, H_\ell}^{(1)} (\tau_\ell, Y) \,,
\end{aligned}
\end{equation}
where $Y = (X, \Theta, V , \Omega)$, in the integral, stands for the state of the rigid body at time $\tau_\ell$
and the modified density is 
\begin{align}
\label{eq: representation itere + poids}
f_{N+1, H_\ell}^{(1)}  (\tau_\ell,Y) := & \sum_{m_1 + \dots+ m_{\ell} =0}^{N} 
 Q _{1,1+ m_1} (\tau_\ell - \tau_{\ell-1} ) 
\Big( h_{\ell-1} Q _{1+ m_1,1+ m_1+m_2} (\tau_{\ell-1} - \tau_{\ell-2} )   \\
& \qquad \qquad 
\dots h_1 Q _{1+ m_1+\dots + m_{\ell-1} ,1+ m_1+ \dots + m_{\ell}} (\tau_1) \Big) f^{(1 + m_1+ \dots + m_{\ell})}_{N+1,0}  \,.
\nonumber
\end{align}
In other words, the collision tree is   generated backward starting from $Y= Y(\tau_\ell)$ and the iterated Duhamel formula is weighted by the product
$h_1 \big(Y  (\tau_1) \big) \dots  h_\ell \big( Y (\tau_\ell) \big)$ evaluated on the backward pseudo-trajectory associated with the rigid body.  

More generally any function $h$ in $({\mathbb T}^2\times{\mathbb S}^1 \times  \R^3 )^{\otimes \ell}$ can be approximated in terms of products of functions in $({\mathbb T}^2\times {\mathbb S}^1 \times  \R^3 )^{\otimes \ell}$, thus \eqref{eq: fHk 2} leads to 
\begin{align}
\mathbb{E}_N  \Big(  h \big(Y  (\tau_1) , \dots,  Y (\tau_\ell) \big) \Big) 
 = 
\int dY \sum_{m =0}^{N} 
Q _{1,1+ m} (T) \;  h \big(Y  (\tau_1) , \dots,  Y (\tau_\ell) \big) \;  f^{(1 + m)}_{N+1,0}  \,,
\label{eq: fHk 3}
\end{align}
where the Duhamel series are weighted by the rigid body trajectory at times $\tau_1, \dots, \tau_\ell$.

\medskip

For any $0 \leq \tau_1 < \dots < \tau_\ell \leq T$, we denote by~$\pi_{\tau_1, \dots ,\tau_\ell}$  
the projection from $D([0,T])$ to~$({\mathbb T}^2\times {\mathbb S} \times  \R^3 )^{\otimes \ell}$
\begin{equation}
\label{eq: projection}
\pi_{\tau_1, \dots ,\tau_\ell} (Y) = ( Y(\tau_1), \dots, Y(\tau_\ell)) \,.
\end{equation}
The $\sigma$-field of Borel sets for the Skorohod topology can be generated by the sets of the form~$\pi_{\tau_1, \dots ,\tau_\ell}^{-1} H$ with $H$ a subset of $({{\mathbb T}^2\times {\mathbb S} \times  \R^3 })^\ell$ (see Theorem 12.5 in \cite{billingsley}, page 134).
Thus~\eqref{eq: fHk 3} is sufficient to characterize the probability of any measurable set $\cC$. 
This completes the proof of Proposition~\ref{prop: identification proba Duhamel}.
\end{proof}

\subsection{Structure of the paper}

In order to prove that in the Boltzmann-Grad limit, the mechanical motion  of the rigid body can be reduced 
to a stochastic process, we are going to use successive  approximations of the microscopic dynamics
 by idealized models. The first step is to compare 
the microscopic dynamics of the rigid body with a Markov chain by showing that 
for $N$ large (in the Boltzmann-Grad scaling~$N\eps = 1$), the complex interaction between the rigid body and the atoms can be replaced by the interaction with an ideal gas and the deterministic correlations can be neglected. 
This first step boils down to showing that the distribution of the rigid body follows closely a linear Boltzmann equation with an error controlled in $N,\eps, \alpha$: this corresponds to Theorem~\ref{prop-approximate tensorization} whose proof is achieved in Section~\ref{endoftheproofthm1}. The linear regime   still keeps track of some dependency in $\eps$ and $\alpha$ due to the fast rotation of the rigid body and to the large amount of collisions (with small deflections). In Section~\ref{proof theorem}, we   show that this dependency averages out when $\eps$ and~$\alpha$ tend to 0, thus proving Theorem~\ref{thm: convergence density}. 
Finally  Theorem~\ref{thm: convergence process}  is proved in Section \ref{process-section}, and  requires in particular   studying correlations at different times, as well as   checking some tightness conditions.

\bigskip

\noindent
\underbar{\sl Series expansions and pseudo-trajectories}

The initial data \eqref{defdata} is a small fluctuation around the equilibrium Gibbs measure, thus we expect that the atom distribution  will remain close to equilibrium and that in the large $N$ limit the rigid body will behave as if it 
were in contact with an ideal gas.
As a consequence, for large $N$, the distribution of the rigid body should be well approximated by
$$
  f_{N+1}^{(1)} (t,Y) \sim  \bar M_{\beta, I} (Y)  g_\eps(t,Y) \, ,
$$
where $ \bar M_{\beta, I}  $ was introduced in \eqref{eq: Gauss} and $g_\eps$ solves the linear Boltzmann equation
\eqref{eq: linear boltz}  with  initial data~$g_0$.

This approximation is made quantitative in Section  \ref{endoftheproofthm1}.
This is the key step of the proof of Theorem~\ref{prop-approximate tensorization}: one has to control the dynamics of the whole gas and to prove that the atoms act as a stochastic bath on the rigid body (up to a small error). Note that, as usual in the Boltzmann-Grad limit, we are not able to prove directly the tensorization for the joint probability of atoms, and the decoupling of the equation for the rigid body distribution.

\medskip

  We   therefore approximate the BBGKY hierarchy by  another hierarchy,
 the initial data of which is given by
\begin{equation}\label{dataBoltz}
\forall s \geq 1 \, , \quad  f_0^{(s)} (Y,Z_{s-1}) := g_0(Y) \bar M_{\beta, I} (Y)   \prod_{i =1}^{s-1} M_{\beta } (v_i)  \, .
\end{equation}
This hierarchy (referred to in the following as the Boltzmann hierarchy) is obtained by taking formally the~$N\to \infty$, $\eps \to 0$  asymptotics in the collision operators appearing in the BBGKY hierarchy under the Boltzmann-Grad scaling~$N\eps = 1$. It represents the dynamics where the rigid body and the atoms are reduced to points, however its solution still depends  on~$\alpha$ and~$ \eps$, which appear in the scaling of   the angular velocity $O(\alpha/\eps)$  of the rigid body and the velocities of the atoms $O(1/\alpha)$, as well in the collision frequency and in the fact that the collision integral is on~$\partial   \Sigma_\alpha$.
We thus define 
$$
   \begin{aligned}
  \big(  \bar C_{s,s+1} f^{(s+1)}\big) (Y,Z_{s-1}) := \frac 1 \alpha  \sum_{i=1}^s \int_{{\mathbb S} \times {\mathbb R}^2} \Big[  f^{(s+1)}(\dots, x_i, v_i',\dots , x_i, v'_{s}) \big((v'_{s}-v'_i) \cdot \nu_{s}\big)_-  \\
  -f^{(s+1)}(\dots, x_i, v_i,\dots , x_i, v_{s})    \big((v_{s}-v_i) \cdot \nu_s \big)_- \Big]    d\nu_s dv_{s}   \end{aligned}
$$
and   
$$
   \begin{aligned}
\big(  \bar D_{s,s+1} f^{(s+1)}\big) (Y,Z_{s-1})
  :=  \frac 1 \alpha  \int_{[0,L_\alpha]  \times {\mathbb R}^2}  \Big[f^{(s+1)} ( Y', Z_{s-1}, X ,v'_{s})  \Big( \big ( \frac 1 \alpha v'_{s}-V'-
 \Omega'   r_\Theta^\perp
\big ) \cdot   n_\Theta\Big)_-\\
-f^{(s+1)} ( Y, Z_{s-1}, X ,v_{s}) 
 \Big( \big ( \frac 1 \alpha v_{s}-V-
 \Omega    r_\Theta^\perp
\big ) \cdot   n_\Theta\Big)_-\Big] d  \sigma_\alpha    dv_{s}   \, , 
 \end{aligned}
$$
for any $Y=(X,V,\Theta, \Omega)$ and $Z_{s-1}$.

One can check that if the initial data is given by~(\ref{dataBoltz}), then the  solution~$(f_\eps^{(s)})_{s\geq 1}$ to the Boltzmann hierarchy  is given by the  tensor product to the solution of the linear Boltzmann equation~(\ref{eq: linear boltz}), namely
$$
\forall s \geq 1 \, , \quad  f_\eps^{(s)} (t,Y,Z_{s-1}) := g_\eps (t,Y) \bar M_{\beta, I} (Y)   \prod_{i =1}^{s-1} M_{\beta } (v_i)  \, ,
$$ 
where~$g_\eps$ solves~(\ref{eq: linear boltz}).

\medskip

The core of the proof of Theorem~\ref{prop-approximate tensorization} therefore consists in controlling  the difference between both hierarchies, starting from the geometric representation of solutions in terms of  pseudo-trajectories.

 Pseudo-trajectories involving a large number of collisions contribute very little to the sum as can be proved by the pruning argument developed in \cite{BGSR1} (see Section \ref{pruning}).
 On the other hand, we expect most pseudo-trajectories with a moderate number of collisions to involve only collisions between independent particles. A geometric argument similar to \cite{GSRT, BGSR1} (see Section \ref{recollisions})  gives indeed  a suitable estimate of the error, provided that locally the interaction between the rigid body and any fixed atom corresponds to a unique collision.

\bigskip\noindent
\underbar{\sl Control of the scattering}
 
 Compared to \cite{GSRT, BGSR1}, we have here an additional step to control the pathological atom-rigid body interactions leading to a different scattering. 
 
\medskip

 Because atoms are expected to have a typical velocity $\hat v = O(1/\alpha)$ while each point of the  rigid body has a typical velocity $ V+  \hat \Omega   r^\perp  = O(1)$, in most cases the atom will escape before a second collision is possible. 
  However, the   set of parameters leading to   pathological situations  can be controlled typically by a power of $\alpha$ (see  Section~\ref{conditioning-geometry}), which  is not small enough  to be neglected as the other recollisions (which are controlled by a power of~$\eps$). To avoid those pathological situations, a modified, truncated dynamics is introduced in Section~\ref{conditioning-def}, which stops as soon as such a pathological collision occurs. Section \ref{controlpathological} provides the proof that the original dynamics coincides with the truncated dynamics for data chosen outside a small set.

\bigskip\noindent
\underbar{\sl Diffusive scaling}

Following the strategy described in the previous paragraph and performed in Sections~\ref{conditioning} to~\ref{recollisions}, we obtain explicit controls in terms of~$N,\eps, \alpha$ on the convergence of the first marginal to the  linear Boltzmann equation \eqref{eq: linear boltz}.
However this equation still depends on $\eps$  and $\alpha$.   
On the one hand,    we prove in Section~\ref{singpertpb} that the density becomes rotationally invariant as~$\eps \to 0$, and on the other hand  in  
 the limit $\alpha \to 0$, we   show in Section \ref{subsec: Taylor expansion} that expanding 
the density in the collision operators, cancellations occur at   first order between the gain and loss terms.
Thus in the joint limit $\eps, \alpha \to 0$, we  prove that the linear Boltzmann equation~\eqref{eq: linear boltz} remains close (in a weak sense) to the weak solution of 
$$\begin{aligned}
\d_t g +V\cdot \nabla_x g = \left( \frac8{\pi \beta}\right)^{1/2}   \cL g \, ,
\end{aligned}
$$
where the diffusion operator is given by  \eqref{eq: diffusion operator}.
Finally, these estimates are used in Section~\ref{process-section} to prove the convergence 
towards an Ornstein-Uhlenbeck process.

\section{The modified BBGKY hierarchy}
\label{conditioning}

\subsection{Geometry of  the atom-rigid body interaction}
\label{conditioning-geometry}

As the rigid body rotates, even though it is convex, there are   situations where the binary  interaction between an atom and the rigid body leads to many collisions. One can imagine for instance   the extreme case when the rigid body is very long so that it almost separates the plane into two half planes  (with a rotating interface) : then, whatever the motion of the atom is, we expect    infinitely many collisions to occur.

Below, we focus on   collisions between the molecule and a single atom, forgetting the rest of the gas for a moment. Furthermore,  we consider the dynamics in the whole space and do not take into account periodic recollisions. Rescaling time and space by a factor $\alpha/\eps$,
we first use the scaling invariances of the system to reduce to the case when 
\begin{itemize}
\item the rigid body has size $O(1)$ and the diameter of the atom is $\alpha$;
\item the velocity and angular velocity of the rigid body are  $O(1)$;
\item the typical velocity of the atom is $O(1/\alpha)$.
\end{itemize}

Then we shall take advantage of the scale separation between the velocity of the atoms and the velocity of the rigid body, to show that with high probability the atom will escape a security ball around the rigid body before the rigid body has really moved. More precisely, we shall prove that if at the time of   collision, one has the following conditions (where~$\eta>0 $ is chosen small enough, typically~$\eta<1/6$ will do)
\begin{equation}
\label{good-collision1}
 \left| V' - V \right|  \geq \alpha ^{2+\eta }   \,  , \qquad 
 \max \big\{ |V|, |V'|, |\Omega|, |\Omega' | \}   \leq |\log \alpha| \, ,
\end{equation}
and
\begin{equation}
\label{good-collision2}
| v -\alpha V | \geq  \alpha ^{2/3 +\eta }
\quad \text{and} \quad
| v' -\alpha V' | \geq  \alpha ^{2/3 +\eta } \,  ,
\end{equation}
 there cannot be any direct  recollision between the rigid body and the atom.
 Notice that under~(\ref{good-collision1}), both conditions in~(\ref{good-collision2}) can be deduced one from the other thanks to~(\ref{eq: collision laws}) up to a (harmless) multiplicative constant. 
 \begin{Prop}\label{norecollision} Fix~$\eta < 1/6$ and consider a collisional configuration between an atom and the rigid body. Under assumptions~{\rm(\ref{good-collision1}), (\ref{good-collision2})}  on the collisional velocities, the atom cannot recollide with the rigid body.
\end{Prop}
 
 \begin{proof}
From the scattering law (\ref{eq: collision laws}), we   know that
\begin{equation}
\label{eq: scattering big}
\left| V' - V \right| = \frac{2\alpha^2}{A+1} \Big|( {v\over \alpha } - V-\Omega r_\Theta^\perp  ) \cdot n_\Theta\Big| =\frac{2\alpha^2}{A+1} ( {v'\over \alpha } - V'-\Omega' r_\Theta^\perp  ) \cdot n_\Theta \, ,
\end{equation}
  where the last term is nonnegative as the velocities are outgoing.
  Under     assumption~(\ref{good-collision1}) and recalling that~$A\sim \alpha^2$, we therefore have that the  normal relative velocity at the contact point  is bounded from below by
$O( \alpha^{\eta})$:
\begin{equation}
\label{velocitycontactpoint}
( {v'\over \alpha } - V'-\Omega' r_\Theta^\perp  ) \cdot n_\Theta  \geq C \alpha^{\eta}\,.
\end{equation}
 As this velocity is relatively  small, in order to prove that 
the atom   escapes without any recollision, we will use the two key geometrical properties 
(\ref{defrmax}), (\ref{defkappa}) on the rigid body:
\begin{itemize}
\item it  has finite size, it is included in a ball of radius $r_{max}$;
\item it is strictly convex, with a curvature  uniformly bounded from below by~$\kappa_{min}$.
\end{itemize}
Note that we   only deal with kinematic conditions, so that we can stay in the reference frame of the rigid body, and split the dynamics into two components : the rotation of the rigid body and the translation of the atom.

 - If there is a recollision, it should occur before time 
\begin{equation}
\label{tmax}
\delta _{max}: = {2 r_{max} \over  \min |{v'\over \alpha} - V'|}
\end{equation}
 at which  the atom escapes from the range of action of the rigid body. From Assumption~(\ref{good-collision2}), we deduce that
 $$ \delta _{max} \leq C \alpha ^{1/3 - \eta}\,.$$
 In particular, on this time scale, the angle of rotation of the rigid body is very small thanks to~(\ref{good-collision1})
\begin{equation}\label{smallanglerotation}
 \delta_{max} |\Omega'| \leq  C \alpha ^{1/3 - \eta}|\log \alpha| \,.
 \end{equation}

 \begin{figure} [h] 
\centering
\includegraphics[width=10cm]{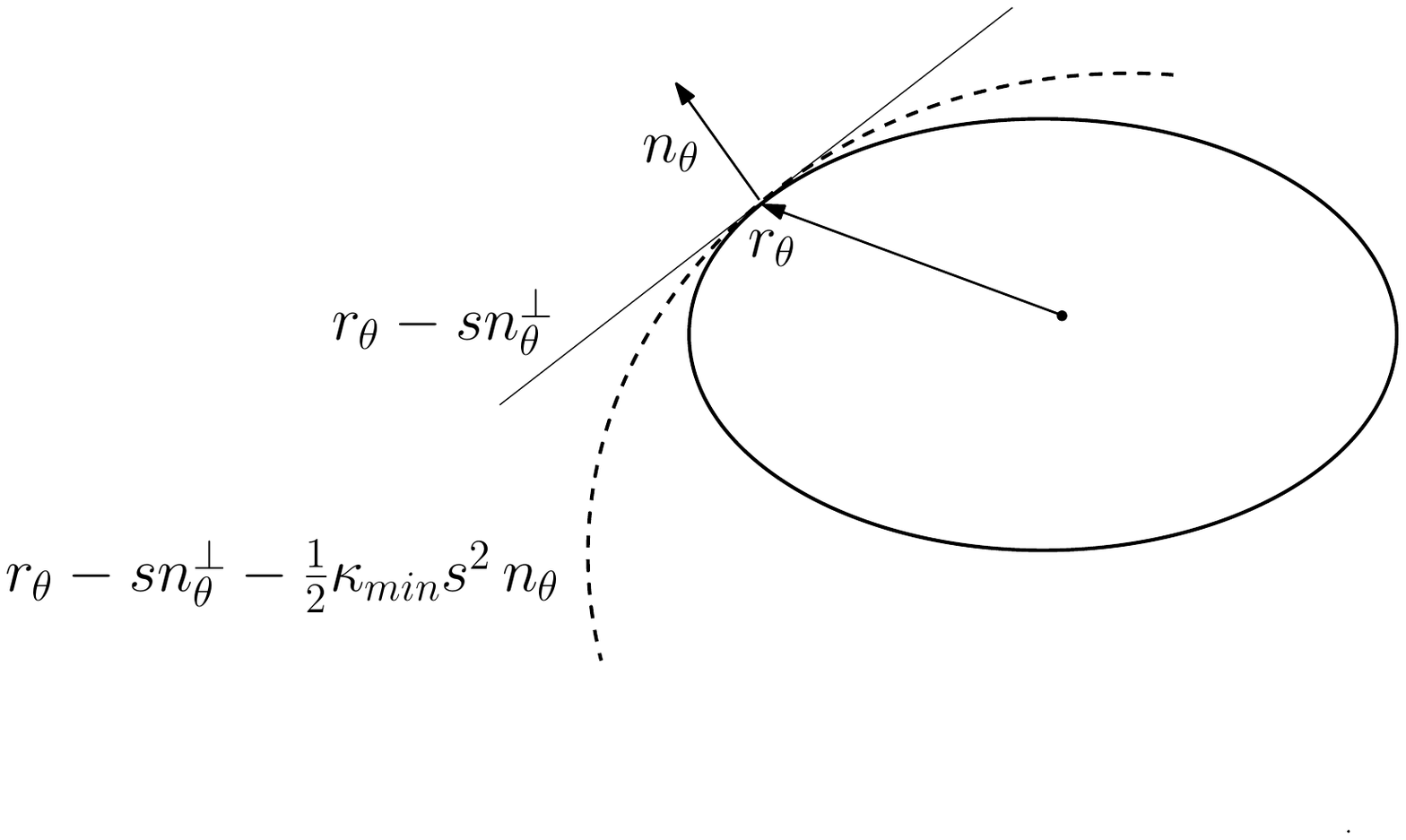}
\caption{The collision between the rigid body and a particle occurs at $r_\Theta$ and the dotted line is the tangent in the direction $-n_\Theta^\perp$ at $r_\Theta$. The parabola is represented in dashed line.}
\label{fig: rigid body 2}
\end{figure}

Let us look at the motion in the reference frame associated with  the center of mass $G$ of the rigid body.

 -  Because of the strict convexity of the rigid body, it is in the subset of the plane delimited by the parabola parametrized by $s\in \R$ (see Figure \ref{fig: rigid body 2})
 $$
   r_\Theta - s  n_\Theta^\perp -\frac12 \kappa_{min}  s^2 n_\Theta\,,
 $$
 where we have chosen the origin at the center of mass  of the rigid body, and we denote by~$r_\Theta$ the contact point at first collision and by $(-n_\Theta^\perp, n_\Theta)$ the tangent and the normal at $R_\Theta\Sigma$ at this point.
 Note that, because the rigid body is contained in a ball of radius $r_{max}$, we are only interested in the portion of the curve with $s \leq C r_{max}$.
 After a small time~$\delta$, the rigid body has rotated by a small  angle $\delta \Omega'$ around the origin, so this curve is parametrized by 
 $$ 
  (r_\Theta - s  n_\Theta^\perp  -\frac12 \kappa_{min} s^2 n_\Theta) + \delta \Omega' (r_\Theta - s   n_\Theta^\perp -\frac12 \kappa_{min} s^2 n_\Theta) ^\perp + O( (\delta \Omega')^2) \,.
  $$ 
 
 In order for the atom to recollide  with the rigid body, it has first to intersect the parabola at some time $\delta \leq \delta_{max}$, which leads to the following equation
$$ 
\delta ( {v'\over \alpha } - V')  -\delta \Omega' r_\Theta^\perp + s n_\Theta^\perp +\frac12 \kappa_{min} s^2 n_\Theta  - s \delta \Omega' n_\Theta + \frac12 \delta \Omega' \kappa_{min} s^2 n_\Theta^\perp + w = O((\delta \Omega')^2)
$$
 denoting by $w$ the relative position of any given   point of the atom with respect to the contact point at the time of first collision. Note that
  $$
  w\cdot n_\Theta \geq 0\,.
  $$
 
 Taking the scalar product by $n_\Theta$, we get
 $$ 
 \delta ( {v'\over \alpha } - V' -\Omega' r_\Theta^\perp) \cdot n_\Theta +\frac12 \kappa_{min} s^2 - s \delta \Omega' + w \cdot n_\Theta = O( (\delta \Omega')^2)\,.$$
 The canonical form of the polynomial in the left hand side is
 $$\begin{aligned}
\frac12  \kappa_{min}  ( s -  {\delta \Omega'\over \kappa_{min}})^2 &+ \delta( {v'\over \alpha } - V' -\Omega' r_\Theta^\perp) \cdot n_\Theta  - \left(  {\delta \Omega'\over \kappa_{min}} \right) ^2 \\
 &  \geq \delta\Big(( {v'\over \alpha } - V' -\Omega' r_\Theta^\perp) \cdot n_\Theta  - \delta \left(  { \Omega'\over 2\kappa_{min}} \right) ^2\Big) \\
 &  \geq C \delta  (\alpha^{\eta} + O (\delta (\Omega')^2)),
   \end{aligned}
   $$
where on the last line, the first lower bound comes from~(\ref{velocitycontactpoint}).   

Recall that by~(\ref{good-collision2}) and~(\ref{smallanglerotation}), 
$$\delta (\Omega')^2 \leq C \alpha^{1/3-\eta } |\log \alpha|^2 \,.$$
We thus conclude, since~$\eta<1/6$, that
$$\delta ( {v'\over \alpha } - V' -\Omega' r_\Theta^\perp) \cdot n_\Theta +\frac12 \kappa_{min} s^2 - s \delta \Omega' + w \cdot n_\Theta + O( (\delta \Omega')^2)>0\,,$$
which implies that no recollision can occur.
\end{proof}

\subsection{Modified dynamics}
\label{conditioning-def}

Initial data such that the microscopic dynamics do not satisfy  conditions 
\eqref{good-collision1}, \eqref{good-collision2} at some time in a given time interval $[0,T]$ will lead to pathological trajectories which cannot be easily controlled in 
terms of the Duhamel series \eqref{eq: iterated-Duhamel}. 
In this section, we   prove that  such initial data  contribute very little to the average, and therefore that  the dynamics can be modified on these bad configurations without changing the law of large numbers.

\medskip

Fix $T$ a given time. The set $\cA^{\eps,\alpha} $ of initial data such that the rigid body encounters a pathological collision during  $[0,T]$, i.e. a collision for which \eqref{good-collision1} or \eqref{good-collision2} is not satisfied,
is  included in   the union of the  following two sets, the first one being defined only in terms of the trajectory of the rigid body, and the second one involving additionally  one atom of small relative velocity~:
$$
\begin{aligned}
\cA_1^{\eps,\alpha} &:= \Big\{ (Y,Z_N) \; \Big| \;  \exists s  \leq T, 
\quad  0< \left| V(s^-) - V(s^+) \right|  <  \alpha ^{2+\eta }\\
& \qquad \qquad \qquad \qquad \qquad \text{or} \quad
 |\Omega (s)|    \geq |\log \alpha| \quad \text{or} \quad \  |V (s)|    \geq |\log \alpha| 
\Big\}\,,\\
\cA_2 ^{\eps,\alpha}&:= \Big\{ (Y,Z_N) \; \Big| \;  \exists s \leq T, i \leq N, \quad 
d\big(  x_i (s), X (s)+ \frac \eps\alpha R_\Theta    \Sigma\big)  = \frac\eps{2} 
\quad \text{and}  \\
& \qquad \qquad \qquad \text{either} \quad 
 |v_i(s^+)  - \alpha V(s^+) | \leq \alpha^{2/3 + \eta}
 \quad \text{or} \quad 
 |v_i(s^-)  - \alpha V(s^-) | \leq \alpha^{2/3 + \eta}
\Big\} \,.
\end{aligned}
$$
 
We shall prove that both sets have vanishing probability under the invariant measure when~$\alpha$ tends to 0, uniformly in $\eps\ll \alpha$.   
In the following,  $\mathbb{E}_{M_{\beta,I, N}}, \mathbb{P}_{M_{\beta,I, N}}$ stand for  the expectation and the probability of the microscopic dynamics on the time interval~$[0,T]$ starting from $M_{\beta,I, N}$. In the following we shall use indifferently the notation~$\mathbb{E}_{M_{\beta,I, N}}, \mathbb{P}_{M_{\beta,I, N}}$ when applied to the trajectory or to the corresponding initial data.
\begin{Prop}
\label{lem: small set}
Assume that  the   scaling relations $N \eps =1$ and $\alpha |\log \eps|\geq 1$ hold.
Then, for any given $T>0$, we have
$$
\lim_{ \alpha \to 0 }  \mathbb{P}_{M_{\beta,I, N}} (\cA^{\eps,\alpha}) = 0\, .
$$
\end{Prop}
The proof of this  result is postponed to Section~\ref{controlpathological}.
We stress the fact that the pathological trajectories are not   estimated in terms of the Duhamel series but 
directly at the level of the microscopic dynamics.

%

\bigskip
\noindent
\underbar{\sl  A non conservative dynamics}

To avoid multiple collisions of an atom with the rigid body, we  kill trajectories when one of the conditions (\ref{good-collision1}), (\ref{good-collision2}) is violated.
More precisely, we define  $\tilde f_{N+1}$ as the  solution  to the Liouville equation
$$
\partial_t  \tilde f_{N+1} + V\cdot \nabla_X \tilde  f_{N+1} +\frac1 \alpha \sum_{i=1}^N v_i \cdot \nabla_{x_i}  \tilde f_{N+1} +\frac \alpha\eps \Omega  \,  \partial_{\Theta}  \tilde f_{N+1} = 0   \, ,
$$
with the following modification of the boundary conditions
\begin{equation}
\label{eq: killed bc} 
\tilde f_{N+1} (t,Y', Z_N')  = \tilde f_{N+1} (t,Y, Z_N) 
 \indc_{ \big\{ (v_i, V,\Omega) \text{  satisfy  \eqref{good-collision1} and \eqref{good-collision2}} \big\}}  
\end{equation}
where  $(Y', Z_N')$ is the post-collisional configuration  defined by (\ref{eq: collision laws}).
Note  that $\tilde f_{N+1}$ coincides with $f_{N+1}$ on all characteristics which do not involve a pathological collision between the rigid body and an atom.
 The marginals of~$\tilde f_{N+1} $ are
$$
\tilde   f_{N+1}^{(s)} (t,Y,Z_{s-1}):= \int  \tilde  f_{N+1} (t,Y, Z_N) \, dz_{s} \dots dz_N\, ,$$
and for all $i \in \{1,\dots, s-1\}$,  there holds    at the boundary
\begin{equation}
\label{eq: killed bcs}
\tilde f_{N+1} ^{(s)} (t,Y', Z_s')  = \tilde f_{N+1} ^{(s)} (t,Y, Z_s) 
 \indc_{ \big\{ (v_i, V,\Omega) \text{  satisfy  \eqref{good-collision1} and \eqref{good-collision2}} \big\}}  \,,
\end{equation}
and  in particular
$$
\tilde f_{N+1} ^{(s)} (t,Y', Z_s')  
\indc_{ \big\{ (v_i, V,\Omega) \text{ do not satisfy  \eqref{good-collision1} or  \eqref{good-collision2}} \big\} }  \equiv 0\,.$$
From the maximum principle, we furthemore obtain that
\begin{equation} \label{estimatemaxprinciple}
 0 \leq \tilde f_{N+1} \leq f_{N+1} \leq   \| g_0\|_\infty M_{\beta, I,N} \,.
 \end{equation}

\bigskip
\noindent
\underbar{\sl  Modified  hierarchy}

The modified BBGKY hierarchy can be written for~$t \leq T$
\begin{equation}
\label{eq: modified hierarchy}
\begin{aligned}
\tilde f_{N+1}^{(s)}  (t) 
=
\sum_{n=0}^{N+1-s}    \int_0^t \int_0^{t_{s}}\dots  \int_0^{t_{s+n-2}}  {\bf S}^\dagger_s(t-t_{s}) ( C_{s,s+1} +  D^\dagger_{s,s+1}) {\bf S}^\dagger_{s+1}(t_{s}-t_{s+1}) \dots  \\
\dots  {\bf S}^\dagger_{s+n}(t_{s+n-1})  f_{N+1}^{(s+n)}(0) \: dt_{s+n-1} \dots dt_{s} \, ,
\end{aligned}
\end{equation}
 where 
 \begin{itemize}
 \item ${\bf S}^\dagger_s$ denotes the  semigroup associated with~$\displaystyle V \cdot \nabla_X + \frac1\alpha\sum_{i=1}^{s-1} v_i \cdot \nabla_{x_i} + \frac \alpha \eps \Omega \, \partial_\Theta$ in $\cD_\eps^{s}$ with partial specular reflection (\ref{eq: killed bcs});	
\item   and the truncated collision operator $D_{s,s+1}^\dagger
= D_{s,s+1}^{+, \dagger} - D_{s,s+1}^-$ is obtained by modifying the gain operator  \eqref{eq: operator D}  so that
\begin{equation}
\label{eq: operator D truncated}
\begin{aligned}
\big(  D_{s,s+1}^{+, \dagger}   \tilde f_{N+1}^{(s+1)}\big) (Y,Z_{s-1})    :=   (N-s+1)  \frac \eps \alpha  
\int_{[0,L_\alpha] \times {\mathbb R}^2} 
 \indc_{ \big\{ (v'_s, V',\Omega') \text{ \ satisfy  \eqref{good-collision1}, \eqref{good-collision2}} \big\} }\\
{}\times \tilde f_{N+1}^{(s+1)} ( Y', Z_{s-1}, X   + \frac \eps \alpha \,     r_{\alpha,\Theta},v'_{s}) \Big( \big (  \frac 1 \alpha v'_{s}-V'-
 \Omega'  r_\Theta^\perp  \big ) \cdot      n_\Theta\Big)_-  \!  d  \sigma_\alpha dv_{s}\, .
\end{aligned}
\end{equation}
\end{itemize} Notice that
\begin{equation}
\label{estimateonSsdagger}
\forall p \in [1,\infty] \, , \quad \forall 1\leq s \leq N+1\, , \quad
\|{\bf S}^\dagger_s f_s \|_{L^p} \leq \|f_s \|_{L^p}\, .
\end{equation} 

\bigskip

\noindent
\underbar{\sl  Error estimates}
 
 Both dynamics ${\bf S}^\dagger_{N+1}$ and ${\bf S}_{N+1}$ coincide on $[0,T]$ for initial data which are supported on~$(\cA^{\eps, \alpha}) ^c$ as no pathological collision occurs.  Define
 $$ r_{N+1,0} := f_{N+1, 0}  \indc _{\cA^{\eps, \alpha} }\,.$$
Then, for any $t\in [0,T]$, thanks to~(\ref{estimateonSsdagger}),
$$ \begin{aligned}
\| f_{N+1}(t)   - \tilde f_{N+1} (t) \|_{L^1(\cD_\eps^{N +1})} &= \| {\bf S}^\dagger_{N+1} (t)   f_{N+1, 0}  -{\bf S}_{N+1} (t)   f_{N+1, 0} \|_{L^1(\cD_\eps^{N +1 })}\\
& \leq\big \| \big({\bf S}^\dagger_{N+1} (t)    -{\bf S}_{N+1} (t) \big)   f_{N+1, 0}(1- \indc _{\cA^{\eps, \alpha}}) \big \|_{L^1(\cD_\eps^{N +1 })}\\
& \quad + \| {\bf S}^\dagger_{N+1} (t)     r_{N+1, 0} \|_{L^1(\cD_\eps^{N +1})}+ \| {\bf S}_{N+1} (t)   r_{N+1, 0} \|_{L^1(\cD_\eps^{N +1})}\\
& \leq 2 \|  r_{N+1,0}\|_{L^1(\cD_\eps^{N +1})}\,.
\end{aligned}
$$
From  Proposition \ref{lem: small set}, we therefore deduce the following corollary.
\begin{cor}
\label{corconditioneddata}
Assume that the   scaling relations $N \eps =1$ and $\alpha |\log \eps|\geq 1$ hold.
Then, for any given $T>0$ we have
$$
\sup_{t \in [0,T]}  \, \lim_{ \alpha \to 0 } \| f_{N+1}(t)   - \tilde f_{N+1} (t) \|_{L^1(\cD_\eps^{N +1})}= 0\, .
$$
\end{cor}
  We shall therefore from now on work on    the modified series expansion
\begin{equation}
\label{iterated-Duhamel}
   \tilde   f_{N+1}^{(s)} (t) =  \sum_{n=0}^{N+1-s}  Q^\dagger_{s,s+n} (t)  f_{N+1}^{(s+n)}(0) \,,
\end{equation}
where we denote
\begin{equation}
\label{Qdagger}
\begin{aligned}
Q^\dagger_{s,s+n} (t) :=  \int_0^t \int_0^{t_{s}}\dots  \int_0^{t_{s+n-2}}  {\bf S}^\dagger_s(t-t_{s}) ( C_{s,s+1} +  D^\dagger_{s,s+1}) {\bf S}^\dagger_{s+1}(t_{s}-t_{s+1}) \dots  \\
\dots  {\bf S}^\dagger_{s+n}(t_{s+n-1}) \: dt_{s+n-1} \dots dt_{s} \,.
\end{aligned}
\end{equation}
Note that  the operators~$Q^\dagger_{s,s+n}$  satisfy the   estimates  stated in Proposition~\ref{estimatelemmacontinuity}. In particular the series expansion~(\ref{iterated-Duhamel})
converges on any time~$t \ll \alpha^2$. 

\subsection{Control on the pathological configurations: proof of Proposition \ref{lem: small set}}\label{controlpathological}
\label{Control on the pathological configurations}

 The probability of $\cA_1^{\eps,\alpha}$ and $\cA_2^{\eps,\alpha}$  are estimated by different arguments.

\medskip

\noindent
{\bf Step 1.} {\it Estimating the probability of $\cA_1^{\eps,\alpha}$.}

We   first reduce  the analysis to small time intervals. 
Let $  \cA_{\alpha^5}$ be the event $\cA_1^{\eps,\alpha}$ restricted to the time interval  $[0, \alpha^5]$.
Suppose that  
\begin{equation}
\lim_{ N \to \infty \atop \alpha \to 0} \frac{1}{\alpha^5}  \mathbb{P}_{M_{\beta,I, N}} (   \cA_{\alpha^5}) = 0 \, .
\label{eq: small jumps alpha}
\end{equation}
Then the limit
$$ 
\lim_{ N \to \infty} \mathbb{P}_{M_{\beta,I, N}}  (\cA_1^{\eps,\alpha}) = 0
$$
 can be deduced by decomposing the event $\cA_1^{\eps,\alpha}$ over the time intervals
$([(k-1) \alpha^5, k\alpha^5[)_{k \leq \frac{T}{\alpha^5}}$ and using the fact that $M_{\beta, I, N}$ is invariant
$$
 \mathbb{P}_{M_{\beta,I, N}}  (\cA_1^{\eps,\alpha}) \leq \frac{T}{\alpha^5} {\mathbb P}_{M_{\beta, I, N}} (   \cA_{\alpha^5}) \, .
$$

\medskip

We turn now to the proof of \eqref{eq: small jumps alpha}.
The set $  \cA_{\alpha^5}$ is a set of initial conditions for the whole dynamics, but it can be seen also as a set  $   \cC_{\alpha^5}$ on  the single trajectory $Y_{\alpha^5} = \{ Y(t) \}_{t \leq \alpha^5}$ of the rigid body. In the latter formulation,  $  \cC_{\alpha^5}$ 
is measurable in the Skorohod space $D([0,\alpha^5])$. 
It is indeed
 the union of 
$$
\cC_1 := \bigcup_{k \geq 1} \bigcap_{n \geq k}  \bigcup_{r \in {\mathbb Q} \cap [0,\alpha^5 - \frac{1}{n}]} 
\Big \{Y_{\alpha^5} \, ; \qquad   0 < \big|V(r+\frac{1}{n}) - V(r) \big| < \alpha^{2 +\eta}   \Big\}
$$
and 
$$
 \cC_2 :=   \bigcup_{r \in {\mathbb Q} \cap [0,\alpha^5]} 
\Big \{Y_{\alpha^5}\, ; \qquad    
 |V(r) | \geq |\log \alpha| \quad \text{or} \quad   \big|\Omega(r) \big|
\geq | \log \alpha|\Big\} \, .
$$
Using Proposition  \ref{prop: identification proba Duhamel}, the probability of the event 
$   \cC_{\alpha^5}$ can be rephrased in terms of  the Duhamel formula
\begin{align*}
 \mathbb{P}_{M_{\beta,I, N}}  \Big(   \cA_{\alpha^5} \Big) 
= \sum_{m =0}^{N-1} 
\int dY \;
Q_{1,1+ m} (\alpha^5) \; \indc_{\{  Y_{\alpha^5} \in     \cC_{\alpha^5} \}} \;  M_{\beta, I, N}^{(1 + m)}  \, .
\end{align*}
The contribution of $\cC_2$ is exponentially small: there are two constants~$C$ and~$C'$, depending on~$\beta$ and~$I$ and which may change from one line to the other, such that
$$
\begin{aligned}
\sum_{m =0}^{N-1} 
\int dY \;
Q_{1,1+ m} (\alpha^5) \; \indc_{\{  Y_{\alpha^5} \in   \cC_2 \}} \;  M_{\beta, I, N}^{(1 + m)}  &\leq C \exp (-C' |\log \alpha|^2 )\sum_{m =0}^{N-1} 
\int dY \;
Q_{1,1+ m} (\alpha^5) \;  M_{\beta/2, I, N}^{(1 + m)}
\\
& \leq C \exp (-C' |\log \alpha|^2 ) \, ,
\end{aligned}
$$
where the last inequality is obtained using the conservation of energy to replace~$\Omega$ by~$\Omega'$ in a fraction of the Maxwellian, and  applying Proposition \ref{estimatelemmacontinuity} with $t = \alpha^5$. 

To evaluate the contribution of $\cC_1$, we treat the terms of the series differently according  to the number of collisions.
If there is no collision $m =0$ then the rigid body is not deflected so   the term is equal to 0.
 Going back to the proof of Proposition \ref{estimatelemmacontinuity} with $t = \alpha^5$, we get
$$
\forall m \geq 2, \qquad 
\Big \| Q_{1,1+ m} (\alpha^5) \; \;  M_{\beta, I, N}^{(1 + m)}  
\Big \|_{L^\infty} 
\leq C^m \alpha^{3 m} \, ,
$$
 and we deduce that the total contribution of the terms $m \geq 2$  will be bounded by $\alpha^6$, so   it will vanish in \eqref{eq: small jumps alpha}.

Thus it remains to control the term with a single collision $m =1$: let us show that 
\begin{equation}
\label{eq: single collision cost}
\int dY Q_{1,2} (\alpha^5) \; \indc_{\{  Y_{\alpha^5} \in \cC_1\}} \;  M_{\beta, I, N}^{(2)} 
\leq C \alpha^{5 +2 \eta} \, ,
\end{equation}
which will complete \eqref{eq: small jumps alpha}.
The time integration in $Q_{1,2} (\alpha^5)$ provides a factor $\alpha^5$, so   it is enough to gain a factor 
$\alpha^{2\eta}$ from the collision operator.
The event $\cC_1$ is supported by the pseudo-trajectories with a deflection, thus only the part 
$D_{1,2}^+$ of the collision operator~\eqref{eq: collision operator rigid body} will be contributing
\begin{align*}
&\big(  D_{1,2}^+   M_{\beta, I, N}^{(2)}\big) (Y,Z_{1})  =     \\
&  \qquad \frac{1}{\alpha} \int_{[0,L_\alpha] \times {\mathbb R}^2}  \!  \!  
M_{\beta, I, N}^{(2)} ( Y',   X  + \frac \eps \alpha \,     r_{\alpha, \Theta},v'_1) \Big( \big (  \frac 1 \alpha v'_1-V'-
 \Omega'     r_\Theta^\perp  \big ) \cdot   n_\Theta\Big)_-  \!  d  \sigma_\alpha dv_1 \, .  
\end{align*}

Suppose that the collision leads to a small deflection, then 
according to \eqref{eq: scattering big} we  get 
$$ 
\left| V' - V \right| 
= {2\alpha\over 1+A} |( v_1   - \alpha V- \alpha \Omega r_\Theta^\perp  ) \cdot n_\Theta| 
= {2\alpha\over 1+A} |( v'_1  - \alpha V'- \alpha \Omega' r_\Theta^\perp  ) \cdot n_\Theta|
\leq  2\alpha^{2 + \eta}  \, .
$$
Given the coordinates of the rigid body $(V, \Omega)$ and the impact parameter $n_\Theta$, this implies that the velocity $v_1$ has to belong to a tube of diameter $\alpha^{1 + \eta}$ which has a measure less than~$\alpha^{2(1 + \eta)}$ under   $|( v_1   - \alpha V- \alpha \Omega r_\Theta^\perp  ) \cdot n_\Theta| M_\beta$.
Plugging this estimate in the collision operator $D_{1,2}^+$, we get an upper bound of the type 
$O( \alpha^{2 \eta})$. This completes \eqref{eq: single collision cost}.

\bigskip

\noindent
{\bf Step 2.} {\it Estimating the probability of $\cA_2^{\eps,\alpha} \setminus \cA_1^{\eps,\alpha} $.}

We first recall that  since~$|\Omega(s) | \leq |\log \alpha|$ and due to conditions~(\ref{eq: collision laws}),
outside~$ \cA_1^{\eps,\alpha} $ the condition 
$$
|v_i(s^-)  - \alpha V(s^-) | \leq C\alpha^{2/3 + \eta}
$$
is equivalent (up to a change of the constant~$C$), to
$$
|v_i(s^+)  - \alpha V(s^+) | \leq C\alpha^{2/3 + \eta}\,.
$$
{\color{black}
If a series of collisions occurs between a particle and the rigid body, then  thanks to Proposition~\ref{norecollision}
necessarily the first of these collisions is pathological. We are going to estimate the probability of this first pathological collision to estimate the probability of 
the event~$\cA_2^{\eps,\alpha}  \setminus \cA_1^{\eps,\alpha}$.}
It  can occur due to two possible scenarios: either there is only one particle in the vicinity of the solid body before the pathological shock, or there are several. In the following we denote by~$i$ the particle having a pathological collision with the solid body, by $\tau_c$ the time of this collision and we define
$$\tau_1:= \min \{\tau \geq 0 \,/\, \forall \tau' \in [\tau ,\tau_c]\, ,\quad  | x_i (\tau' ) -  X(\tau') |\leq 2\eps/\alpha \}\,,$$
$$\tau_2:= \min \{\tau \geq 0 \,/\, \forall \tau' \in [\tau ,\tau_c]\, ,\quad  | x_i (\tau' ) -  X(\tau') |\leq 3\eps/\alpha \}\,.$$

\begin{itemize}
\item[(i)]   If $\tau_2 = 0$, then the corresponding probability 
under $M_{\beta, I, N}$ can be estimated by 
$$
CN \frac{\eps^2}{\alpha^2}   =  C \frac{\eps }{\alpha^2} \, ,
$$
where the factor $N$ takes into account all the possible choices of the label $i$.

\item[(ii)]  If $\tau_2 >0$ and there is no other particle at a distance less than $2\eps/\alpha $ from the rigid body on $[\tau_1,\tau_c]$,
 then the particle $i$  has traveled a distance  at least ${\eps}/{\alpha}$ through the neighborhood of the rigid body
before the pathological collision.  In this case, the following event will be satisfied
$$
\cB_1 := \Big\{  \exists i \leq N\, , \quad \int_0^T ds\frac{|v_i(s) - \alpha V(s)| }{\alpha}  \indc_{\{ |x_i(s) - X(s) | \leq \frac{2 \eps}{\alpha}, 
\  |v_i(s) - \alpha V(s)| \leq \alpha^{2/3 + \eta}\}}\geq  \frac{\eps}{\alpha}
\Big\}  \,.
$$

\item[(iii)] If $\tau_2 >0$ and there is at least one other particle $j$ at a distance less than $2\eps/\alpha $ from the rigid body for some  $\tau \in [\tau_1,\tau_c]$,
 then one of the particles $i,j$ has to travel a distance at least ${\eps}/{\alpha}$ while the two particles remain 
at distance less than ${3 \eps}/{\alpha}$ of the rigid body.  Thus it is enough to estimate the event
\begin{equation}
\label{eq: slow shock case 2}
\begin{aligned}
\cB_2 := \Big\{  \exists i,j \leq N\, , \quad &
\int_0^T ds   \left( \frac{|v_i (s) - \alpha V(s) |}{\alpha} + \frac{|v_j (s) - \alpha V(s) |}{\alpha} \right)  \\ 
&\qquad \qquad 
\times \indc_{\{ |x_i(s) - X(s) | \leq \frac{3 \eps}{\alpha},  | x_j(s) - X(s)  | \leq \frac{3 \eps}{\alpha} \}}  
\geq  \frac{\eps}{\alpha}
\Big\} \, . 
\end{aligned}
\end{equation}
\end{itemize}
We turn now to estimating the probabilities of the events $ \cB_1$, $\cB_2$ to conclude the proof.

\bigskip

We first bound from above the probability of $\cB_1$ by using the invariant measure. On the one hand
\begin{align*}
{\mathbb P}_{M_{\beta,I, N}} (\cB_1) & 
\leq N \frac{1}{\eps}  
{\mathbb E}_N \left( 
\int_0^T ds  |v_1 (s) - \alpha V(s)  | \, \indc_{\{ |x_1(s) - X(s) | \leq \frac{2 \eps}{\alpha}, \  |v_1(s) - \alpha V(s)| \leq \alpha^{2/3+\eta}\}}
\right)\\
& \leq N \frac{ T}{\eps}  
{\mathbb E}_{M_{\beta,I,N}}  \left( 
  |v_1 -  \alpha V | \, \indc_{\{ |x_1 - X  | \leq \frac{2 \eps}{\alpha}, \  |v_1 - \alpha V| \leq \alpha^{2 /3 + \eta}\}}
\right) \, .
\end{align*}
As the position and velocity are independent under the invariant measure, we get 
\begin{align*}
{\mathbb P}_{M_{\beta,I, N}} (\cB_1) &
\leq N \frac{ T}{ \eps}  
{\mathbb E}_{M_{\beta,I,N}}   \left(  |x_1 - X | \leq \frac{2 \eps}{\alpha} \right) \; 
{\mathbb E}_{M_{\beta,I,N}}  \left(  |v_1 - \alpha V| \;  \indc_{ \{ |v_1 - \alpha V| \leq \alpha^{2 /3 + \eta} \}} \right)\\
& 
\leq N \frac{ T}{ \eps}  \frac{4 \eps^2}{\alpha^2} \alpha^{2 + 3 \eta}
= 4 T \alpha^{3 \eta} \, ,
\end{align*}
where we used that in dimension 2
$$
{\mathbb E}_{M_{\beta,I,N}}  \left(  |v_1 - \alpha V| \;  \indc_{ \{ |v_1 - \alpha V| \leq \alpha^{2 /3 + \eta} \}} \right)
\simeq \int_{\R^2}  M_\beta (v) |v| \indc_{ \{ |v| \leq \alpha^{2 /3 + \eta} \}} dv\simeq \alpha^{2 + 3 \eta} \,.
$$

\medskip

We turn now to the second event. We have
$$
\begin{aligned}
{\mathbb P}_{M_{\beta,I, N}} (\cB_2) & 
\leq N^2 \frac{1}{\eps}  
{\mathbb E}_{M_{\beta,I, N}}  \Big( 
\int_0^T ds  \left( |v_1(s) - \alpha V(s) | + |v_2 (s) - \alpha V(s) |  \right) \\
&\qquad\qquad\qquad\qquad \times \indc_{\{ |x_1(s) - X(s) | \leq \frac{3 \eps}{\alpha}, \   |x_2(s) - X(s) | \leq \frac{3 \eps}{\alpha} \}}
\Big) 
\\
& 
\leq C N^2 \frac{ T}{\eps}  
{\mathbb E}_{M_{\beta,I,N}}   \Big(  |x_1 - X | \leq \frac{3 \eps}{\alpha} \hbox{ and }   |x_2 - X | \leq \frac{3 \eps}{\alpha}\Big) \; 
{\mathbb E}_{M_{\beta,I,N}}   \Big(  2   |v_1 - \alpha V|  \Big) \nonumber
\\
& 
\leq C N^2 \frac{T}{ \eps}  \frac{  \eps^4}{\alpha^4}  
= \frac{C T \eps}{\alpha^4}  \,,  \end{aligned}
$$
bounding~${\mathbb E}_{M_{\beta,I,N}}    (  2   |v_1 - \alpha V|   )$ by  a constant.
Proposition~\ref{lem: small set} is proved.
\qed

\section{The modified Boltzmann hierarchy}\label{modified-Boltz-section}

\subsection{Removing the pathological collisions}
\label{subsec: Boltzmann hierarchy}

We shall prove in the following sections  that the modified BBGKY hierarchy~(\ref{iterated-Duhamel}) behaves asymptotically as the following modified Boltzmann hierarchy:  for all~$s \geq 1$,
\begin{equation}
\label{eq: modified hierarchy}
\begin{aligned}
\tilde   f_\eps^{(s)}  (t) 
:=
\sum_{n=0}^{\infty}    \int_0^t \int_0^{t_{s}}\dots  \int_0^{t_{s+n-2}} \bar  {\bf S}_s(t-t_{s}) (\bar  C_{s,s+1} +  \bar D^\dagger_{s,s+1})\bar  {\bf S}_{s+1}(t_{s}-t_{s+1}) \dots  \\
\dots  \bar {\bf S}_{s+n}(t_{s+n-1})  f_0^{(s+n)} \: dt_{s+n-1} \dots dt_{s} \, ,
\end{aligned}
\end{equation}
 where 
 \begin{itemize}
 \item $\bar {\bf S}_s$ denotes the group associated with the free transport~$\displaystyle V \cdot \nabla_X + \frac1\alpha\sum_{i=1}^{s-1} v_i \cdot \nabla_{x_i} + \frac \alpha \eps \Omega \, \partial_\Theta$ 
 \item the collision operators~$\bar C_{s,s+1} $ are defined as usual by 
 $$
   \begin{aligned}
  \big(  \bar C_{s,s+1} f^{(s+1)}\big) (Y,Z_{s-1}) &:= \frac 1 \alpha  \sum_{i=1}^s \int_{{\mathbb S} \times {\mathbb R}^2} \Big[  f^{(s+1)}(\dots, x_i, v_i',\dots , x_i, v'_{s})  \\
 &\quad   -f^{(s+1)}(\dots, x_i, v_i,\dots , x_i, v_{s})\Big]        \big((v_{s}-v_i) \cdot \nu\big)_- d\nu dv_{s}   \end{aligned}
 $$
 
 \item   and the truncated collision operator $\bar D_{s,s+1}^\dagger
= \bar D_{s,s+1}^{+, \dagger} - \bar D_{s,s+1}^{-,\dagger}$ is obtained by modifying the gain operator  
\begin{equation}
\label{eq: operator barD truncated}
\begin{aligned}
\big(  \bar D_{s,s+1}^{+, \dagger}   f^{(s+1)}\big) (Y,Z_{s-1})  &  :=   \frac1 \alpha  
\int_{[0,L_\alpha] \times {\mathbb R}^2} 
     \indc_{ \big\{ (v'_s, V',\Omega', \Theta, \sigma) \text{ \ satisfy  \eqref{good-collision1}, \eqref{good-collision2}} 
     \big\} }\\
&\quad \times f^{(s+1)} ( Y', Z_{s-1},X,v'_{s}) \Big( \big (  \frac 1 \alpha v'_{s}-V'-
 \Omega'  r_\Theta^\perp  \big ) \cdot      n_\Theta\Big)_-  \!  d  \sigma_\alpha dv_{s}
 \end{aligned}
\end{equation}
and    the loss operator 
\begin{equation}
\label{eq: operator barD- truncated}
\begin{aligned}
\big(  \bar D_{s,s+1}^{-, \dagger}   f^{(s+1)}\big) (Y,Z_{s-1})  &  :=   \frac1 \alpha  
\int_{[0,L_\alpha] \times {\mathbb R}^2} 
 \indc_{ \big\{ (v_s, V,\Omega, \Theta, \sigma) \text{ \ satisfy  \eqref{good-collision3}} \big\} }\\
&\quad \times f^{(s+1)} ( Y, Z_{s-1}, X   ,v_{s}) \Big( \big (  \frac 1 \alpha v_{s}-V-
 \Omega  r_\Theta^\perp  \big ) \cdot      n_\Theta\Big)_-  \!  d  \sigma_\alpha dv_{s}\, ,
 \end{aligned}
\end{equation}
where the additional constraint is defined as
\begin{equation}
\label{good-collision3}
\begin{aligned}
(v_s, V,\Omega, \Theta, \sigma) \hbox{ are such that starting from a configuration $ ( Y, Z_{s-1}, X   + \frac \eps \alpha \,     r_{\alpha,\Theta},v_{s}) $ } \\
 \hbox{ there is no direct recollision in the (two-body) backward evolution.}
\end{aligned}
\end{equation}
Note that this constraint depends only on $\alpha$ and not on $\eps$, as the relative size between the rigid body and the atoms scales with $\alpha$.
\end{itemize}

Note that, in order to avoid direct recollisions between the rigid body and a new atom, we modify  the   collision operators.   
  The advantage in this approach, instead of modifying the transport in (\ref{eq: modified hierarchy}) is that there is no memory effect, which allows for chaotic solutions to the modified Boltzmann hierarchy.
  One can check in particular that  the initial data~(\ref{dataBoltz})
 gives rise to the unique solution to the Boltzmann hierarchy 
 \begin{equation}\label{tensorizedsolutionhierarchy}
   \tilde f_\eps^{(s)}(t,X,V,\Theta,\Omega,Z_{s-1}) =    \tilde  g_\eps(t,Y)  \bar M_{\beta, I} (Y) \prod_{i =1}^s M_{\beta } (v_i)  
 \end{equation}
 where~$  \tilde  g_\eps$ solves the linear equation
\begin{equation}
\label{eq: modified linear boltz}
\begin{aligned}
\partial_t  \tilde g_\eps &+ V \cdot \nabla_X  \tilde  g_\eps + \frac\alpha \eps \Omega  \, \partial_\Theta  \tilde  g_\eps \\
&= \frac1{\alpha}
\int _{[0,L_{ \alpha}] \times {\mathbb R}^2}  M_{\beta } (v)  
 \Big( \indc_{(\ref{good-collision1}),(\ref{good-collision2})  }  \tilde  g_\eps(Y') \big(  (\frac1{\alpha}v' -V'- \Omega'      r_{ \Theta}^\perp
) \cdot    n_{ \Theta}\big)_-     \\
&\qquad\qquad \qquad\qquad\qquad \qquad - \indc_{(\ref{good-collision3})} \tilde  g_\eps(Y) \big(  (\frac1{\alpha}v
-V- \Omega    r_{ \Theta}^\perp
\big ) \cdot    n_{ \Theta} \big)_-\Big)\, d  \sigma_\alpha dv  \, .  
\end{aligned}
\end{equation}

We further introduce the notation
\begin{equation}
\label{iterated-Duhamel2}
     \tilde  f^{(s)} _\eps(t) =  \sum_{n=0}^{\infty}  \bar Q_{s,s+n}^\dagger (t)  f^{(s+n)}_0  \, ,
\end{equation}
with
$$
	\begin{aligned}
\bar Q_{s,s+n} ^\dagger (t) :=    \int_0^t \int_0^{t_{s}}\dots  \int_0^{t_{s+n-2}}\bar  {\bf S}_s(t-t_{s}) (\bar  C_{s,s+1} +\bar D^\dagger_{s,s+1}) \bar {\bf S} _{s+1}(t_{s}-t_{s+1}) \dots \\
\dots  \bar {\bf S}_{s+n}(t_{s+n-1})     dt_{s+n-1} \dots dt_{s}\,.
	\end{aligned}
$$
Note that  the operators~$ \bar Q_{s,s+n}^\dagger$ satisfy the   estimates  stated in Proposition~\ref{estimatelemmacontinuity}. In particular the series expansion~(\ref{iterated-Duhamel2})
converges on any time~$t \ll \alpha^2$.

\subsection{Asymptotics of the truncated Boltzmann equation}

The very rough  estimates of  Proposition~\ref{estimatelemmacontinuity} can be  improved using the fact that the solution to the Boltzmann hierarchy  is a tensor product, provided that we can establish an a priori bound for the solution to the Boltzmann equation.
There is however a small difficulty here as the collision operator in (\ref{eq: modified linear boltz}) is truncated, which breaks the symmetry.
We therefore need to relate (\ref{eq: modified linear boltz}) to  \eqref{eq: linear boltz}.

Note that 
\begin{itemize}
\item at the level of the hierarchy, it is crucial to introduce the truncations (\ref{good-collision1}), (\ref{good-collision2}) and (\ref{good-collision3}) in order to minimize the errors between the modified BBGKY hierarchy and the Boltzmann hierarchy, because these errors will sum up.
\item at the level of the Boltzmann equation, the truncation will not be  a big deal. It is indeed proved in   Appendix~\ref{proofpropproximity}  that the solutions to (\ref{eq: modified linear boltz}) and \eqref{eq: linear boltz} stay very close to each other, as stated in the following proposition.
\end{itemize}

\begin{Prop}
\label{prop: truncatedBoltztoBoltz}
Consider an initial data $g_0$    satisfying the assumptions~{\rm(\ref{g0 borne})}.
 Let~$  g_\eps $ and~$\tilde  g_\eps$  be the solutions to~{\rm(\ref{eq: linear boltz})} and~{\rm(\ref{eq: modified linear boltz})} with initial data $g_0$. {\color{black} Then there exists~$C_0$ such that 
 $$
\| \tilde g_\eps \| _{L^\infty(\R^+\times \R^4\times {\mathbb S} \times \R)} 
{\color{black} \leq C_0}
$$
and     for any $T \geq 0$
there exists $C_T$ such that 
}
  $$
  \|M_{\beta,I} (g_\eps - \tilde g_\eps )\|_{L^\infty([0,T] ;  L^1( \R^4\times  {\mathbb S} \times \R))} \leq C_T \alpha ^{2\eta}
$$
where $\eta$ is the exponent defining the truncations  {\rm(\ref{good-collision1}), (\ref{good-collision2})}.
\end{Prop}

\section{Control of collisions}
\label{pruning}
 
\subsection{Pruning procedure}
\label{subsec: Pruning procedure}

We now recall  the strategy devised in \cite{BGSR1} in order  to control the convergence of the series expansions, or equivalently the number of collisions for  times much longer than the mean free time, in a linear setting. Here by collision we mean the collision of a particle with a new one through the collision operator.

The idea is to introduce a sampling in time with a (small) parameter $h>0$.
Let~$\{n_k\}_{k \geq 1}$ be a  sequence of integers, typically~$n_k = 2^k$.
We    study the dynamics up to time  $t := K h$ for some large integer $K$, by splitting the time interval~$[0,t]$  into~$K$ intervals of size~$h$,  and controlling the number of collisions  on each interval: we discard   trajectories having more than~$n_k$ collisions on the  interval~$[t-kh ,t-(k-1) h]$.
Note that by construction, the trajectories are actually  followed ``backwards", from time~$t$ (large)   to time~$0$. So
we decompose the iterated Duhamel formula~\eqref{iterated-Duhamel} by writing
\begin{equation}
\label{series-expansion}
\begin{aligned}
\widetilde  f_{N+1}^{(1)} (t)    =   \sum_{j_1=0}^{n_1-1}\! \!   \dots \!  \! \sum_{j_K=0}^{n_K-1} Q^\dagger_{1,J_1} (h )Q^\dagger_{J_1,J_2} (h )
 \dots  Q^\dagger_{J_{K-1},J_K} (h )      f^{(J_K)}_{N+1,0}   + R_{N,K} (t)
 \end{aligned}
\end{equation}
with a remainder  
\begin{equation}
\label{remainder expansion}
\begin{aligned}
R_{N,K} (t) :=   \sum_{k=1}^K \; \sum_{j_1=0}^{n_1-1} \! \! \dots \! \! \sum_{j_{k-1}=0}^{n_{k-1}-1}  \!\sum_{j_k \geq n_k} \; 
 Q^\dagger_{1,J_1} (h ) \dots  Q^\dagger_{J_{k-2},J_{k-1}} (h ) \,  Q^\dagger_{J_{k-1},J_{k}} (h)   \tilde  f_{N+1}^{(J_{k})}(t-kh)   \end{aligned}
\end{equation}
with~$J_0:=1$,~$J_k :=1+ j_1 + \dots +j_k$. The first term on the right-hand side of \eqref{series-expansion} 
corresponds to  a controlled number of collisions,  and the second term is the remainder:  it represents trajectories   having at least~$n_k$ collisions during the last time lapse, of size~$h$. One proceeds in a similar way for the Boltzmann hierarchy~(\ref{eq: modified hierarchy})
and decompose it as
\begin{equation}
\label{eq: boltz hierarchy decomposition}
   \tilde   f^{(1)} _\eps(t) =  \sum_{j_1=0}^{n_1-1}\! \!   \dots \!  \! \sum_{j_K=0}^{n_K-1} \bar Q^\dagger_{1,J_1} (h ) \bar Q^\dagger_{J_1,J_2} (h )
 \dots  \bar Q^\dagger_{J_{K-1},J_K} (h )       f^{(J_K)}_{0}   + \bar R_{\eps,K} (t)
\end{equation}
with a remainder  $$
\begin{aligned}
\bar R_{\eps,K} (t) :=   \sum_{k=1}^K \; \sum_{j_1=0}^{n_1-1} \! \! \dots \! \! \sum_{j_{k-1}=0}^{n_{k-1}-1}\sum_{j_k \geq n_k} \; 
\bar Q^\dagger_{1,J_1} (h ) \dots  \bar Q^\dagger_{J_{k-2},J_{k-1}} (h ) \,  \bar Q^\dagger_{J_{k-1},J_{k}} (h)   \tilde  f_{\eps}^{(J_{k})}(t-kh) \, .
 \end{aligned}
$$

\begin{Prop}
\label{prop: prunning}
There is a constant~$C $ such that the following holds.
For any  (small)~$\gamma>0$  and~$T>1 $, if
\begin{equation}
\label{eq: h petit}
h \leq   \gamma {\color{black} \frac{\alpha^4}{CT} } 
\end{equation}
then uniformly in~$t \leq T$
\begin{equation*}
\big \| R_{N,K} (t) \big\|_{L^1} + \big \| \bar R_{\eps,K} (t) \big\|_{L^1} \leq  \|g_0\|_{L^\infty} \; \gamma \, .
\end{equation*} 
\end{Prop}

\begin{proof}
We   follow the main argument of~\cite{BGSR1}.  The maximum principle~(\ref{estimatemaxprinciple})  ensures that the~$L^\infty$ norm of the marginals are bounded at all times 
$$\begin{aligned}
\big |   \widetilde  f_{N+1}^{(s)} (t,Y,Z_{s-1}) \big | \leq        \| g_0\|_{L^\infty}  M_{\beta,I,N}^{(s)} (Y,Z_{s-1}) \, .
\end{aligned}
$$
Combining this uniform bound with the $L^\infty$ estimate on the collision operator given in Proposition~\ref{estimatelemmacontinuity}, we can bound each term of the remainder \eqref{remainder expansion}
as follows
$$
\begin{aligned}
\big| Q^\dagger_{1,J_1} \big| (h ) \dots  \big| Q^\dagger_{J_{k-2},J_{k-1}} \big| (h ) \,  \big|Q^\dagger_{J_{k-1},J_{k}} \big| (h)  &     \widetilde  f_{N+1}^{(J_{k})}(t-kh) \\
 &\leq  \|g_0\|_{L^\infty} \big| Q^\dagger_{1,J_{k-1}} \big| ((k-1)h) \,  \big|Q^\dagger_{J_{k-1},J_{k}} \big| (h)   M_{\beta,I,N}^{(J_k)} \\
& \leq \|g_0\|_{L^\infty} \Big({C_1 t\over \alpha^2} \Big) ^{J_{k-1}-1} \Big({C_1 h\over \alpha^2} \Big)^{{ j_k}} M_{3\beta/4,I}   \, .
\end{aligned}
$$
 Summing over the different intervals and recalling that~$n_k= 2^k$, we deduce  that
 \begin{align*}
\big \| R_{N,K}  (t) \big\|_{L^1} & \leq   \|g_0\|_{L^\infty}\sum_{k=1}^K \; \sum_{j_1=0}^{n_1-1} \! \! \dots \! \! \sum_{j_{k-1}=0}^{n_{k-1}-1}\sum_{j_k \geq n_k} \; 
   \Big({C_1 t\over \alpha^2} \Big) ^{J_{k-1}-1} \Big({C_1 h\over \alpha^2} \Big)^{{ j_k}} \, .\end{align*} 
Given a small parameter $\gamma >0$, we take $h$ such that \eqref{eq: h petit} holds, with~$C := C_1^2$.
The previous formula can be estimated from above 
$$
\begin{aligned}
\big \| R_{N,K}  (t) \big\|_{L^1} &\leq  \| g_0\|_{L^\infty}  \sum_{k=1}^K \; \sum_{j_1=0}^{n_1-1} \! \! \dots \! \! \sum_{j_{k-1}=0}^{n_{k-1}-1} \exp \big(  2^k \log \gamma \big) \\
 & \leq   \|g_0\|_{L^\infty} \sum_{k=1}^K \; 
\left( \prod_{i = 1}^k n_i \right) \exp \big(  2^k \log \gamma \big)\\
& \leq \|g_0\|_{L^\infty}\sum_{k=1}^K \;  \exp \left( \frac{k (k-1)}{2} + 2^k \log \gamma \right)
\leq  \|g_0\|_{L^\infty} \gamma \, .
\end{aligned} 
$$

The same argument applies to the Boltzmann hierarchy due to the specific form of the solution to the Boltzmann hierarchy 
\begin{equation}
\label{specificformboltz}
    \tilde f_\eps^{(s)}(t,Y,Z_{s-1}) :=   \tilde g_\eps(t,Y)   \bar M_{\beta, I} (Y) \prod_{i =1}^s M_{\beta } (v_i)  
\end{equation}
  together with the $L^\infty$ bound (see Proposition~\ref{prop: truncatedBoltztoBoltz})
  for the modified Boltzmann equation (\ref{eq: modified linear boltz}).

\medskip

Proposition~\ref{prop: prunning} is proved.
\end{proof}

\subsection{Truncation of large energies and separation of collision times}
\label{subsec: Control of  large velocities}

We now prove  that   pseudodynamics with   large velocities  or  close collision times contribute very little to the
iterated Duhamel series.
More precisely we first define 
\begin{equation}
\label{eq: truncated velocities}
{\mathcal V}_{s}:= \Big  \{ (Y , X_{s-1},V_{s-1})   \; \big| \; \;  
 {\mathcal E}_{s} (V  , \Omega , V_{s-1}) \leq C_0^2 |\log \eps| ^2\Big\} \, ,
\end{equation}
for some~$C_0>0$, where
$$
 {\mathcal E}_{s} (V  , \Omega , V_{s-1}):=\frac12 \big(  \sum_{j=1}^{s-1} |v_j|^2 +|V|^2 +       I   \Omega^2 \big) \, .
$$
Then let
\begin{equation}\label{reste cut}
  R_{N,K}^{vel} (t) 
    :=
   \sum_{j_1=0}^{n_1-1}\! \!   \dots \!  \! \sum_{j_K=0}^{n_K-1} Q^\dagger_{1,J_1} (h )Q^\dagger_{J_1,J_2} (h )
 \dots  Q^\dagger_{J_{K-1},J_K} (h ) \, \left(
 f^{(J_K)}_{N+1,0}    
 \indc_{{\mathcal V}_{J_K}^c }\right) 
\end{equation} and
 $$\begin{aligned}
\bar R_{\eps,K}^{ vel} (t)   :=
   \sum_{j_1=0}^{n_1-1}\! \!   \dots \!  \! \sum_{j_K=0}^{n_K-1} 
\bar Q^\dagger_{1,J_1} (h )  \, \bar Q^\dagger_{J_1,J_2} (h )
 \dots  \bar  Q^\dagger_{J_{K-1},J_K} (h ) \, \left(   f^{(J_K)}_{0}    
 \indc_{{\mathcal V}_{J_K}^c }\right)   \, .
\end{aligned}
$$
The contribution of   large energies can be estimated  by the following result.
\begin{Prop} 
\label{lem: truncation0}
There exists a constant $C \geq 0 $ such that for all $t\in [0,T]$
$$
\begin{aligned}
\|   R_{N,K}^{vel} (t) 
 \|_{L^1}+ \|\bar R_{\eps,K}^{ vel} (t)   \|_{L^1} \leq   \| g_0\|_{L^\infty}\Big(\frac{   CT}{\alpha^2}\Big)^{    2^K} \eps  \, .
 \end{aligned}
$$\end{Prop}
\begin{proof}
We have for $C_0$ large enough
$$ 
\begin{aligned}
\big | 
 f^{(J_K)}_{N+1,0}      \indc_{{\mathcal V}_{J_K}^c }  \big| 
& \leq   M^{(J_k)}_{\beta,N}  \indc_{{\mathcal V}_{J_K}^c }
\;  \|g_0\|_{L^\infty}
\\
& \leq\|g_0\|_{L^\infty}C^{J_k}  M_{5 \beta / 6 }^{\otimes (J_k-1)}  M_{5 \beta  / 6,I }   \exp \left( - \frac{\beta}{6}  {\mathcal E}_{J_k} (V  , \Omega , V_{J_k-1})\right)  \indc_{{\mathcal V}_{J_K}^c  }\\
& \leq  \eps\|g_0\|_{L^\infty}C^{J_k} M_{5 \beta/6}^{\otimes (J_k-1)}  M_{5 \beta  / 6,I }   \, .
\end{aligned}
$$
   Then  Proposition \ref{estimatelemmacontinuity} implies as previously, for some constant~$C \geq C_1$,
  $$
  \begin{aligned}
   \sum_{j_1=0}^{n_1-1}\! \!   \dots \! & \! \sum_{j_K=0}^{n_K-1} \left|  Q^\dagger_{1,J_1} (h )Q^\dagger_{J_1,J_2} (h )
 \dots  Q^\dagger_{J_{K-1},J_K} (h ) \, \big(
 f^{(J_K)}_{N+1,0}      \indc_{{\mathcal V}_{J_K}^c }  \big) \right|\\
& \leq  \sum_{J=1}^{2^K} \left|  Q^\dagger_{1,J}(t) \big(
 f^{(J)}_{N+1,0}      \indc_{{\mathcal V}_{J}^c }  \big) \right| \leq \| g_0\|_{L^\infty}\,  \left(\frac{CT}{\alpha^2}\right)^{  2^{K }  }   \eps  M_{5 \beta/8,I}  \,.
 \end{aligned}
 $$
The remainder in the Boltzmann series expansion can be controlled in the same way and this concludes the proof of Proposition~\ref{lem: truncation0}.
\end{proof}

In a similar way, we   remove pseudodynamics with close collision times. Let~$\delta>0$ be a given small parameter, and  define
\begin{equation}
\label{duhamelpseudotraj1}
\begin{aligned}
 \widetilde
 f_{N+1}^{ (1,K) } (t)    : =   \sum_{j_1=0}^{n_1-1}\! \!   \dots \!  \! \sum_{j_K=0}^{n_K-1} Q^\delta_{1,J_1} (h )Q^\delta_{J_1,J_2} (h )
 \dots  Q^\delta_{J_{K-1},J_K} (h )       \left(
f^{(J_K)}_{N+1,0}   \indc_{{\mathcal V}_{J_K}  } \right) \, ,
 \end{aligned}
\end{equation}
and
\begin{equation}
\label{duhamelpseudotrajboltz}
 \begin{aligned}
\tilde  f^{ (1,K) } _\eps(t) :=  \sum_{j_1=0}^{n_1-1}\! \!   \dots \!  \! \sum_{j_K=0}^{n_K-1} \bar Q^\delta_{1,J_1} (h ) \bar Q^\delta_{J_1,J_2} (h )
 \dots  \bar Q^\delta_{J_{K-1},J_K} (h )  \left(    f^{(J_K)}_{0}   \indc_{{\mathcal V}_{J_K}  } \right)  \, ,
 \end{aligned}
\end{equation}
with  
$$
\begin{aligned}
Q^\delta_{s,s+n} (t) :=    \int  {\bf S} _s(t-t_{s}) ( C_{s,s+1} +D^\dagger_{s,s+1}) {\bf S} _{s+1}(t_{s}-t_{s+1})\dots  \\
\dots  {\bf S} _{s+n}(t_{s+n-1})\Big( \prod  {\mathbf 1}_{t_{i-1}-t_i \geq \delta}\Big)\,  dt_{s+n-1} \dots dt_{s} \, , \\
 \bar Q^\delta_{s,s+n} (t) :=    \int  \bar {\bf S} _s(t-t_{s}) (  \bar C_{s,s+1} + \bar D^\dagger_{s,s+1})  \bar {\bf S} _{s+1}(t_{s}-t_{s+1})\dots  \\
\dots   \bar {\bf S} _{s+n}(t_{s+n-1}) \Big(    \prod  {\mathbf 1}_{t_{i-1}-t_i \geq \delta}\Big)\,  dt_{s+n-1} \dots dt_{s}\,.
	\end{aligned}
$$
Applying the  continuity bounds for the transport and collision operators obtained in Paragraph~\ref{subsec : continuity estimates}, one   proves easily that
$$
 R_{N,K}^{ \delta,vel}   :=\widetilde  f_{N+1}^{(1 ) }  - \widetilde f_{N+1}^{  (1,K)  } -  R_{N,K}^{   vel}  -  R_{N,K} \quad \mbox{and} \quad  \bar R_{\eps,K}^{\delta,vel }   :=  \tilde  f^{(1 ) } _\eps- \tilde   f^{ (1,K) } _\eps -  \bar R_{\eps,K}^{  vel} -  \bar R_{\eps,K}
 $$
satisfy the  estimates given in the following proposition.
\begin{Prop} 
\label{lem: truncation}
There exists a constant $C \geq 0 $ such that for all $t\in [0,T]$
$$
\begin{aligned}
\|  R_{N,K}^{\delta,vel}   \|_{L^1}+ \| \bar  R_{\eps,K}^{\delta,vel}   \|_{L^1} \leq    \| g_0\|_{L^\infty}\left(  {\delta\over \alpha^2} \right) \Big(\frac{   CT}{\alpha^2}\Big)^{    2^K}    \,.
 \end{aligned}
$$ 
\end{Prop}

In the sequel, we will choose typically $\delta=\eps^{1/3}$, so that  in particular
\begin{equation}\label{choiceparameters}
{\delta \over \alpha} |\log \eps| \ll 1\quad \mbox{and} \quad \delta \gg \eps^{2/3}\,.
\end{equation}

 \section{Coupling pseudo-trajectories}
\label{recollisions}
 
\subsection{Collision trees}
\label{subsec: iterated Duhamel formula and collision trees}

In Section \ref{appendixprobaduhamel}, the iterated Duhamel series were interpreted in terms of 
pseudo-trajectories. A similar graphical representation holds for the series expansions~(\ref{duhamelpseudotraj1}) and~(\ref{duhamelpseudotrajboltz}) and we explain below how    pseudo-trajectories have to be modified to take into account the killing procedure.

Given a collision tree $a \in \cA_s$ (recall Definition \ref{trees-def}),   pseudo-dynamics 
start at time~$t$ from the coordinates $Y= (X,V,\Theta,\Omega)$ of the molecule (which has label 0)
and then go backward in time.
The pruning procedure and the time separation induce some constraints on the branching times.
\begin{Def}[Admissible sequences of times]
 \label{times-def}
 Let $s\geq 1$, $t\geq 0$ be fixed. An admissible sequence of times  $T_{1,s-1}=(t_i)_{1\leq i\leq s-1}$ is a decreasing sequence of $[0,t] $ 
 \begin{itemize}
 \item
 having at most $2^k$ elements in $[t- kh, t-(k-1)h]$ for~$k \in [0,K]$;
 \item and such that
 $ t_i - t_{i+1} \geq \delta  \hbox{ with } t_0 = t, t_s = 0\,.$
 \end{itemize}
 We will denote $\cT _{1,s-1}$ the set of such sequences.
 
 \end{Def}

\begin{Def}[Pseudo-trajectory]
\label{pseudotrajectory}
Fix~$s \geq 1$,  a collision tree~$a \in \cA_s$ and~$Y= (X,V,\Theta,\Omega)$, and consider a collection~$(T_{1,s-1}, \cN_{1,s-1}, V_{1,s-1}) \in \cT_{1,s-1} \times ({\mathbb S} \times B_{C_0|\log \eps|})^{s-1} $  of times, impact parameters and velocities,
where we identify  $\partial   \Sigma_\alpha$  to ${\mathbb S}$. 

 We   define recursively   pseudo-trajectories in terms of the  backward BBGKY dynamics as follows
\begin{itemize}
\item in between the  collision times~$t_i$ and~$t_{i+1}$   the particles follow the~$i+1$-particle backward flow with the partially absorbing reflection~{\rm(\ref{eq: killed bc})}; 
\item  at time~$t_i^+$, if $a_i \neq 0$,   the atom labeled~$i$ is adjoined to atom $a_i$ at position~$x_{a_i} + \eps \nu_{i}$ and 
with velocity~$v_{i}$. 
If $ (v_{i} - v_{a_i} (t_i^+)) \cdot \nu_i >0$, velocities at time $t_i^-$ are given by the scattering laws \eqref{V-def}.

{The pseudo-trajectory is  killed if some particles overlap.}

\item at time~$t_i^+$, if $a_i = 0$,  the atom labeled~$i$    is adjoined to the rigid body at position~$X+{\eps \over \alpha}R_{\Theta(t_i^+)}   r_{\alpha, i}  $ and 
with velocity~$v_{i}$. 

If $ (\alpha^{-1}  v_{i} - V (t_i^+)-\Omega(t_i^+)  R_{\Theta(t_i^+)}   r_{i}^\perp  ) \cdot R_{\Theta(t_i^+)}   n_{i} >0$, velocities at time $t_i^-$ are given by the scattering laws \eqref{VOmega-def}.
{The pseudo-trajectory is  killed if some particles overlap or   if~{\rm(\ref{good-collision1}),(\ref{good-collision2})} do not hold in the post-collisional case.}
\end{itemize}
We denote $\big(     Y,  Z_{s-1}\big) (a, T_{1,s-1}, \cN_{1,s-1}, V_{1,s-1},0)$ the initial configuration.  

\medskip

Similarly, we define    pseudo-trajectories associated with the modified Boltzmann hierarchy. 
These pseudo-trajectories evolve according to the backward Boltzmann dynamics as follows
\begin{itemize}
\item in between the  collision times~$t_i$ and~$t_{i+1}$   the particles follow the~$i+1$-particle backward free flow;
\item at time~$t_i^+$,  if $a_i \neq 0$, atom   $i$ is adjoined to atom $a_i$  at exactly the same position~$x_{a_i}$. Velocities are given by the laws {\rm(\ref{V-def})} if there is scattering.
\item at time~$t_i^+$,  if $a_i =0$, atom   $i$ is adjoined to the rigid body  at exactly the same position~$X$. Velocities are given by the laws {\rm(\ref{VOmega-def})} if there is scattering.

{  The pseudo-trajectory is  killed   if~{\rm (\ref{good-collision1}),(\ref{good-collision2})}
do not hold in the post-collisional case, and if~{\rm(\ref{good-collision3})} does not hold in the pre-collisional case.}

\end{itemize}
We denote $\big(   \bar Y,\bar Z_{s-1}\big) (a, T_{1,s-1}, \cN_{1,s-1}, V_{1,s-1},0)$ the initial configuration. \end{Def}

\begin{Def}[Admissible parameters]
\label{def: overlap}
Given~$Y = (X,V,\Theta,\Omega)   $ and a collision tree~$a \in \cA_s$,  the  set of admissible parameters are  defined by
\begin{align*}
G_s(a) := \Big\{(T_{1,s-1}, \cN_{1,s-1}, V_{1,s-1}) \in \cT_{1,s-1} \times ({\mathbb S} \times B_{C_0|\log \eps|})^{s-1}   \, \Big/ \, \\
\begin{aligned}
 &\hbox{ the pseudotrajectory }  (Y, Z_{s-1})(a, T_{1,s-1}, \Omega_{1,s-1}, V_{1,s-1},\tau)\\
 & \, \hbox{ is defined backwards up to time 0 with }  {\mathcal E}_{s} (V  , \Omega , V_{s-1}) (0) \leq C_0^2 |\log \eps| ^2 \Big \} \, ,
  \end{aligned} 
\end{align*}
and
\begin{align*}
\bar G_s(a) := \Big\{(T_{1,s-1}, \cN_{1,s-1}, V_{1,s-1}) \in \cT_{1,s-1} \times ({\mathbb S} \times B_{C_0|\log \eps|})^{s-1}   \, \Big/ \, \\
\begin{aligned}
 &\hbox{ the pseudotrajectory }  (\bar Y, \bar Z_{s-1})(a, T_{1,s-1}, \Omega_{1,s-1}, V_{1,s-1} ,\tau)\\
 & \, \hbox{ is defined backwards up to time 0 with }  {\mathcal E}_{s} (V  , \Omega , V_{s-1}) (0) \leq C_0^2 |\log \eps| ^2  \Big \} \, .
  \end{aligned}
\end{align*}
\end{Def}

We recall following important semantic distinction.
\begin{Def}[Collisions/Recollisions]
\label{def: recollisions}
The term collision will be used only for the creation of a new atom, i.e. for a branching in the collision trees. 
A shock between two particles in the  backward dynamics will be called a recollision.
\end{Def}
Note that no recollision can occur in the Boltzmann hierarchy as the particles have zero diameter. 

\medskip

With these notations the iterated Duhamel formula (\ref{duhamelpseudotraj1}) and (\ref{duhamelpseudotrajboltz}) for the first marginals  can be rewritten 
\begin{equation}
\label{duhamelpseudotraj}
\begin{aligned}
\tilde  f_{N+1}^{(1,K)} (t) =\sum_{s=0}^{N}   N& \dots\big(N-(s-2)\big) \eps^{s-1}\sum_{a \in \cA_s}   \int_{G_s(a)} dT_{1,s-1}   d\cN_{1,s-1}  dV_{1,s-1} \\
 &\times 
\Big( \prod_{i=1}^{s-1}   b_i\Big)
f_{N+1,0}^{(s)}\big ((Y,Z_{s-1})(a, T_{1,s-1}, \cN_{1,s-1}, V_{1,s-1},0)\big)  \, ,
 \end{aligned}
\end{equation}
and
\begin{equation}
\label{duhamelpseudotraj2}
 \begin{aligned}
  \tilde f^{(1,K)} _\eps(t) =\sum_{s=1}^{\infty}   &  \sum_{a \in \cA_s}  \int_{\bar G_s(a)}  dT_{1,s-1}   d\cN_{1,s-1}  dV_{1,s-1}  \\
 &\times  \Big( \prod_{i=1}^{s-1}   b_i\Big)f_{0}^{(s)}\big ((\bar Y,\bar Z_{s-1})(a, T_{1,s-1}, \cN_{1,s-1}, V_{1,s-1},0)\big)  \, ,
 \end{aligned}
\end{equation}
denoting  
$$\begin{aligned}
& b_{i} : =\alpha^{-1}
 (v_i -v_{a_i} (t_i)) \cdot \nu_i  \quad \hbox{ if }\quad a_i \neq 0 \,  ,\\
&b_{i} :=\alpha^{-1}
(\alpha^{-1} v_i -V (t_i)- \Omega(t_i) \,   r_{ \Theta }^\perp) \cdot    n_{ \Theta }  \quad\hbox{ if } \quad a_i =0 \, .
\end{aligned}
$$
We stress the fact that the contributions of the loss and gain terms in \eqref{duhamelpseudotraj} and~\eqref{duhamelpseudotraj2} are coded
in the sign of $b_i$.

\bigskip

In order to show that  $\tilde  f_{N+1}^{(1,K)} $ and $ \tilde  f^{(1,K)} _\eps$
   are close to each other when $N$ diverges, we shall prove in Section~\ref{endoftheproofthm1} that the pseudo-trajectories $(Y,Z_{s-1})$ and~$(\bar Y,\bar Z_{s-1})$ can be coupled   up to a small error  due to the micro-translations of the added particle at each collision time $t_k$, provided that   the set of parameters leading to the following events is discarded:
\begin{itemize}
\item  recollisions on the interval $]t_k, t_{k-1}[$ along the flow ${\bf S}^\dagger_k$   (which do not occur for the free flow $\bar {\bf S}_k$);
\item  killing the Boltzmann or the BBGKY  pseudo-trajectory without killing  the other.
\end{itemize}

\subsection{Geometry of the recollision sets}
\label{subsec: Control of recollisions}

 The next step is to construct a small set of deflection angles and velocities  such that the pseudo-trajectories $(Y,Z)$ induced by the complementary of this set set have no recollision and therefore remain very close to the pseudo-trajectories $(\bar Y, \bar Z)$ associated with the free flow. These good pseudo-trajectories will be identified  by a recursive process selecting for each $k$, good configurations with $k$ atoms.

By definition, a good configuration with $k$ atoms is such that the atoms and the rigid body remain at a distance $\eps_0 \gg \eps/\alpha $ one from another for a time $T$ 
$$\begin{aligned}
\cG_k(\eps_0) := \Big\{ (Y,Z_{k-1})\in {\mathcal V}_k \; 
\Big| \; \forall u \in [0,T]  ,\quad \forall i\neq j ,\quad  |(x_i -u \,{ v_i\over \alpha} ) - ( x_j- u \,{ v_j\over \alpha}) |\geq \eps_0 \\
 \hbox{ and } |(X-u \, V)- (  x_j- u \, {v_j\over \alpha}) |\geq \eps_0  \Big\}\,.
\end{aligned}
$$
On  $\cG_k(\eps_0)$, ${\bf S}_{k}^\dagger $ coincides with the free flow.    

If   the configurations $(Y, Z_{k-1})$, $(\bar Y,  \bar Z_{k-1})$ are  such that  $Y=\bar Y$ and 
$$
\forall i \in [1,k-1] \, , \qquad |x_i - \bar x_i| \leq  \eps_0/2    \, ,
\qquad 
v_i = \bar v_i
$$
and if $(\bar Y,  \bar Z_{k-1})$ belongs to $\cG_k(\eps_0)$, then  $ (Y,Z_{k-1})\in \cG_k(\eps_0/2)$   and there is no recollision as long as no new particle is adjoined. 

\medskip

\noindent We are now going to show that   good configurations are stable by adjunction of a 
$k^{\text{th}}$-atom.
More precisely, let $(\bar Y,  \bar Z_{k-1})$  and  
$(Y, Z_{k-1})$ be in~$\cG_k(\eps_0)$ with  
\begin{equation}
\label{eq: example}
Y = \bar Y, \qquad 
 \sup _{j\leq k-1} |x_j - \bar x_j | \leq { 2 r_{max} \eps \over \alpha} + (k-1) \eps \, , \qquad V_{k-1} = \bar V_{k-1}\,.
\end{equation}
Then,
by  choosing the velocity $v_{k}$ and the deflection angle 
$\nu_{k}$ of the new particle $k$ outside a bad set $\cB_{k} (\bar Y,  \bar Z_{k-1})$, both configurations
$(Y, Z_{k})$ and  $(\bar Y,  \bar Z_{k})$ will remain close to each other.
Immediately after the adjunction, the colliding particles $a_k$ and~$k $ will not be at distance~$\eps_0$, but~$v_{k }, \nu_{k }$ will be chosen such that the particles drift rapidly far apart and
after a short time $\delta>0$ the configurations $(Y, Z_{k})$ and  $(\bar  Y,   \bar Z_{k})$
will be again in the good sets $\cG_{k+1} (\eps_0/2)$ and~$\cG_{k+1} (\eps_0)$.
By construction, $Y = \bar Y$ at all times as long as there is no recollision.
The following proposition defines  and quantifies the bad sets outside of which particles drift rapidly far apart.

For the sake of simplicity, we assume without loss of generality that the time of addition is $t_k = 0$. Note that no overlap can occur  here since  the $k$ existing particles  are far from each other.

\begin{Prop}
\label{geometric-prop}
We   fix   a parameter~$ \eps_0\gg \eps/\alpha$ and assume that~{\rm(\ref{choiceparameters})} holds. Given~$(\bar Y,  \bar Z_{k-1}) $ in~$ \cG_k(\eps_0)$, there is a 
subset~$\cB_{k} (\bar Y,  \bar Z_{k-1})$ 
of~${\mathbb S} \times B_ {C_0|\log \eps|} $  of small measure
\begin{equation}
\label{pathological-size}
\int  \indc_{ {\mathcal V} _k } \indc_{ \cB_{k} (\bar Y,  \bar Z_{k-1})} b_k dv_k d\nu_k   \leq  \frac{CT^2 k^2}{\kappa_{min}\alpha^5 }  |\log \eps|  ^7 \left(  {\eps \over  \alpha \eps_0}+  {  \eps_0 \alpha \over \delta}  +{\eps^{1/3} \over \alpha}  \right) 
\end{equation}
such that good configurations close to $(\bar Y,  \bar Z_{k-1}) $  are stable by adjunction of a collisional particle close to any particle~${a_k}$, and remain close to~$(\bar Y,  \bar Z_{k}) $ in the following sense. 

\smallskip
\noindent 	 
Let~$(Y,  Z_{k-1})$ be a configuration with $k-1$ atoms, satisfying \eqref{eq: example}.

$ \bullet $ If  $1\leq a_k\leq k-1$, a new atom with velocity $v_{k}$ is added to $(Y,  Z_{k-1})$
at $x_{a_k} + \eps \nu_{k}$, and to $(\bar Y,  \bar Z_{k-1})$ at $\bar x_{a_k}$. Post-collisional velocities $(v_{a_k}, v_k)$ are   updated by scattering to pre-collisional velocities.

Then, if $(\nu_{k},v_{k}) \in {\mathbb S} \times B_ {C_0|\log \eps|} \setminus \cB_{ k} (\bar Y,  \bar Z_{k-1})$,
the configuration $(Y, Z_k)$ has no  recollision under  the backward flow, and $(\bar Y,  \bar Z_{k})$ becomes a good configuration after a time lapse $\delta$:
$$ (\bar Y, \bar Z_k) (\delta) \in \cG_{k+1} (\eps_0) \,.$$
Moreover,  for all $t\in [0,T]$
\begin{equation}
\label{goodconfig1}
Y = \bar Y\, ,\quad   \sup _{j\leq k} |x_j - \bar x_j | \leq {2  r_{max}\eps \over \alpha} + k \eps  \, ,\quad \hbox{ and } \quad V_{k} = \bar V_k \, .
\end{equation}

$ \bullet$ If $a_k = 0$, a new atom with velocity $v_{k}$ is added to $(Y,  Z_{k-1})$
at $X + (\eps/\alpha)   r_{\alpha, \Theta }$, and to~$(\bar Y,  \bar Z_{k-1})$ at $\bar X$. Pre-collisional configurations are killed for the Boltzmann dynamics if~{\rm(\ref{good-collision3})} does not hold.
Post-collisional configurations are killed for both dynamics if ~{\rm(\ref{good-collision1})(\ref{good-collision2})} do not hold, and updated by scattering to pre-collisional velocities if not.

Then, if $(\nu_{k},v_{k}) \in [0,L_\alpha]  \times B_ {C_0|\log \eps|} \setminus  \cB_{0}( \bar Y,  \bar Z_{k-1})$,\\
- either both pseudo-trajectories are killed before time $\delta$;\\
- or the configuration $(Y, Z_k)$ has no collision under  the backward flow, and $(\bar Y,  \bar Z_{k-1}) $becomes a good configuration after a time lapse $\delta$:
$$ (\bar Y, \bar Z_k) (\delta) \in \cG_{k+1}(\eps_0) \, . $$
Moreover
\begin{equation}\label{goodconfig2}
Y = \bar Y\, ,\quad    \sup _{j\leq k} |x_j - \bar x_j | \leq {2  r_{max}\eps \over \alpha} + k \eps  \, ,\quad \hbox{ and } \quad V_{k} = \bar V_k \, .
\end{equation}
\end{Prop} 

%
%

 \begin{proof}[Proof of Proposition~{\rm\ref{geometric-prop}}]
 The proof follows closely the arguments in \cite{GSRT,BGSR1}. The main difference lies in the possible killing of trajectories to avoid pathological recollisions with the molecule.

 The conditions for $(\bar Y, \bar Z_k)$ to be a good configuration after a  time lapse $\delta$ can be written simply
$$
\forall u \geq \delta, \, \forall j \notin \{ k, a_k\},\, \forall q \in \Z^2 \cap B_{C_0|\log \eps| T/\alpha} \,,\quad \begin{cases}
  | q - \frac u\alpha (v_k^- - v_{a_k}^- )| \geq \eps_0\, , \\
   | q + \bar x_{a_k}  - \bar x_j- \frac u\alpha (v_k^- - v_j )| \geq \eps_0 \,,\\
   | q + \bar x_{a_k}  - \bar x_j- \frac u\alpha (v_{a_k} ^- - v_j )|\geq \eps_0\,,   
\end{cases}
$$
   denoting     with a slight abuse $\bar x_0 = \bar X$ and $v_0 = \alpha V$.
   This means that $v_k^-$ and $v_{a_k} ^-$ have to be outside   a union of $(C_0|\log \eps| T/\alpha)^2$ rectangles of width at most $\eps_0\alpha / \delta$.
   Note   that the last condition on $v_{a_k}^- $ is not necessary in the absence of scattering, since we already know that the initial configuration is a good configuration.

\medskip

   The conditions for $( Y,  Z_k)$ to have no recollision  are a little bit more involved. Note that, provided there is no recollision, since the velocities are equal 
   $$ V=\bar V\, , \quad \Omega = \bar \Omega \quad\hbox{ and } \quad V_k = \bar V_k$$
  then $( Y,  Z_k)$ will stay close to $(\bar Y, \bar Z_k)$ and therefore it will be a good configuration after a  time lapse $\delta$.
 We therefore only have to check that~$( Y,  Z_k)$  has no recollisions on~$[0,\delta]$.
 \noindent
 \underline{Case of a collision between two atoms}.

By definition of pre-collisional velocities $(v_{a_k}^-, v_k^-)$, we know that for short times $u\leq \delta$, atoms $k$ and $a_k$ will not recollide directly one with the other.
 Indeed,  they  move away from each other, and no periodic recollision may occur since all velocities are bounded by $C_0|\log \eps|$   and there holds $\delta |\log \eps|/\alpha  \ll 1$.

We then need to    ensure that for short times $u\leq \delta$, atoms $a_k$ and $k$ cannot recollide with another atom $j\neq a_k,k$ nor with the rigid body $j=0$. A necessary condition for such a recollision to hold  is that there exist $u\geq 0$ and $q \in \Z^2 \cap B_{C_0|\log \eps| T/\alpha} $ such that 
  $$\begin{aligned}
   | q + \bar x_{a_k}  - \bar x_j- \frac u\alpha (v_k^- - v_j )| \leq {2 r_{max}\eps \over \alpha} +2 (k-1) \eps   \,,\\
   | q + \bar x_{a_k}  - \bar x_j- \frac u\alpha (v_{a_k} ^- - v_j )| \leq  {2 r_{max}\eps \over \alpha} +2 (k-1) \eps  \,.\\
   \end{aligned}$$
    Since $|\bar x_{a_k}  - \bar x_j+q |\geq \eps_0$ for all $q \in \Z^2$, this means that $v_k^-$ or $v_{a_k} ^-$ has to belong to a union of $(C_0|\log \eps| T/\alpha )^2$ cylinders of opening 
 $3\eps_1/\eps_0$, defining
 $$
 \eps_1:= 2  r_{max} \eps/\alpha+2 (k-1) \eps\,.
 $$

  \medskip
 \noindent
 \underline{Case of a collision between an atom and the rigid body}.

 For post-collisional configurations, either (\ref{good-collision1}) and (\ref{good-collision2}) hold and thanks to Proposition~\ref{norecollision} there will be no direct recollision between $k$ and the rigid body in the backward dynamics, or (\ref{good-collision1}) or (\ref{good-collision2}) fail to hold, in which case both the Boltzmann and the BBGKY pseudo-trajectories are killed.

 The case of pre-collisional configurations is a little bit more involved. If (\ref{good-collision3}) is satisfied, there will be no direct recollision between $k$ and the rigid body in the backward dynamics.  If (\ref{good-collision3}) is  not satisfied, we first exclude  small relative velocities 
 \begin{equation}\label{relativevelocities13}
  |v_k - \alpha V | \leq \eps^{1/3}\,.
  \end{equation}
Then the geometric argument (\ref{tmax})
in  Section \ref{conditioning-geometry} shows that the first recollision   between the rigid body  and the particle $k$ (which exists by definition) has to happen before time (recalling that  Section \ref{conditioning-geometry} was under a scaling by~$\alpha/\eps$) 
$$ t_{max} := {2 r_{max} \eps  \over \min |v_k - \alpha V|} \leq C \eps^{2/3} \ll \delta \,.$$
As all other particles are at a distance at least~$\eps_0 - \eps_1 - C_0|\log \eps| t_{max}/\alpha $ from the rigid body, there cannot be any interaction, changing the two-body dynamics during this time lapse.
This means that the BBGKY pseudo-trajectory is killed as well  before time $\delta$.

   Then we check that all other situations (recollisions  with another atom~$j\neq k$) do not use the geometry of particles (replacing the rigid body by   a security sphere of radius~$r_{max}\eps /\alpha$ around it). So according to the previous paragraph  we find that  $v_k^-$ and $V^-$ have to be outside a union of $(C_0|\log \eps| T/\alpha)^2$ rectangles of sizes $ C|\log \eps|^2 \eps_1/\eps_0$. We also need, as in~(\ref{relativevelocities13}), for~$V^-$ to be  outside a union of $ k-1$ balls of radius~$\eps^{1/3} /\alpha$.

 \bigskip

From these conditions on the pre-collisional velocities, we now deduce the definition and estimate on the bad sets $ \cB_{k} (\bar Y,  \bar Z_{k-1}).$

 If $a_k \neq 0$,  using Lemma \ref{aa-scattering} in the Appendix to translate these conditions in terms of~$(\nu_k, v_k)$, we   obtain  that there is a 
subset~$\cB_{k, a_k} (\bar Y,  \bar Z_{k-1})$ 
of~${\mathbb S} \times B_ {C_0|\log \eps|} $ with
 $$\int_{{\mathcal V} _k} \indc_{ \cB_{ k, a_k} (\bar Y,  \bar Z_{k-1})} b_k dv_k d\nu_k   \leq  {Ck \over \alpha} \left({C_0|\log \eps| T\over \alpha} \right)^2 (C_0 |\log \eps|  )^2 |\log \eps| \left((C_0| \log \eps|)^2 {\eps_1\over \eps_0} +C_0 |\log \eps| {\eps_0 \alpha  \over \delta} \right)$$
 such that the addition of an atom close to $a_k$ with $(\nu_k, v_k) \notin  \cB_{ k, a_k} (\bar Y,  \bar Z_{k-1})$ provides a  good pseudo-trajectory.

 If $a_k= 0$, we need to compute the pre-image of the bad sets by the scattering. By Lemma~\ref{am-scattering} in the Appendix, we   obtain that there is a set $\cB_{ 0} (\bar Y,  \bar Z_{k-1})$ of measure
 $$
\begin{aligned}
 \int_{{\mathcal V} _k} \indc_{ \cB_{ k,0} (\bar Y,  \bar Z_{k-1})} b_k dv_k d\nu_k   &\leq  {Ck\over \kappa_{min} \alpha^4 } \left({C_0|\log \eps| T\over \alpha} \right)^2(C_0  |\log \eps|) ^2 |\log \eps|\\
&\quad  \times\left((C_0|\log \eps|^2 ){\eps_1\over \eps_0} +C_0 |\log \eps| {\eps_0 \alpha \over \delta} +{\eps^{\frac13}\over \alpha} \right)
\end{aligned} $$
 such that the addition of an atom close to the rigid body with $(\nu_k, v_k) \notin \cB_{ k,0}$ provides a good pseudo-trajectory.
 Proposition~\ref{geometric-prop} is proved.
 \end{proof}

\subsection{Truncation of the collision parameters}\label{v-nu-truncation}


Thanks to Proposition \ref{geometric-prop} we know that given a good configuration~$(\bar Y,\overline Z_{s-1})$, 
  if the adjoined particle does not belong to~$ {\mathcal B}_{ s }( \bar Y ,\overline Z_{s-1}) $ then the resulting configuration~$(\bar Y,\overline Z_{s})$ is again a good configuration after the time~$\delta$. As a consequence we can define recursively the  set of good parameters as follows~:
\begin{Def}[Good parameters]
\label{def: overlap}
Given~$Y = (X,V,\Theta,\Omega)   $ and a collision tree~$a \in \cA_s$,  we say that $(T_{1,s-1}, \cN_{1,s-1}, V_{1,s-1})$ is a sequence of good parameters if
\begin{itemize}
\item $T_{1,s-1}$ is a sequence of admissible times;
\item for all $k \in \{1,\dots, s-1\}$,  the following recursive condition holds
$$(\nu_k, v_k) \in ({\mathbb S} \times B_{C_0|\log \eps|} )\setminus \cB_k ( (\bar Y,  \bar Z_{k-1}) (t_k)) ;$$
\item the total energy at time 0 satisfies
 $${\mathcal E}_{s} (V  , \Omega , V_{s-1}) (0) \leq C_0^2 |\log \eps| ^2\,.$$
\end{itemize}
We denote by $G_s^0(a)$ the set of good parameters.
\end{Def}

We the define the approximate BBGKY and Boltzmann solutions by~:
\begin{equation}
\label{duhamelpseudotrajbis}
\begin{aligned}
\tilde  f_{N+1}^{(1,K),0} (t) =\sum_{s=0}^{N}   N& \dots\big(N-(s-2)\big) \eps^{s-1}\sum_{a \in \cA_s}   \int_{G^0_s(a)} dT_{1,s-1}   d\cN_{1,s-1}  dV_{1,s-1} \\
 &\times 
\Big( \prod_{i=1}^{s-1}   b_i\Big)
f_{N+1,0}^{(s)}\big ((Y,Z_{s-1})(a, T_{1,s-1}, \cN_{1,s-1}, V_{1,s-1},0)\big)  \, ,
 \end{aligned}
\end{equation}
and
\begin{equation}
\label{duhamelpseudotrajbis2}
 \begin{aligned}
 \tilde f^{(1,K),0} _\eps(t) =\sum_{s=1}^{\infty}   &  \sum_{a \in \cA_s}  \int_{ G^0_s(a)}  dT_{1,s-1}   d\cN_{1,s-1}  dV_{1,s-1}  \\
 &\times  \Big( \prod_{i=1}^{s-1}   b_i\Big)f_{0}^{(s)}\big ((\bar Y,\bar Z_{s-1})(a, T_{1,s-1}, \cN_{1,s-1}, V_{1,s-1},0)\big)  \, .
 \end{aligned}
\end{equation}

\medskip
The results proved in  this   paragraph  imply directly the following proposition, choosing~$\eps_0= \eps^{2/3},\delta = \eps^{1/3}$, and recalling that~$\alpha \gg |\log \eps|$.
\begin{Prop} 
\label{prop: 1 recoll}
The  contribution of pseudo-trajectories involving  recollisions  is bounded by
 $$
 \begin{aligned}
 \forall t\in [0,T] \, , \quad  \|  \tilde  f_{N+1}^{ (1,K) } (t) - \widetilde f_{N+1}^{ (1,K) , 0} (t) \|_{L^1} &+ \| \widetilde  f_{\eps}^{ (1,K) } (t) - \tilde f_{\eps}^{ (1,K) , 0} (t) \|_{L^1} \\
&\quad   \leq  \|g_0\|_{L^\infty}  
  \Big( \frac{CT }{\alpha^2} \Big)^{  2^{K+1}  }    \eps^{1/4}
    \, .
\end{aligned}
$$  
\end{Prop}

\section{Convergence to the Boltzmann hierarchy}\label{endoftheproofthm1}

The last step to conclude the proof of Theorem \ref{prop-approximate tensorization} is to evaluate the difference $\widetilde f^{(1,K), 0}_{N+1}(t) -  f_\eps^{(1,K), 0} (t)$.
Once   recollisions   have been excluded, the only discrepancies between the BBGKY and the Boltzmann  pseudo-trajectories come from the micro-translations due to the diameter of the colliding 
particles (see Definition \ref{pseudotrajectory}): note that the rigid body follows the same trajectory in both settings since atoms alone are ``added" to the pseudo-dynamics. 
At the initial time, the error between the two configurations  after $s$ collisions   is given by Proposition~\ref{geometric-prop}.
Recall that the discrepancies are only for   positions, as   velocities remain equal in both hierarchies.
These configurations are then evaluated   on the marginals of the 
 initial data~$f_{N+1,0}^{(s)} $ or   $ f^{(s)} _0$ which are close to each other thanks to Proposition~\ref{exclusion-prop2}.
We have
\begin{align*}
& \Big|  f_{0 }^{(s)} \Big( (\bar Y ,\bar Z_{s-1}) (a, T_{1,s-1}, \cN_{1,s-1}, V_{1,s-1},0) \Big)
- 
 f_{N+1,0 }^{(s)} \Big( (Y ,Z _{s-1}) (a, T_{1,s-1}, \cN_{1,s-1}, V_{1,s-1},0) \Big) \Big|\\
&   \leq 
\Big| f_{0 }^{(s)} \Big(  (\bar Y ,\bar Z_{s-1}) (a, T_{1,s-1}, \cN_{1,s-1}, V_{1,s-1},0) \Big)
- 
 f_{0 }^{(s)} \Big( (Y ,Z _{s-1}) (a, T_{1,s-1}, \cN_{1,s-1}, V_{1,s-1},0) \Big) \Big|\\
&      +
\Big|  f_{0 }^{(s)} \Big( (Y ,Z _{s-1}) (a, T_{1,s-1}, \cN_{1,s-1}, V_{1,s-1},0) \Big)
- 
 f_{N+1,0 }^{(s)} \Big( (Y ,Z _{s-1}) (a, T_{1,s-1}, \cN_{1,s-1}, V_{1,s-1},0) \Big) \Big| \, .
\end{align*}
Since~$g_{0}$ has   Lipschitz regularity, by the estimate on the shift on the initial configurations given by~(\ref{goodconfig1}) and~(\ref{goodconfig2}) in Proposition \ref{geometric-prop}, we get (using the conservation of energy at each collision)
$$
\begin{aligned}
\Big| f_{0 }^{(s)} \Big(  (\bar Y ,\bar Z_{s-1}) (a, T_{1,s-1}, \cN_{1,s-1}, V_{1,s-1},0) \Big)
- 
 f_{0 }^{(s)} \Big( (Y ,Z _{s-1}) (a, T_{1,s-1}, \cN_{1,s-1}, V_{1,s-1},0) \Big) \Big|
\\ \leq C^s  \left(  \frac \eps\alpha  +2^K \eps\right)     M_{\beta, I} (V,\Omega) M_{\beta } ^{\otimes (s-1)}(V_{s-1})  \, .
\end{aligned}
$$
On the other hand, by construction, the good pseudo-trajectories reach only good configurations at time 0, which means that 
we can use the convergence of the initial data stated in Proposition \ref{exclusion-prop2}~:
$$
\begin{aligned}
\Big|  f_{0 }^{(s)} \Big( (  Y ,  Z_{s-1}) (a, T_{1,s-1}, \cN_{1,s-1}, V_{1,s-1},0) \Big)
- 
 f_{N+1,0 }^{(s)} \Big( (Y ,Z _{s-1}) (a, T_{1,s-1}, \cN_{1,s-1}, V_{1,s-1},0) \Big) \Big|\\
  \leq C^s  \frac{\eps}{\alpha^2}        M_{\beta, I} (V,\Omega) M_{\beta } ^{\otimes (s-1)}(V_{s-1})  \, .
\end{aligned}
$$

The last source of discrepancy between the formulas defining $\widetilde f^{(1,K), 0}_{N}$ and $ \widetilde f_\eps^{(1,K), 0} $
comes from the prefactor~$N\dots (N-s+2) \eps^{s-1}$ which has been replaced by $1$.
For  $s \ll N$, the corresponding error is 
$$ \Big( 1- {N \dots (N- s+2) \over N^{s-1}}\Big)  \leq C {s^2 \over N}\leq Cs^2 {\eps } $$
which, combined with the bound on the collision operators,  leads to an error of the form 
\begin{equation}
\label{error3}
\left(\frac{C    T}{\alpha^2} \right)^{s-1} s^2{\eps }   \,.
\end{equation}
Summing the previous bounds gives finally
\begin{equation}
\label{eq: erreur 0N cut}
 \Big \|\widetilde f^{(1,K), 0}_{N+1}(t) - \tilde  f_\eps^{(1,K), 0} (t) \Big\|_{L^1} \leq  C  2^{2K}   \frac{\eps}{\alpha^2}    
  \left(\frac{CT }{\alpha^2} \right)^{    2^{K+1} }
     \, .
\end{equation} 
Combining Proposition \ref{prop: 1 recoll} and (\ref{eq: erreur 0N cut}) to control  the difference in the parts with controlled branching process,
we find
 \begin{equation}
\label{mainpart-est K}
\Big\|  \widetilde f^{(1,K)} _{N+1}(t)- \tilde f_\eps^{(1,K)} (t) \Big\|_{L^1} 
\leq   
C     \left(\frac{CT }{\alpha^2} \right)^{    2^{K+1} }
   \eps^{1/4}     \, \cdotp
\end{equation}
Finally to conclude the proof of Theorem~\ref{prop-approximate tensorization}
we put together Propositions~\ref{prop: prunning},  \ref{lem: truncation0} and \ref{lem: truncation}, and use the comparison (\ref{mainpart-est K}). We find that if
$$
h < \gamma   \frac{\alpha^4}{CT}  
$$
then 
 for all~$t \leq T$
 \begin{equation*}
\Big\|  \widetilde f^{(1)} _{N+1}(t)-\tilde f_\eps^{(1)}(t) \Big\|_{L^1}  \leq   C \gamma 
+ C \left(\frac{CT }{\alpha^2} \right)^{    2^{K+1} }
    {  \eps^{1/4}  } \, ,\end{equation*} 
    so thanks to~(\ref{tensorizedsolutionhierarchy})
\begin{equation*}
\Big\|  \widetilde f^{(1)} _{N+1}(t)-  \bar M_{\beta, I}   \tilde g_\eps(t) \Big\|_{L^1}  \leq   C \gamma 
+ C \left(\frac{CT }{\alpha^2} \right)^{    2^{K+1} }
    {  \eps^{1/4}  } \, \cdotp
\end{equation*} 
Finally, we choose 
$$
K = T/h  = {CT^2 \over \gamma \alpha^4} \leq c \log\log N,
$$
with $c$ a constant small enough.
Then, we get that
$$
\Big\| \widetilde  f^{(1)} _{N+1}(t)-M_{\beta, I}    \tilde g_\eps(t) \Big\|_{L^1} 
\leq  \frac{C  T^2}{\alpha^4\log \log N }  \, \cdotp 
$$
Finally it remains to use Proposition~\ref{prop: truncatedBoltztoBoltz} 
and Corollary \ref{corconditioneddata}, giving the closeness of the truncated BBGKY (resp. Boltzmann) hierarchy and the original one, to conclude the proof of Theorem~\ref{prop-approximate tensorization}.
 \qed

\section{Convergence to the Fokker-Planck equation}
\label{proof theorem}

\subsection{The singular perturbation problem}\label{singpertpb}

Theorem~\ref{prop-approximate tensorization}  states that in the limit $\eps \to 0$, provided that $\alpha ({\log \log N})^\frac14 \gg1$, the solution $f_{N+1}^{(1)}$  is asymptotically close to the solution~$\bar M_{\beta,I} \,  g_\eps $ of the singular linear Boltzmann equation
\eqref{eq: linear boltz}: 
\begin{equation}\label{singular-Boltz}
\begin{aligned}
\partial_t g_\eps  &+ V \cdot \nabla_X g_\eps  + \frac\alpha \eps \Omega  \, \partial_\Theta g_\eps  \\
&= \frac1{\alpha}
\int _{[0,L_\alpha] \times {\mathbb R}^2}  M_{\beta } (v)  
 \Big( g_\eps (Y') \big(  ( \frac1{\alpha}v' -V'- \Omega'      r_\Theta^\perp
) \cdot     n_\Theta\big)_-      \\
&\qquad\qquad \qquad\qquad\qquad \qquad -  g_\eps (Y) \big(  ( \frac1{\alpha}v
-V-\Omega     r_\Theta^\perp
\big ) \cdot     n_\Theta \big)_-\Big)\, d  \sigma_\alpha dv \, , 
\end{aligned}
\end{equation}
with $Y= (X,V,\Theta,\Omega)$ and $Y' = (X,V',\Theta,\Omega')$ and initial data~$g_0$.

\medskip\noindent
\underbar{Constraint equation}  

From the uniform $L^\infty$ bound on $g_\eps$ (coming from the maximum principle), we   deduce that   there is a function~$g$ such that up to extraction of a subsequence, for all times~$[0,T]$,
$$g_\eps \rightharpoonup g \hbox{ weakly * in } L^\infty ([0,T];\T^2 \times \R^2\times {\mathbb S}  \times\R ) \, ,$$
as $\eps, \alpha \to 0$.
Multiplying (\ref{singular-Boltz}) by $ \eps/\alpha$ and taking limits in the sense of distributions, we get
$$ \Omega \, \partial_\Theta g = 0\,,$$
since $\eps \ll \alpha ^3$.
 This implies that~$g$ must satisfy
$$ g := g  (t,X,V,\Omega)\,.$$
%

\medskip

\noindent
\underbar{Averaged evolution  equation}

We then integrate (\ref{singular-Boltz}) with respect to $\Theta$, which provides
\begin{equation}
\label{average-Boltz}
\begin{aligned}\partial_t \int g_\eps d\Theta &+ V \cdot \nabla_X\int  g_\eps  d\Theta  = \frac1{\alpha}
\int d\Theta  \int _{[0,L_\alpha]\times {\mathbb R}^2}
   M_{\beta } (v)   \\ 
& \qquad\qquad\qquad
 \Big( g_\eps(t,X,V',\Theta,\Omega') \big(  (\frac 1 \alpha v' -V'-\Omega'    r_\Theta^\perp
) \cdot    n_\Theta\big)_- \\
&\qquad\qquad \qquad  -  g_\eps (t,X,V,\Theta,\Omega) \big(  (\frac 1 \alpha v
-V- \Omega r_\Theta^\perp
\big ) \cdot    n_\Theta \big)_-\Big)\, d  \sigma_\alpha dv  \, .
\end{aligned} \end{equation}
This is our starting point to derive the limiting equation in the limit $\eps,\alpha\to 0$.

\subsection{The~$\eps,\alpha\to 0$ limit}
\label{subsec: Taylor expansion}
      
 To investigate the joint limit $\eps,\alpha\to 0$, we   use a weak formulation of the collision operator.
Let~$\varphi=\varphi(X,V,\Omega)$ be a test function with compact support and let us compute 
\begin{align*}
F_\eps := 
  \frac1 {\alpha }  & \int   d   \sigma_\alpha dv dY
  M_{\beta } (v)  M_{\beta, I} (V,\Omega)
 \Big( g_\eps (t,X,V',\Theta,\Omega') \big(  (\frac 1 \alpha v' -V'-\Omega'  r_\Theta^\perp
) \cdot    n_\Theta\big)_- \\
&\qquad\qquad \qquad  -  g_\eps  (t,X,V,\Theta,\Omega) \big(  (\frac 1 \alpha v
-V- \Omega r_\Theta^\perp
\big ) \cdot    n_\Theta \big)_-\Big)  \varphi(X,V,\Omega) 
  \, .
\end{align*}
By a change of variables using the conservation of energy we find
\begin{align*}
F_\eps  = 
  \frac1 {\alpha }  & \int   d   \sigma_\alpha dv dY
  M_{\beta } (v)  M_{\beta, I} (V,\Omega)
g_\eps (t,X,V, \Theta,\Omega) \nonumber \\
& \qquad \qquad  \times\big ( \varphi  (X,V', \Omega') - \varphi(X,V,\Omega) \big) 
\, \big( (\frac 1 \alpha v
-V- \Omega  r_\Theta^\perp \big ) \cdot     n_\Theta\big)_- \, .
\end{align*}
Defining
\begin{equation}
\label{eq: b alpha}
b_\alpha(v,V,\Omega) := (v-\alpha V-\alpha \Omega   r_\Theta^\perp ) \cdot   n_\Theta\, 
\quad \text{and} \quad
b_{\alpha -} (v,V,\Omega) := - \inf \big \{ 0, b_\alpha(v,V,\Omega) \big\},
\end{equation}
we have
\begin{align}
F_\eps := 
  \frac1 {\alpha^2 }  & \int   d   \sigma_\alpha  dv dY
  M_{\beta } (v)  M_{\beta, I} (V,\Omega)
g_\eps (t,X,V, \Theta,\Omega)  \nonumber \\
& \qquad\qquad\qquad \qquad  \times \big ( \varphi  (X,V', \Omega') - \varphi(X,V,\Omega) \big) 
b_{\alpha-}(v,V,\Omega) \, .
\label{eq: weak form}
\end{align}
The idea is now to use   cancellations on the right-hand side.  Recalling, as stated in~(\ref{eq: collision laws}),  that the tangential part of~$V$ is constant through the scattering and
$$ 
V' \cdot {n_\Theta}=V\cdot{  n_\Theta}+\frac{2\alpha}{A+1} b_\alpha \quad \mbox{and} \quad \Omega'= \Omega  +\frac{2\alpha}{(A+1)I}     (  r_\Theta \cdot   n_\Theta^\perp) b_\alpha  \,$$
we find that
$$
\begin{aligned}
& \varphi  (X,V',\Omega') - \varphi(X,V,\Omega) \\
&\qquad = \frac{2\alpha}{A+1}b_\alpha \,  ( n_\Theta \cdot \nabla_V) \varphi (X,V,\Omega) +\frac{2\alpha}{(A+1)I} b_\alpha   (r_\Theta \cdot   n_\Theta^\perp) \partial_\Omega  \varphi (X,V,\Omega)  \\
& \qquad \quad + \frac{2\alpha^2} {(A+1)^2}  b_\alpha^2 ( n_\Theta\cdot \nabla_V)^2  \varphi (X,V,\Omega) +  \frac{2\alpha^2} {(A+1)^2I^2} b_\alpha^2  (r_\Theta \cdot   n_\Theta^\perp) ^2    \partial_\Omega ^2  \varphi (X,V,\Omega) \\
&\qquad \quad +  \frac{4\alpha^2} {(A+1)^2I}  b_\alpha^2(  r_\Theta \cdot   n_\Theta^\perp )( n_\Theta \cdot \nabla_V)     \partial_\Omega  \varphi (X,V,\Omega) +  {\color{black}
O \big( \alpha^3 b_\alpha^3 \| \varphi\|_{W^{3,\infty}} \big) } \,  .
\end{aligned}
$$
Notice that
\begin{align}
\label{eq: b alpha relation}
\frac1\alpha b_\alpha \times b_{\alpha-} &={\color{black} -\frac1\alpha  (v\cdot n_\Theta)_-^2   +2   (v\cdot n_\Theta)_-(V+\Omega r_\Theta^\perp) \cdot n_\Theta 
 + O\big( \alpha (|V|^2 + |\Omega|^2 ) \big) },
 \end{align}
and that 
\begin{align*} 
b_\alpha^2 \times b_{\alpha-}  &=  (v\cdot n_\Theta)_-^3 
 {  + O\big( \alpha (|V| + |\Omega| + |v| )^3  \big) } \,.
\end{align*}
We also note that $A = O(\alpha^2)$, so   we can neglect its contribution. We therefore can write 
$$
\begin{aligned}
&\frac1{\alpha^2} \big( \varphi  (X,V',\Omega') - \varphi(X,V,\Omega)  \big)  b_{\alpha-} (v,V,\Omega) \\
&  =\Big( -\frac2\alpha (v\cdot  n_\Theta)_-^2  + 4 (v\cdot  n_\Theta)_-(V+\Omega   r_\Theta^\perp) \cdot  n_\Theta\Big) \big(    n_\Theta \cdot \nabla_V  +  I^{-1} r_\Theta \cdot   n_\Theta^\perp     \partial_\Omega \big) \varphi (X,V,\Omega)  \\
& \quad + 2  (v\cdot n_\Theta)_-^3 \Big( ( n_\Theta \cdot \nabla_V)^2 + I^{-2} (r  \cdot n^\perp)^2\partial_\Omega^2+2  I^{-1} r_\Theta \cdot   n_\Theta^\perp ( n_\Theta \cdot \nabla_V)     \partial_\Omega\Big)   \varphi (X,V,\Omega)\\
&  \quad + O \big(\alpha\| \varphi\|_{W^{3,\infty}} ( |V| + |\Omega| +|w|)^4 \big).
\end{aligned}
$$
We  therefore find
\begin{equation}
\label{eq: 1ere etape}
\begin{aligned}
& F_\eps =   2 \int   d   \sigma_\alpha dv dY \, 
    M_\beta(v) M_{\beta, I} (V,\Omega)   \, g_\eps  (t,X,V,\Theta,\Omega)   \\
&\times\Big (-\frac1\alpha (v\cdot n_\Theta)_-^2   \big(   n_\Theta \cdot \nabla_V +   I^{-1} r_\Theta \cdot   n_\Theta^\perp     \partial_\Omega \big) \varphi (X,V,\Omega)  \\
& \qquad +2 (v\cdot n_\Theta)_-(V+\Omega  r_\Theta^\perp) \cdot n_\Theta \; \big(   n_\Theta \cdot \nabla_V+   I^{-1} r_\Theta \cdot   n_\Theta^\perp     \partial_\Omega \big) \varphi (X,V,\Omega)  \\
& \quad +    (v\cdot n_\Theta)_-^3 \Big( ( n_\Theta \cdot \nabla_V)^2 + I^{-2} (r_\Theta \cdot   n_\Theta^\perp) ^2 \partial^2_\Omega  +2I^{-1} r_\Theta \cdot   n_\Theta^\perp  ( n_\Theta \cdot \nabla_V)     \partial_\Omega\Big)   \varphi (X,V,\Omega)\Big) \, ,
\end{aligned}
\end{equation}
up to terms of order $  O(\alpha)$.
{\color{black} 
The following identities hold for any unit vector $e \in {\mathbb S}$
\begin{equation}
\label{expectations}
\begin{aligned}
\int M_\beta(v)  (v\cdot e)_\pm dv & = \left( \frac1{2\pi \beta}\right)^{1/2},
\qquad \int M_\beta (v) (v\cdot e)_\pm^2 dv & = \frac1{2\beta} \, ,\\
 \int M_\beta(v)  (v\cdot e)^3_\pm dv & = \left( \frac 2{\pi \beta^3}\right)^{1/2}  = {2\over \beta} \left( \frac1{2\pi \beta}\right)^{1/2} .
 \end{aligned}
\end{equation}
They imply that the velocity $v$ in \eqref{eq: 1ere etape} can be integrated out.
Furthermore,  the  terms in the second line of  \eqref{eq: 1ere etape}  cancel thanks to the relation 
\begin{equation}
\label{usefulcancellations}
\begin{aligned}
&\int_{0}^{L_\alpha}  n\cdot e \,  d   \sigma_\alpha  =  \int _{\Sigma_\alpha} \nabla\cdot e \, dr = 0 \, ,\\
&\int_{0}^{L_\alpha}    r^\perp \cdot n  \, d   \sigma_\alpha  
= \int _{\Sigma_\alpha} \nabla\cdot ( r^\perp)\,  dr =0 \, .
\end{aligned}
\end{equation}
Thus only the terms of order $O(1)$ remain:}
$$
\begin{aligned}
F_\eps  =2 &  \int   d   \sigma_\alpha \, dv dY \, 
  M_\beta(v) M_{\beta, I} (V,\Omega)\, g_\eps  (t,X,V,\Theta,\Omega) \\
&\times\Big ( 2 (v\cdot n_\Theta)_-(V+\Omega  r_\Theta^\perp) \cdot n_\Theta \;
\big(   n_\Theta \cdot \nabla_V  +  I^{-1} r_\Theta \cdot   n_\Theta^\perp \partial_\Omega \big) \varphi (X,V,\Omega)  \\
& \quad + (v\cdot n_\Theta)_-^3 \Big( ( n_\Theta \cdot \nabla_V)^2 + I^{-2} (r_\Theta \cdot   n_\Theta^\perp)^2 \partial^2_\Omega +2 I^{-1} r_\Theta \cdot n_\Theta^\perp   ( n_\Theta \cdot \nabla_V)\partial_\Omega \Big)   \varphi (X,V,\Omega) \\
& \quad +  {\color{black}
O \big( \alpha  \| \varphi\|_{W^{3,\infty}} \big) } Â \, .
\end{aligned}
$$
 Let us introduce the notation (as in \cite{DGL_2})
  \begin{equation}
 \label{NKT-cal}
 \begin{aligned}
 \cN_\alpha (\Theta)& :=\int_{0}^{L_\alpha}   n_\Theta \otimes n_\Theta d   \sigma_\alpha= R_\Theta \cN_\alpha R_\Theta \, ,\\
 \quad 
 \Gamma_\alpha(\Theta) & := \int_{0}^{L_\alpha}   r_\Theta \cdot   n_\Theta^\perp n_\Theta d   \sigma_\alpha = R_\Theta \Gamma_\alpha \, ,\\
 \quad 
 \cK_\alpha & := \int_{0}^{L_\alpha}(  r \cdot n^\perp  )^2  d   \sigma_\alpha \,,
 \end{aligned}
 \end{equation}
 and notice that~$ \cN_\alpha (\Theta)$ and~$\Gamma_\alpha(\Theta) $ both converge strongly,
{\color{black} when $\alpha$ tends to 0}, to
 $$
  \begin{aligned}
 \cN  (\Theta)& :=\int_{0}^{L }   n_\Theta \otimes n_\Theta d   \sigma = R_\Theta \cN  R_\Theta \, ,\\
 \quad 
 \Gamma (\Theta) & := \int_{0}^{L }   r_\Theta \cdot   n_\Theta^\perp n_\Theta d   \sigma  = R_\Theta \Gamma  \,,
 \end{aligned}
 $$  
 and~$ \cK_\alpha $ converges to
 $ 
\displaystyle \cK := \int_{0}^{L }(  r \cdot n^\perp  )^2  d   \sigma,  $  where~$L$ is the perimeter of $\Sigma$.

Using \eqref{expectations},  $F_\eps$ can be rewritten 
$$
 \begin{aligned}
 & F_\eps = \left( \frac8{\pi \beta}\right)^{1/2} \int dY\, M_{\beta, I} (V,\Omega) g_\eps (t,X,V,\Theta,\Omega) \\
 &\qquad\qquad  \times \Big( (V\cdot \cN_\alpha  (\Theta) + \Omega \Gamma_\alpha (\Theta)) \cdot \nabla_V \varphi (X,V,\Omega)
+ I^{-1} (V\cdot \Gamma_\alpha (\Theta) +\Omega \cK_\alpha) \partial_\Omega \varphi (X,V,\Omega)\\
 &\qquad \qquad \quad +\frac1\beta  \Big(\nabla_V\cdot  \cN_\alpha (\Theta)\cdot \nabla_V + I^{-2} \cK_\alpha \partial^2_\Omega + 2 I^{-1}  \partial _\Omega\nabla_V\cdot \Gamma_\alpha (\Theta)
 \Big)  \varphi (X,V,\Omega)\Big)\\
 &\qquad \qquad\qquad + O( \alpha \| \varphi\|_{W^{3,\infty}} \| g_\eps \|_{L^\infty}) \, .
 \end{aligned}
$$
Note that the  remainder  converges to 0 as $\eps,\alpha$ tend to $0$ since we have a uniform $L^\infty$ bound on $g_\eps$.

  \medskip\noindent
\underbar{Convergence to the Fokker-Planck equation}

We turn now to the joint limit $\eps,\alpha \to 0$ with $\eps \ll \alpha$.
From the weak-$\star$ convergence 
$$g_\eps  \rightharpoonup g \quad \hbox{ with } \qquad g= g(t,X,V,\Omega) $$
and the strong convergence of~$ \cN_\alpha (\Theta)$ and~$\Gamma_\alpha(\Theta) $ we immediately deduce that 
$$ 
\partial_t \int g_{\eps}d\Theta + V \cdot \nabla_X\int  g_\eps  d\Theta  \to \d_t g + V \cdot \nabla_X g  
$$
in the sense of distributions.

 For any test function $\varphi = \varphi(X,V,\Omega)$ $$
 \begin{aligned}
 F_\eps \to &  \left( \frac8{\pi \beta}\right)^{1/2} 
 \int  dX d \Omega dV \;   \, M_{\beta, I} (V,\Omega) g(t,X,V,\Omega) \\
 &\qquad\qquad  \times \Big(   \frac L2 V  \cdot \nabla_V \varphi (X,V,\Omega) +  \cK I^{-1}  \Omega \partial_\Omega \varphi (X,V,\Omega)\\
 &\qquad \qquad \quad +\frac1\beta  ( \frac L2 \Delta_V + I^{-2} \cK  \partial^2_\Omega  )  \varphi (X,V,\Omega)\Big)\,.
 \end{aligned}
 $$
Integrating with respect to $\Theta$ and then integrating by parts in $V,\Omega$, we finally get that
$$\begin{aligned}
\d_t g +V\cdot \nabla_x g = \left( \frac8{\pi \beta}\right)^{1/2}   \cL g,
\end{aligned}
$$
where the diffusion operator \eqref{eq: diffusion operator}   is given by 
\begin{align*}
\cL =   \frac1\beta  \big( \frac L2  \Delta_V +  \frac{\cK}{I^2}  \partial^2_\Omega \big)    - \frac L2 V  \cdot \nabla_V    -  \frac{\cK}{I}  \Omega \partial_\Omega  .
\end{align*}
Note indeed that $\cL$ is symmetric in the space $L^2 (M_{\beta, I} (V,\Omega)dVd\Omega)$.
This concludes the proof of Theorem \ref{thm: convergence density}.  \qed

\section{Convergence to the Ornstein-Uhlenbeck process}
\label{process-section}

We are going to study the path fluctuations and  derive Theorem  \ref{thm: convergence process}.
Throughout this section, the limit $N \to \infty$ refers to the joint limit $N \to \infty, \eps \to 0$ and $\alpha \to 0$ in the  Boltzmann-Grad scaling $N \eps =1$ with $\alpha \gg \left( \frac{1}{\log \log N}\right)^{1/4} $.

%
\medskip

To prove the convergence of the  process $\brown$ 
to  $\cW$, we will proceed as in \cite{BGSR1}
and check 
\begin{itemize}
\item the convergence of the time marginals  for any $0\leq \tau_1 < \dots < \tau_\ell 	\leq T$
\begin{equation}
\label{eq: finite dim x1}
\lim_{N \to \infty} \mathbb{E}_{M_{\beta,I, N}}  \Big(  h_1 \big( \brown ( \tau_1) \big) \dots  
h_\ell \big(  \brown (  \tau_\ell) \big)
 \Big) 
=
\mathbb{E}  \Big(  h_1 \big(  \cW( \tau_1) \big) \dots  h_\ell \big(  \cW ( \tau_\ell) \big)  \Big) \, ,
\end{equation}
where $\{ h_1, \dots, h_\ell \}$ is a collection of continuous functions in ${\mathbb R}^2 \times \mathbb R$.
\item the tightness of the sequence, that is  
\begin{equation}
\label{eq: tightness criterion}
\forall \petit >0, \qquad 
\lim_{\eta \to 0} \lim_{N \to \infty}
\mathbb P_{M_{\beta,I, N}} \left( \sup_{ |\sigma - \tau| \leq  \eta \atop \tau \in [0,T]}  \big| \brown(\sigma) - \brown(\tau)  \big| \geq \petit 
\right) = 0 \, .
\end{equation}
\end{itemize}	
The linear Boltzmann equation \eqref{eq: linear boltz} is associated with 
a stochastic process  $\bar \brown( t)$ which converges to $\cW$.
Thus it is enough to prove that the mechanical process $\brown$ is close to the stochastic process $\bar \brown$
for an appropriate coupling.

\medskip

{\color{black} 

By construction, the position is given by 
$$
X(t) = \int_0^t V(s) ds\,,
$$
thus we deduce from the convergence in law of the velocity that
in the limit, the process is a Langevin process \eqref{eq: OU process}.
}

\subsection{Auxiliary Markov process}
\label{subsec: Auxiliary Markov process}

We first define the stochastic process associated with the linear Boltzmann equation \eqref{eq: linear boltz}.
The Markov process  $\bar Y( t) = ( \bar X(t), \bar \Theta(t), \bar V (t) , \bar \Omega(t) )$ is characterized by the generator 
\begin{align}
\label{eq: generateur L}
 \cT_\alpha \varphi(    Y) &= - 
V \cdot \nabla_X \varphi (    Y)
- \frac{\alpha}{\eps} \Omega \partial_\Theta \varphi (    Y) 
 \\
&  
+ \frac1{\alpha}
\int _{ [0, L_\alpha] \times {\mathbb R}^2}  M_{\beta} (v)  
 \Big( \varphi(    Y' ) - \varphi(    Y) \Big)
  \big(  ({v \over \alpha} - V - \Omega R_\Theta   r^\perp
\big ) \cdot R_\Theta   n \big)_- \, d \sigma_\alpha dv  
\nonumber \\
&  
=- 
V \cdot \nabla_X \varphi (    Y)
- \frac{\alpha}{\eps} \Omega \partial_\Theta \varphi (    Y)  + \cL_\alpha \varphi(    Y),
\nonumber
\end{align}
with $Y = ( X, \Theta, V, \Omega)$ and $Y' = ( X, \Theta, V', \Omega')$.
The dependency of the process $\bar Y$ on $\alpha$ and~$\eps$ is omitted in the notation.
Note that the dependency in $\bar \Theta$ will average out in the limit, 
so that asymptotically it is enough to consider the process  $\bar \brown( t) = (\bar V (t) , \bar \Omega(t) )$.
The~$(\bar X, \bar \Theta)$ dependency has been kept for later purposes, but it does not influence the evolution of~$\bar \brown$
as the stochastic dynamics model a rigid body in a uniformly distributed ideal gas.
It is proved in Lemma \ref{lem: markov simple}
 that the invariant measure associated with the process~$\bar Y$ is given by 
$$
\bar M_{\beta,I} ( X, \Theta, V, \Omega) = 
\frac{1}{2 \pi} M_{\beta,I} (V, \Omega).
$$ 
Note that the position $X$ and the angle $\Theta$ are uniformly distributed in 
${\mathbb T}^2 \times {\mathbb S}$ under $\bar M_{\beta,I}$.

\begin{lem}
\label{lem: markov convergence}
Fix $T>0$ and consider the Markov chain $\bar Y$ on $[0,T]$ starting from $\bar M_{\beta,I}$.
Then the stochastic process $\bar \brown( t) = ( \bar V (t) , \bar \Omega(t) )$ converges in law to $\cW$ in $[0,T]$ in the joint limit $\eps \to 0$ and $\alpha \to 0$ with $ \alpha^4 \, \log |\log\eps| \gg 1 $.
\end{lem}
This is the analogue to the convergence Theorem \ref{thm: convergence process} for the process  $\bar Y$. 
The proof relies on the martingale approach which is standard to establish the convergence of stochastic processes (see \cite{Durrett}). 
A similar convergence was derived in \cite{DGL_2} but in our case the fast rotation leads to some degeneracy, thus we sketch the proof below for convenience.

\begin{proof}
The limiting diffusion $\cW$ can be identified as the unique solution of the martingale problem, i.e.
for any  test function $\varphi$ in $C^2 ( {\mathbb R}^2 \times \mathbb R)$
$$
\varphi \big( \cW(t) \big)  -  a \int_0^t \cL \varphi \big( \cW(s) \big) ds 
$$
is a martingale and the generator $\cL$ was introduced in \eqref{eq: diffusion operator} 
\begin{align*}
\cL =   \frac1\beta  \big({ L \over 2}  \Delta_V +  \frac{\cK}{I^2}  \partial^2_\Omega \big)    
-  { L \over 2} V  \cdot \nabla_V    -  \frac{\cK}{I}  \Omega \partial_\Omega 
\quad \text{with} \quad 
a = \left( \frac{8}{\pi \beta}\right)^{1/2}.
\end{align*}
To prove the convergence, we will first show that the distributions of the trajectories of $\bar \brown$ are tight in the 
Skorokhod space $D([0,T])$, then we identify the limiting distribution as the unique solution of the martingale problem.

\medskip

\noindent
{\bf Step 1. Tightness.}

From Aldous' criterion (see  \cite{billingsley}, Theorem 16.10), 
the tightness of the sequence \eqref{eq: tightness criterion} boils down to proving the following assertion
\begin{equation}
\label{eq: aldous tightness criterion Y}
\forall \petit >0, \qquad 
\lim_{\eta \to 0} \limsup_{\alpha,\eps \to 0} \sup_{ \mathcal T \atop 
0 < u < \eta}  
\mathbb P_{\bar M_{\beta,I}} \left( \big| \bar \brown(u + \mathcal T) - 
\bar \brown(\mathcal T)  \big| \geq \petit 
\right) = 0 \, ,
\end{equation}
where the supremum is taken over any stopping time $\mathcal T$  in $[0,T]$ and by abuse of notation~$u + \mathcal T$ stands for  $\inf \{ u + \mathcal T, T \}$. The stopping times are measurable with respect 
to the filtration associated with the random kicks.

\medskip

Since $\bar Y$ is a Markov process with generator $\cT_\alpha$ defined in \eqref{eq: generateur L}, we know that 
\begin{equation}
\label{eq: martingale bar Y}
\cM(t) = \varphi \big( \bar Y(t) \big) - \varphi \big( \bar Y(0) \big)  - \int_0^t \cT_\alpha \varphi \big( \bar Y (s) \big) ds
\end{equation}
is a martingale for any test function  $\varphi$.  We start by estimating the fluctuations of $\bar V$ and 
applying identity \eqref{eq: martingale bar Y} to $\varphi(Y) = V$, we deduce from \eqref{eq: collision laws} that 
\begin{align*}
\cM(t) & = \bar V(t) - \bar V(0)  
-  \frac1{\alpha^2} \int_0^t ds \int _{[0,L_\alpha] \times {\mathbb R}^2}  M_{\beta} (v)  
 \big( \bar V'(s) -  \bar V(s) \big)  b_{\alpha -}
\, d \sigma_\alpha  dv\\
& = \bar V(t) - \bar V(0)  
-  \int_0^t ds \int _{[0,L_\alpha] \times {\mathbb R}^2}  M_{\beta} (v)  
 n_{\bar \Theta(s)} \, \frac1{\alpha} {2\over A+1}  b_\alpha \times  b_{\alpha -}
 \, d \sigma_\alpha  dv   \, .
\end{align*}
 From relation \eqref{eq: b alpha relation}
\begin{align*}
\frac1\alpha b_\alpha \times b_{\alpha -} 
= \frac1\alpha  (v\cdot n_{\bar \Theta})_-^2   - 2   (v\cdot n_{\bar \Theta})_-(\bar V+\bar \Omega r_{\bar \Theta}^\perp) \cdot n_{\bar \Theta} 
 + O\big( \alpha (|\bar V|^2 + |\bar \Omega|^2 ) \big)     \, ,
\end{align*}
we get  
\begin{align}
 \bar V(t) - \bar V(0)  = & \cM(t)  \nonumber \\
&  -
 4    \int_0^t ds \int _{[0,L_\alpha] \times {\mathbb R}^2}  M_{\beta} (v)  
 n_{\bar \Theta(s)} \,  (v\cdot n_{\bar \Theta(s)})_-( \bar V (s) + \bar \Omega (s) r_{\bar \Theta(s)}^\perp) \cdot n_{\bar \Theta(s)}  \, d \sigma_\alpha  dv
 \label{eq: last terms} \\
 & 
 + \alpha \int_0^t ds \int _{[0,L_\alpha] \times {\mathbb R}^2}  M_{\beta} (v) 
 O\big(   | \bar V (s)|^2 + | \bar \Omega (s)|^2  \big) \, d \sigma_\alpha  dv \, , \nonumber
\end{align}
where we integrated in $v$  the first  term $ (v\cdot n_\Theta)_-^2$
and the corrections from the factor $A \simeq \alpha^2$ have been added to the error term.
We finally obtain for any $\petit>0$
\begin{align}
& \mathbb P_{\bar M_{\beta,I}} \left(  \big| \bar  V(\mathcal T + u) - \bar V(\mathcal T )  \big| \geq \petit 
\right) \nonumber \\
\label{eq: fast rotation term}
& \qquad \leq  
\mathbb P_{\bar M_{\beta,I}} \left(  \big| \cM (\mathcal T + u) - \cM (\mathcal T )  \big| \geq \petit /2 \right)
\\
& \qquad + 
\mathbb P_{\bar M_{\beta,I}} \left(   \int_{\mathcal T}^{\mathcal T +u} ds  O\big( |\bar V (s) | + | \bar V (s)|^2 + |\bar \Omega (s)| + |\bar \Omega (s)|^2 )  \geq \petit /2 \right), \nonumber
\end{align}
where the last probability is an upper bound on the fluctuations of the last two terms in~\eqref{eq: last terms}.
This probability can be easily bounded by using a Chebyshev estimate and the time invariance of the measure $\bar M_{\beta,I}$. We treat only one term for simplicity.  Let $C$ be a  large enough constant and choose $\eta \leq  \petit / {(2 C)}$. As $u \leq \eta$, we get 
\begin{align}
\label{eq: borne integrale temps}
\mathbb P_{\bar M_{\beta,I}} \left(   \int_{\mathcal T}^{\mathcal T +u} 
ds  |\bar V (s) | \geq \petit  \right)
& \leq 
\mathbb P_{\bar M_{\beta,I}} \left(   \int_0^T ds  |\bar V (s) | 1_{ \{ |\bar V (s) | \geq C\}} \geq \frac{\petit}{2}  \right)\\
& \leq \frac{2 T}{\petit} \mathbb E_{\bar M_{\beta,I}} \left(   | \bar V|   1_{ \{ |\bar V  | \geq C\}} \right).
\nonumber
\end{align}
The last term vanishes when $C$ tends to infinity so that the probability also vanishes when $\eta$ tends to 0 (for any given $\petit >0$).

\medskip

Finally, we will prove that 
\begin{equation*}
\lim_{\eta \to 0} \limsup_{\alpha,\eps \to 0} 
\sup_{ \mathcal T \atop 0 < u < \eta}   
\mathbb P_{\bar M_{\beta,I}} \left(  \big| \cM (\mathcal T + u) - \cM (\mathcal T )  \big| \geq \petit \right) = 0 \, .
\end{equation*}
By the Chebyshev estimate and the martingale property, we get 
\begin{align}
\mathbb P_{\bar M_{\beta,I}} \left(  \big| \cM (\mathcal T + u) - \cM (\mathcal T )  \big| \geq \petit \right) 
& \leq \frac{1}{\petit^2}
\mathbb E_{\bar M_{\beta,I}} \Big( \big( \cM (\mathcal T + u) - \cM (\mathcal T ) \big)^2  \Big) \nonumber\\
& \leq \frac{1}{\petit^2}
\mathbb E_{\bar M_{\beta,I}} \Big(  \cM (\mathcal T + u)^2 - \cM (\mathcal T )^2  \Big) \, .
\label{eq: carre martingale}
\end{align}
For the martingale $\cM$ defined by \eqref{eq: martingale bar Y}, we know that  
\begin{equation*}
t \mapsto
\cM(t)^2 -  \int_0^t \Big[ \cT_\alpha \varphi^2  - 2 \varphi \cT_\alpha \varphi  \Big]\big( \bar Y (s) \big) ds
\end{equation*}
is also a martingale and furthermore 
\begin{equation*}
[ \cT_\alpha \varphi^2  - 2 \varphi \cT_\alpha \varphi ] (Y)
= \frac1{\alpha}
\int _{[0,L_\alpha] \times {\mathbb R}^2}  M_{\beta} (v)  
 \Big( \varphi( \bar  Y' ) 
 -  \varphi( \bar  Y) \Big)^2 \big(  ( \frac{v}{\alpha} - V - \Omega R_\Theta\bar  r^\perp
\big ) \cdot R_\Theta\bar  n \big)_- \, d  \sigma_\alpha  dv  \,  .
\end{equation*}
Thus inequality \eqref{eq: carre martingale} can be rewritten (with $\varphi(Y) = V$) 
\begin{align*}
& \mathbb P_{\bar M_{\beta,I}} \left(  \big| \cM (\mathcal T + u) - \cM (\mathcal T )  \big| \geq \petit \right) \\
& \qquad \leq \frac{1}{\petit^2 \alpha^2}
\mathbb E_{\bar M_{\beta,I}} \Big( \int_{\mathcal T}^{\mathcal T + u} ds
\int _{[0,L_\alpha] \times {\mathbb R}^2}  M_{\beta} (v)  
 \Big(  V'(s)   -  V(s) \Big)^2 b_{\alpha -}(v, \bar V (s), \bar \Omega(s) ) \, d  \sigma_\alpha  dv \Big)\\
& \qquad \leq \frac{1}{\petit^2 }
\mathbb E_{\bar M_{\beta,I}} \Big( \int_{\mathcal T}^{\mathcal T + u} ds
\int _{[0,L_\alpha] \times {\mathbb R}^2}  M_{\beta} (v)  
  b_{\alpha}^2 \times b_{\alpha -}(v, \bar V (s), \bar \Omega(s) ) \, d  \sigma_\alpha  dv \Big) 
  \,  .
\end{align*}
This last term can be estimated as in \eqref{eq: borne integrale temps}.

\medskip

The derivation of \eqref{eq: aldous tightness criterion Y} can be completed by following the 
same proof  to control the fluctuations of $\bar \Omega$.

\bigskip

\noindent
{\bf Step 2. Martingale problem.}

Consider  a collection of times $0 \leq \tau_1 < \dots < \tau_\ell \leq s < t$ in $[0,T]$   and a collection   of continuous functions~$\{ h_1, \dots, h_\ell \}$ in ${\mathbb R}^2 \times \mathbb R$.
For any $\alpha, \eps$, the martingale relation \eqref{eq: martingale bar Y} implies that 
for any $\varphi$ in $C^2 ( {\mathbb R}^2 \times \mathbb R)$
\begin{equation}
\mathbb{E}_{\bar M_{\beta,I}}
\Big(  h_1 \big( \bar \brown ( \tau_1) \big) \dots  h_\ell \big(  \bar \brown (  \tau_\ell) \big)
\big( 
\varphi \big( \bar \brown (t) \big) - \varphi \big( \bar \brown (s) \big)  - \int_s^t \cL_\alpha \varphi \big( \bar \brown (u) \big) du
\big)
 \Big) 
= 0 \, .
\end{equation}
 A Taylor expansion at   second order  as in Section \ref{subsec: Taylor expansion}   implies that 
any limiting distribution~$\mathbb{E}_{\bar M_{\beta,I}}^*$ will be a solution of the martingale problem 
as
\begin{equation}
\mathbb{E}_{\bar M_{\beta,I}}^*
\Big(  h_1 \big( \bar \brown ( \tau_1) \big) \dots  h_\ell \big(  \bar \brown (  \tau_\ell) \big)
\big( 
\varphi \big( \bar \brown (t) \big) - \varphi \big( \bar \brown (s) \big)  
- \int_s^t \cL \varphi \big( \bar \brown (u) \big) du
\big)
 \Big) 
= 0 \, .
\end{equation}
To derive the limit above, one can consider the process $\bar Y$ starting at time $s$ from the 
weighted measure $\hat g_\eps (s)$ defined for any test function $\Psi$ by 
$$
\int \hat g_\eps (s,Y) \Psi (Y) dY 
=
\mathbb{E}_{\bar M_{\beta,I}}^*
\Big(  h_1 \big( \bar \brown ( \tau_1) \big) \dots  h_\ell \big(  \bar \brown (  \tau_\ell) \big) \Psi ( \bar Y(s)) \Big) \,  ,
$$
and then apply the arguments of Section \ref{subsec: Taylor expansion}.
This completes the proof of Lemma \ref{lem: markov convergence} for the convergence in law of the process  to $\cW$.
\end{proof}

\subsection{Comparison with the truncated process}

The coupling between the mechanical process $\brown$ and the stochastic process $\bar \brown$ will be achieved indirectly by considering a coupling between the truncated processes which are defined next.

We first introduce $\bar Y^\dagger(t)$ the analogue of the process $\bar Y$ with an additional killing term : 
the truncated stochastic process $\bar Y^\dagger(t) = ( \bar X^\dagger (t), \bar \Theta^\dagger (t), \bar V^\dagger (t) , \bar \Omega^\dagger (t) )$ is defined by the generator
\begin{align}
\label{eq: generateur non symmetrique}
\cT_\alpha^\dagger g
 :=
 V \cdot \nabla_X g_\alpha +  \frac{\alpha}{\eps} \Omega \partial_\Theta g&+
\frac1\alpha \int _{[0,L_{ \alpha}] \times {\mathbb R}^2}  M_{\beta } (v)   \indc_{(\ref{good-collision3})}
\big  ( \indc_{(\ref{good-collision1})(\ref{good-collision2}) } g(Y')   -  g(Y)\big)     \\
& \qquad \times
\big(  (\frac1{\alpha}v -V- \Omega    r_{ \Theta}^\perp
\big ) \cdot    n_{ \Theta} \big)_- \, d  \sigma_\alpha dv \,. \nonumber
\end{align}
We set as well $\bar \brown^\dagger(t) = (\bar V^\dagger (t), \bar \Omega^\dagger (t))$.
The rigid body in the mechanical process with a killing term, as introduced in \eqref{eq: killed bc}, 
will be denoted by $Y^\dagger(t) = ( X^\dagger (t), \Theta^\dagger (t), V^\dagger (t), \Omega^\dagger (t) )$.
We define also $\brown^\dagger =( V^\dagger, \Omega^\dagger)$.

Fix $T>0$.
Using Proposition \ref{lem: small set}, the processes $Y$ and $Y^\dagger$ can be coupled with high probability
\begin{equation}
\label{eq: comparison processes Y}
\lim_{\eps,\alpha \to 0}  {\mathbb P}_{M_{\beta, I,N}} \Big( \exists t \leq T, \quad  Y( t) \not =  Y^\dagger ( t) \Big)
= 0 \,.
\end{equation}
In the same way, as shown in Lemma \ref{lem: markov tronque}, the processes $\bar Y$ and $\bar Y^\dagger$ 
coincide asymptotically, on the time interval $[0,T]$,  when $\alpha$ tends to 0. 
Thus Lemma \ref{lem: markov convergence} implies  that the time marginals of $\bar \brown^\dagger$ converge to those of a brownian motion 
\begin{equation}
\label{eq: convergence dagger}
\lim_{\eps, \alpha \to 0} \mathbb{E}_{\bar M_{\beta,I}}  \Big(  h_1 \big( \bar \brown^\dagger ( \tau_1) \big) \dots  
h_\ell \big( \bar \brown^\dagger (  \tau_\ell) \big)
 \Big) 
=
\mathbb{E}  \Big(  h_1 \big(  \cW( \tau_1) \big) \dots  h_\ell \big(  \cW ( \tau_\ell) \big)  \Big)\, ,
\end{equation}
The tightness criterion \eqref{eq: tightness criterion} holds also for $\bar \brown^\dagger$.

\medskip

As a consequence of the previous results, it will be enough to compare the laws of $\brown^\dagger$ and~$\bar \brown^\dagger$ in order to complete the derivation of \eqref{eq: finite dim x1} and \eqref{eq: tightness criterion}.



\subsection{Convergence of the time marginals}

In this section, we are going to show that 
\begin{equation}
\label{eq: finite dim esperance}
\lim_{N \to \infty} 
\mathbb{E}_{\bar M_{\beta,I}}  \Big(   h_1 \big( \bar \brown^\dagger ( \tau_1) \big) \dots  
h_\ell \big( \bar \brown^\dagger (  \tau_\ell) \big) 
 \Big) 
 - 
 \mathbb{E}_{M_{\beta,I,N}}  \Big(   h_1 \big(  \brown^\dagger ( \tau_1) \big) \dots  
h_\ell \big(  \brown^\dagger (  \tau_\ell) \big) \Big) 
= 0 \,.
\end{equation}
To do this,  we will follow the same argument as in \cite{BGSR1} and reduce the 
limit \eqref{eq: finite dim esperance}  to a comparison of the BBGKY and Boltzmann hierarchies.

\medskip

\noindent
{\bf Step 1: Time marginals and iterated Duhamel formula.}

The density  of the Markov process $\bar Y^\dagger$ is given by 
$\tilde f^{(1)}_{\eps} = \bar M_{\beta,I} \, \tilde g_\eps$ 
where  $\tilde g_\eps$ follows the linear Boltzmann equation \eqref{trunc-boltz}.
More generally,  the evolution of the killed Markov process is related to the Boltzmann hierarchy
\eqref{eq: modified hierarchy} 
starting from the initial data \eqref{dataBoltz}
\begin{equation*}
\forall s \geq 1 \,  \quad \tilde  f_0^{(s)} (Y,Z_{s-1}) 
:= g_0(Y) \bar M_{\beta, I} (Y) \prod_{i =1}^{s-1} M_{\beta } (v_i)\,.
\end{equation*}
In particular, a representation similar to \eqref{eq: fHk 2} holds for the Markov process $\bar Y^\dagger (t)$
\begin{align}
\mathbb{E}_{\bar M_{\beta,I}}  
\Big(  h_1 \big(\bar \brown^\dagger  (\tau_1) \big) \dots  h_\ell \big( \bar \brown^\dagger (\tau_\ell) \big) \Big) 
=
\int d  \bar Y \;  h_\ell \big(  V,  \Omega \big) \tilde f_{\eps, H_\ell}^{(1)} (\tau_\ell , \bar  Y)\,,
\label{eq: Boltz Hk}
\end{align}
where $\bar  Y = (\bar   X, \bar  \Theta, \bar   V , \bar  \Omega)$ stands for the position of the Markov process at time $t_\ell$
and the modified distribution can be rewritten in terms of Duhamel series
as in \eqref{eq: representation itere + poids}
\begin{align}
\label{eq: Boltz Hk 0}
 \tilde f_{\eps, H_\ell}^{(1)} (\tau_\ell ,Y) & =  \sum_{m_1 + \dots+ m_{\ell-1} =0}^{\infty} 
\bar Q_{1,1+ m_1}^\dagger (\tau_\ell - \tau_{\ell-1} ) 
  h_{\ell-1} \bar Q_{1+ m_1,1+m_2}^\dagger (\tau_\ell - \tau_{\ell-1} )   h_{\ell -2} \dots   \nonumber  \\
& \qquad \qquad 
\bar Q_{1+ m_1+\dots + m_{\ell-2} ,1+ m_1+ \dots + m_{\ell-1}}^\dagger (\tau_1)  
\tilde f^{(1 + m_1+ \dots + m_{\ell-1})}_0\, .
\end{align} 
Many cancelations occur in the  series above and  the only relevant collision trees are made of a single  backbone 
formed by the pseudo-trajectory associated with the Markov process with a new branch at each deflection.

\medskip

\noindent
{\bf Step 2. Comparison of the finite dimensional marginals.} 

We complete now the derivation of \eqref{eq: finite dim esperance} by comparing term by term the series \eqref{eq: fHk 2} and \eqref{eq: Boltz Hk 0}.
Suppose now that the collection of functions $h_i$ are bounded.
Thanks to this uniform bound on the weights of the collision trees,
the pruning procedure applies also in this case and enables us to restrict to trees with at most $2^{K+1}$ collisions during the time interval~$[0,T]$.
Furthermore, for collision trees of size less than $2^{K+1}$, the arguments of Section \ref{subsec: Control of recollisions} apply and   recollisions can be neglected.
When no recollision occurs,  the pseudo-trajectories associated with  $\widetilde f^{(1)}_{N+1,H_\ell}$ and $\widetilde  f_{\eps, H_\ell}^{(1)}$ are close to each other and in particular the pseudo trajectory of the rigid body coincides with the one associated with the Markov process.
Thus the weights $\prod_{i = 1}^\ell h_i \big( \brown^\dagger (\tau_i) \big)$ 
and $\prod_{i = 1}^\ell h_i \big( \bar \brown^\dagger (\tau_i) \big)$
are identical and the series~\eqref{eq: fHk 2} and~\eqref{eq: Boltz Hk 0} can be compared in the same way as in 
\eqref{mainpart-est}
$$
\lim_{N \to \infty} \big\| \widetilde  f_{N+1, H_\ell}^{(1)}  (\tau_\ell) -   
\widetilde f_{\eps, H_\ell}^{(1)} (\tau_\ell ) 
\big\|_{L^\infty([0,T];L^1(\T^2\times \R^2\times {\mathbb S}\times \R))} = 0\,.
$$
This completes the proof of  \eqref{eq: finite dim esperance}.

\subsection{Tightness of the process}

We have already derived the counterpart of \eqref{eq: tightness criterion} for the limiting process.
Indeed, Aldous' criterion \eqref{eq: aldous tightness criterion Y} (see  \cite{billingsley}, Theorem 16.10) implies that  for any~$\petit >0$
\begin{equation*}
\label{eq: aldous tightness criterion bar Y}
\lim_{\eta \to 0} \limsup_{\alpha,\eps \to 0} \mathbb P_{\bar M_{\beta,I}} \left( 
\sup_{ |\tau - \sigma |\leq \eta \atop \tau \in [0,T]}   
\big| \bar \brown^\dagger (\tau) -  \bar \brown^\dagger (\sigma)  \big| \geq \petit \right) 
= 0\,.
\end{equation*}
Using Proposition \ref{prop: identification proba Duhamel}, the probability of the above event can be rewritten in terms of Duhamel series.
Comparing both hierarchies, we deduce \eqref{eq: tightness criterion} for the deterministic dynamics from 
\eqref{eq: aldous tightness criterion bar Y}.

\medskip

Theorem  \ref{thm: convergence process} is proved.\qed

\section{Conclusion and open problem}

As already mentioned, the main mathematical novelty in this paper is the control of pathological dynamics by computing directly the probability of these events under the invariant measure. 
We hope that this strategy will be useful in other situations, for instance to control multiple recollisions in the linearized setting (see~\cite{BGSR2}). 
This technique also gives a correspondence between real trajectories of tagged particles and pseudo-trajectories coming from the iterated Duhamel formula, which could be used to track the correlations, and maybe to obtain stronger convergence results (such as entropic convergence).

\bigskip
From the physical point of view, the main flaws of our study is that the system is two dimensional, and that the rigid body, although bigger than the atom, cannot be macroscopic.

In 3 dimensions, the dynamics is much more complicated since the rotation of the rigid body has two degrees of freedom and the matrix of inertia is not constant (it oscillates with the rotation). On the other hand,   the moment of inertia scales as $(\eps/\alpha)^2$, so there is as in the 2D case a very  fast rotation. Because of this fast rotation, we have an averaging effect and the rigid body behaves actually as its spherical envelope. In the limit, we completely lose track of the rotation and of the geometry.
This singular regime does not  exists if the  size of the rigid body does not vanish with $\eps$ 
(see \cite{DGL_2}). Ideally we would like to deal with a macroscopic convex body. Then we expect that it will undergo typically $O(N)$ collisions per unit of time, which corresponds to a mean field regime. In particular, we expect a macroscopic part of the atoms to see the rigid body, so that recollisions - as defined in the present paper - will occur with a non negligible probability. We have therefore to take into account the fact that the dynamics of the rigid body depends weakly on the atoms (deflections are infinitesimal), and to have a different treatment for the atom-rigid body interactions and for the atom-atom collisions.

\appendix

\section{A priori  and stability estimates for the Boltzmann equation}
\label{proofpropproximity}

\subsection{Symmetries of the collision operator  and maximum principle}

The linear Boltzmann equation
\begin{equation}
\label{eq: boltz lineaire}
\begin{aligned}
 \partial_t g_\eps  & =  \cT_\alpha^\star g_\eps  \\
\cT_\alpha^\star g & :=
- V \cdot \nabla_X g  -  \frac{\alpha}{\eps} \Omega \partial_\Theta g  \\
& \quad +
\frac1\alpha \int _{[0,L_{ \alpha}] \times {\mathbb R}^2}  M_{\beta } (v)  
 \Big( g(Y') \big(  (\frac1{\alpha}v' -V'- \Omega'      r_{ \Theta}^\perp
) \cdot    n_{ \Theta}\big)_-    \\
  & \qquad \qquad \qquad \qquad \qquad -  g(Y) \big(  (\frac1{\alpha}v
-V- \Omega    r_{ \Theta}^\perp
\big ) \cdot    n_{ \Theta} \big)_-\Big)\, d  \sigma_\alpha dv
\end{aligned}
\end{equation}
can be interpreted as the evolution of the density of the Markov process
defined in \eqref{eq: generateur L}.
Note that, unlike in the usual Boltzmann equation for hard spheres, the collision operator describing the interaction with a non symmetric rigid body is not self-adjoint.

\begin{lem}
\label{lem: markov simple}
The adjoint of the operator $\cT_\alpha^\star$ with respect to the measure $\bar M_{\beta, I} = \frac{1}{2\pi} M_{\beta, I}$
is given by 
\begin{align}
\label{eq: generateur symmetrique}
\cT_\alpha g
& :=
 V \cdot \nabla_X g  +  \frac{\alpha}{\eps} \Omega \partial_\Theta g \\
& \quad +
\frac1\alpha \int _{[0,L_{ \alpha}] \times {\mathbb R}^2}  M_{\beta } (v)   
\big  ( g(Y')   -  g(Y)\big)    \big(  (\frac1{\alpha}v
-V- \Omega    r_{ \Theta}^\perp
\big ) \cdot    n_{ \Theta} \big)_- \, d  \sigma_\alpha dv \,. \nonumber
\end{align}
It is therefore associated with the Markov process  $\bar Y( t) = ( \bar X(t), \bar \Theta(t), \bar V (t) , \bar \Omega(t) )$  introduced in \eqref{eq: generateur L}.
The measure $\bar M_{\beta, I}$ is invariant for this process.
\end{lem}

\begin{proof}
For a given $n_{ \Theta}$, the map $\Gamma (V, \Omega, v) := (V', \Omega', v')$ is an involution, thus 
using the change of variable $\Gamma^{-1}$, we deduce that for any function $h$
\begin{equation}
\begin{aligned}
\label{eq: generateur changement var}
& \int _{[0,L_{ \alpha}] \times {\mathbb R}^5}  M_{\beta } (v) M_{\beta,I } (V, \Omega)   
 g(Y')   h (Y)    \big(  (\frac1{\alpha}v' - V' - \Omega'    r_{ \Theta}^\perp
\big ) \cdot    n_{ \Theta} \big)_- \, d  \sigma_\alpha dv d V d \Omega   \\
&   \,
= \int _{[0,L_{ \alpha}] \times {\mathbb R}^5}  M_{\beta } (v')   M_{\beta,I } (V', \Omega') 
 g(Y)   h (Y')    \big(  (\frac1{\alpha}v - V - \Omega    r_{ \Theta}^\perp
\big ) \cdot    n_{ \Theta} \big)_- \, d  \sigma_\alpha dv' d V' d \Omega'\\
&   \,
= \int _{[0,L_{ \alpha}] \times {\mathbb R}^5}  M_{\beta } (v)   M_{\beta,I } (V, \Omega) 
 g(Y)   h (Y')    \big(  (\frac1{\alpha}v
-V- \Omega    r_{ \Theta}^\perp
\big ) \cdot    n_{ \Theta} \big)_- \, d  \sigma_\alpha dv d V d \Omega\,,  
\end{aligned}
\end{equation}
where we used that the kinetic energy  is conserved by the elastic collisions and therefore the Maxwellian is preserved.
This identity implies that  the   adjoint collision operator with respect to $M_{\beta,I}$
has the form  \eqref{eq: generateur symmetrique}.
As a consequence, the measure  $\bar M_{\beta,I}$ is invariant.
This proves Lemma~\ref{lem: markov simple}.
\end{proof}

\subsection{Convergence of the truncated Boltzmann equation: proof of Proposition~\ref{prop: truncatedBoltztoBoltz}}

We now consider a solution of the truncated Boltzmann equation
\begin{equation}
\label{trunc-boltz}
\begin{aligned}
\partial_t \tilde g_\eps  &  =  \cT_\alpha ^{\dagger, \star} \tilde g_\eps \\
\cT_\alpha^{\dagger, \star}  g 
&:=
- V \cdot \nabla_X g -  \frac{\alpha}{\eps} \Omega \partial_\Theta g\\
& + \frac1\alpha \int _{[0,L_{ \alpha}] \times {\mathbb R}^2}  M_{\beta } (v)  
 \Big(  \indc_{(\ref{good-collision1})(\ref{good-collision2})}g(Y') \big(  (\frac1{\alpha}v' -V'- \Omega'      r_{ \Theta}^\perp
) \cdot    n_{ \Theta}\big)_-    \\
  &\qquad \qquad \qquad -   \indc_{(\ref{good-collision3})}g(Y) \big(  (\frac1{\alpha}v
-V- \Omega    r_{ \Theta}^\perp
\big ) \cdot    n_{ \Theta} \big)_-\Big)\, d  \sigma_\alpha dv\,.
\end{aligned}
\end{equation}
This equation describes the density evolution of the Markov process introduced in \eqref{eq: generateur non symmetrique}.

\begin{lem}
\label{lem: markov tronque}
The operator $\cT_\alpha^\dagger$ defined in \eqref{eq: generateur non symmetrique} is the adjoint 
of $\cT_\alpha ^{\dagger, \star}$.
The killed process~$\bar Y^\dagger$ defined by the generator $\cT_\alpha^\dagger$ 
hardly differs from the process $\bar Y( t) = ( \bar X(t), \bar \Theta(t), \bar V (t) , \bar \Omega(t) )$.
In particular, on any time interval $[0,T]$, both processes can be coupled with large probability so that 
\begin{equation}
\label{eq: comparison processes}
\hat {\mathbb P} \Big( \exists t \leq T, \quad \bar Y( t) \not = \bar Y^\dagger ( t) \Big)
\leq C \alpha^{2 \eta}\,,
\end{equation}
where $\hat {\mathbb P}$ stands for the joint measure of the coupled processes. 
In this coupling, both processes start from the same initial data sampled from the measure $\bar M_{\beta, I}$. 
\end{lem}
This lemma is the counterpart of Proposition \ref{lem: small set} which allowed to compare the gas dynamics to the killed  dynamics. The strategy is identical, indeed both trajectories coincide if none 
of the events (\ref{good-collision1}), (\ref{good-collision2}) or (\ref{good-collision3}) is encountered, 
 thus it is enough to evaluate the probability of the occurence of any of these events.

\begin{proof}
{\color{black}

To compute the adjoint \eqref{eq: generateur non symmetrique}, we first use the same change of variables
as in \eqref{eq: generateur changement var}.
We also use the fact that conditions (\ref{good-collision1}), (\ref{good-collision2}) 
are symmetric with respect to the variables~$(v,V)$ and $(v',V')$ so that 
the adjoint  reads
\begin{equation*}
\begin{aligned}
\cT_\alpha^{\star}  g 
&= V \cdot \nabla_X g +  \frac{\alpha}{\eps} \Omega \partial_\Theta g\\
& + \frac1\alpha \int _{[0,L_{ \alpha}] \times {\mathbb R}^2}  M_{\beta } (v)  
 \Big(  \indc_{(\ref{good-collision1})(\ref{good-collision2})}(Y) g(Y') \big(  (\frac1{\alpha}v -V- \Omega      r_{ \Theta}^\perp) \cdot    n_{ \Theta}\big)_-    \\
  &\qquad \qquad \qquad -   \indc_{(\ref{good-collision3})}g(Y) \big(  (\frac1{\alpha}v
-V- \Omega    r_{ \Theta}^\perp
\big ) \cdot    n_{ \Theta} \big)_-\Big)\, d  \sigma_\alpha dv\, .
\end{aligned}
\end{equation*}
As (\ref{good-collision1}), (\ref{good-collision2}) imply (\ref{good-collision3}),
the indicator function can be factorized leading to the expression~\eqref{eq: generateur non symmetrique} for the adjoint.
}

\medskip

To prove  \eqref{eq: comparison processes}, we are going to build a coupling of the processes $\bar Y, \bar Y^\dagger$. Both processes start with the same initial data and have the same updates up to the collision time such that~(\ref{good-collision1}) or (\ref{good-collision2})  no longer hold.

First of all,  recall that the analysis of the atom-rigid body interaction in Section  \ref{conditioning-geometry} shows that if $(v,V,\Omega, n, \Theta)$ does not satisfy (\ref{good-collision3}), then one of the following conditions is  violated
\begin{align}
\label{eq: condition 1}
& |\Omega | < |\log \alpha|\,, \qquad |V | < |\log \alpha|  \,,  \\
\label{eq: condition 2}
&|V'-V| = {2\alpha^2\over A+1} \big| \big( \frac1{\alpha}v
-V- \Omega    r_{ \Theta}^\perp
\big ) \cdot    n_{ \Theta}\big| \geq \alpha^{2+\eta}\,,  \\
\label{eq: condition 3}
& |v-\alpha V| \geq \alpha^{2/3+\eta}\,.
\end{align}

\medskip

Since the measure  $\bar M_{\beta, I}$ is invariant for the process $\bar Y$,
we can proceed as in the first step of the proof  of Proposition \ref{lem: small set} and show that 
with probability much larger than $1-\alpha^{3 \eta}$ the process $\bar Y$ will remain in the set 
\begin{equation}
\label{eq: good set process}
\big \{  |\Omega | < \frac{1}{2} |\log \alpha|\,, \quad |V | < \frac{1}{2}  |\log \alpha| \big\}
\end{equation}
during the time interval $[0,T]$. The initial data of both processes will also belong to this set.

The process $\bar Y$ is obtained by drawing random times with 
update rates  bounded by 
$$
\sup_{ |V|, |\Omega| \leq |\log \alpha|/2 }
\;
\frac1\alpha \int _{[0,L_{ \alpha}] \times {\mathbb R}^2}  M_{\beta } (v)    
\big(  (\frac1{\alpha}v -V- \Omega    r_{ \Theta}^\perp
\big ) \cdot    n_{ \Theta} \big)_- \, d  \sigma_\alpha dv 
\leq \frac{c}{\alpha^2} \,,
$$
for some constant $c>0$, i.e. by a Poisson process with rate $c /\alpha^2$.
For the process $\bar Y$, the probability density to change from a configuration
$(V,\Omega)$ into~$(V',\Omega')$ is given by 
\begin{align*}
\frac{
M_{\beta } (v)  \big(  ( v
- \alpha V- \alpha \Omega    r_{ \Theta}^\perp
\big ) \cdot    n_{ \Theta} \big)_-  }
{ \displaystyle \int _{[0,L_{ \alpha}] \times {\mathbb R}^2}  M_{\beta } (v)    
\big(  ( v -\alpha V-  \alpha \Omega    r_{ \Theta}^\perp
\big ) \cdot    n_{ \Theta} \big)_- \, d  \sigma_\alpha dv }\,\cdotp
\end{align*}
The probability of violating the event \eqref{eq: condition 2} under this measure is bounded by $\alpha^{2+2\eta}$ 
thanks to the following estimate
$$
\int_0^{\alpha^{1+\eta}}  u du = \frac{1}{2} \alpha^{2+2\eta}\,.
$$
The probability of violating the event \eqref{eq: condition 3} 
is bounded by $\alpha^{2+3\eta}$. 
Combining the previous estimates, we deduce that the coupling  fails with a probability less than 
$$
\hat {\mathbb P} \Big( \exists t \leq T, \quad \bar Y( t) \not = \bar Y^\dagger ( t) \Big)
\leq  C \alpha^{2 \eta}\,.
$$
This concludes the proof of Lemma~\ref{lem: markov tronque}. 
\end{proof}

{\color{black}
Proposition~\ref{prop: truncatedBoltztoBoltz} can be deduced from the previous analysis.  
Indeed since $\tilde g_\eps$ is the density of a Markov chain, the 
 maximum principle holds. This leads to  the uniform bound 
$$
\| \tilde g_\eps \| _{L^\infty(\R^+\times \R^4\times {\mathbb S} \times \R)} \leq C_0 \, .
$$
The previous bound can also be understood by using duality and then 
applying the maximum principle for the operator \eqref{eq: generateur non symmetrique}.
Finally,  the difference
 $$
  \|M_{\beta,I} (g_\eps - \tilde g_\eps )\|_{L^\infty([0,T] ;  L^1( \R^4\times  {\mathbb S} \times \R))} \leq C_T \alpha ^{2\eta}
$$
can be deduced from \eqref{eq: comparison processes} by applying the inequality at any given time $t \leq T$.} 

\medskip
Proposition~\ref{prop: truncatedBoltztoBoltz} is proved.\qed

\section{Technical estimates}
\label{technicalestimates}

In this section, we collect some  estimates that were used in the core of the proof of the main theorem:   Paragraph~\ref{subsec : continuity estimates} is devoted to the proof of  some rather well-known continuity bounds on the collision integrals, and finally an estimate showing the  convergence of the initial data is provided in Paragraph~\ref{subsec : initial data}.

\subsection{Continuity estimates}
\label{subsec : continuity estimates}

The following proposition is a precised version of Proposition~\ref{estimatelemmacontinuity 0}.  
\begin{Prop}
\label{estimatelemmacontinuity} 
There is a constant~$C_1= C_1(\beta,I)$  such that for all $s,n\in \N^*$ and all~$h,t\geq 0$, the operators~$|Q|$   satisfy the following continuity estimates:  
$$
	\begin{aligned}
	& |Q_{1,s}| (t)  \big(
	 M_{\beta,I,N}^{(s)} \big)  \leq \Big({C_1 t\over \alpha^2} \Big) ^{s-1} M_{3\beta/4,I} \\
	& |Q_{1,s }| (t) \,  |Q_{s,s+n} | (h) \big(
	 M_{\beta,I,N}^{(s+n)}  \big) \! \leq \!   \Big({C_1 \over \alpha^2} \Big) ^{n+ s-1} t^{s-1}h^n
	M_{3\beta/4,I}  \, .
	\end{aligned}
$$
Similar estimates hold for~$|\bar Q|$,~$|Q^\dagger|$ and~$|\bar Q^\dagger|$.
\end{Prop}

\begin{proof}[Sketch of proof]
The estimate is simply obtained from the fact that the transport operators preserve the Gibbs measures, along with the continuity of the elementary collision operators~:
\begin{itemize}
\item  the transport operators satisfy the identities
$$
{\mathbf  S}_k (t)   M_{\beta,I,k-1}
=   M_{\beta,I,k-1}  \,.
$$
 \item  the collision operators satisfy the following bounds in the Boltzmann-Grad 
scaling~$N \e  = 1 $  (see \cite{GSRT})
$$
\begin{aligned}
|{C}_{k,k+1} |   M_{\beta,I,k}
  \leq C  \alpha^{-1}  \Big( k \beta   ^{-\frac12} + \sum_{1 \leq i \leq k} |v_i|\Big) 
M_{\beta,I,k-1} ,\\
|{D} _{k,k+1} | M_{\beta,I,k} 
  \leq C \alpha^{-1}  \Big( (\alpha^2\beta)  ^{-\frac12} + |V|+|\Omega| \Big) 
 M_{\beta,I,k-1},
\end{aligned}
$$
almost everywhere. Note that we choose not to track the dependence on~$I$ in the estimates as contrary to the factor~$\beta$ there is no loss in this parameter.
\end{itemize}

Estimating the operator $|Q_{s,s+n}|(h)$ follows from piling together those inequalities (distributing the  exponential weight evenly on each occurence of a collision term).
For the collision operator involving the rigid body, we write
$$
\big( |V|+|\Omega| \big) 
\exp \Big(
-\frac{\beta}{8n} \big(|V|^2 +  I  \Omega^2 \big) \Big) 
\leq 
\sqrt{  \frac{C n }{\beta}} \, \cdotp
$$
For the atoms,  we  notice  that by the Cauchy-Schwarz inequality
\begin{equation}
\label{CS}
\begin{aligned}
\frac1 \alpha \sum_{1 \leq i \leq k} |v_i| \exp\Big(- \frac { \beta } {8n} |V_k|^2\Big) &\leq\frac1 \alpha  \left( k\frac {4n} \beta\right) ^{\frac12} \left( \sum_{1 \leq i \leq k} \frac\beta {4n} |v_i|^2 \exp\Big(- \frac {\beta }{4n} |V_k|^2\Big)\right)^{1/2} \\
&\leq\frac1 \alpha  \Big( \frac{4nk}{e\beta }\Big)^{1/2} \leq\frac1 \alpha  \frac{2}{\sqrt{ e\beta }} (s+n)  \, ,
\end{aligned}
\end{equation}
where the last inequality comes from the fact that  $k \leq s+n$. 
Each collision operator gives therefore at most a loss of $ C \beta  ^{-1/2} \alpha ^{-2}  (s+n) $ together with a loss on the exponential weight, while 
  integration with respect to time provides  a factor $h^n/n!$. By Stirling's formula, we have
$$
{ (s+n)^n\over n!} \leq  \exp \left(  n \log {n+s \over n }  + n\right) \leq \exp ( s+n) \,.
$$
As a consequence, for $\beta' <\beta$
$$
|Q_{s,s+n}| (h) M_{\beta,I,s+n-1}
\leq \left( {C_{\beta, \beta'}h \over  \alpha^2}\right)^{n}   M_{\beta,I,s-1} .
$$
Proposition \ref{estimatelemmacontinuity} follows from this upper bound
and the fact that $M_{\beta,I,N}^{(s)} \leq C^s  M_{\beta,I,s-1}$ for some $C$.
\end{proof}

\subsection{Convergence of the initial data} 
\label{subsec : initial data}

Let us  prove the following result.
\begin{Prop}
\label{exclusion-prop2}
There is a constant~$C>0$ such that  the following holds. For the initial data~$f_{N+1,0 } $ and~$( f_0^{(s)})_{s \geq 1}$ given  in~{\rm(\ref{defdata})} and~{\rm(\ref{dataBoltz})},  there holds as $N\to \infty $, and~$\eps,\alpha\to 0$ in the scaling $N \e =  1   $ with~$\alpha \gg \eps^{ \frac12}$,
	$$
			\Big| \left(f_{N+1,0 } ^{(s)}  - f_0^{(s)}\right) \indc_{{\mathcal D}_{\eps}^{s}} (Y,Z_{s-1})
	\Big|  \leq C ^{s-1}   \frac{\eps}{\alpha^2}     \|g_0\|_{L^\infty}     M_{\beta, I} (V,\Omega) M_{\beta } ^{\otimes (s-1)}(V_{s-1})   \,.
			$$
\end{Prop}
 \begin{proof}
The proof is very similar to the proof of Proposition 3.3 in~\cite{GSRT}, and it is an obvious consequence of the following estimate
\begin{equation}
\label{estimateonmaxwellians}
 	\Big| \left(M_{\beta,I , N }^{(s)}   -  \bar M_{\beta, I}  M_{\beta }^{\otimes (s-1)}  \right) \indc_{{\mathcal D}_{\eps}^{s}} (Y,Z_{s-1}) 
	\Big|  \leq C ^{s-1}   \frac{\eps}{\alpha^2}         M_{\beta, I} (V,\Omega) M_{\beta } ^{\otimes (s-1)}(V_{s-1})  \end{equation}
which we shall now prove.

Let us start by proving, as in~\cite{GSRT}, that  in the scaling $N \e \equiv 1,$ with~$\alpha \gg \varepsilon^{\frac12}$,
\begin{equation}
\label{Z-est}
  1 \leq   \cZ_N^{-1}  \cZ_{N-s} \leq\big( 1 - C\e     \big)^{-s} \, .
\end{equation}
 The first inequality is due to the immediate upper bound 
 $$
   \cZ_N \leq    \cZ_{N-s} \, .
 $$
Let us prove the second inequality. We have  by definition
$$
 \begin{aligned} 
    {\mathcal Z}_{s+1} : =  \int  \Big(\prod_{1 \leq i \neq j \leq s+1} \indc_{|x_i-x_{j} |>\eps}\Big)
       \Big(\prod_{1 \leq\ell \leq s+1} \indc_{d (x_\ell, X+\frac\eps\alpha R_\Theta  \Sigma_\alpha)>0}\Big)  \, dX_{s+1} dX \, .
\end{aligned}
$$ 
By Fubini's equality, we deduce
 $$ \begin{aligned}   {\mathcal Z}_{s+1}  = \int  \left( \int_{ \T^{2}  } \Big(\prod_{1 \leq i \leq s} \indc_{|x_i-x_{s+1} |>\eps}\Big)  \,  \indc_{d (x_{s+1}, X+\frac\eps\alpha R_\Theta  \Sigma_\alpha)>0}\, dx_{s+1}\right) \Big(\prod_{1 \leq i\neq j \leq s}   \indc_{|x_i-x_{j} |>\eps}\Big)   \\
 \times  \Big(\prod_{1 \leq\ell \leq s} \indc_{d (x_\ell, X+\frac\eps\alpha R_\Theta  \Sigma_\alpha)>0}\Big)   \, dX_{s}dX  \, .\end{aligned}$$
 One has
 $$ \int_{ \T^{2}  } \Big(\prod_{1 \leq i \leq s} \indc_{|x_i-x_{s+1} |>\eps}\Big)\indc_{d (x_{s+1}, X+\frac\eps\alpha R_\Theta  \Sigma_\alpha)>0}\, dx_{s+1}\geq 
1- \left( \k s \eps^2 +   C_\alpha\left( \frac\eps\alpha\right)^2\right)   \, ,$$
   where~$\k$ is the volume of the unit ball and~$  C_\alpha$ the volume of~$  \Sigma_\alpha$. Since as~$\alpha\to 0$,~$  C_\alpha $ converges to  the volume of~$  \Sigma$, we deduce  from the fact that~$\alpha \gg \eps^{\frac12}$,    $s \leq N$ and the scaling $N \e \equiv 1$ the lower bound as~$N \to \infty$ and~$\eps, \alpha \to 0$:
$$   {\mathcal Z}_{s+1}   \geq
  {\mathcal Z}_{s}   
(1-C \eps    ) \, .
$$ This implies by induction
$$
  {\mathcal Z}_N  \geq   {\mathcal Z}_{N-s}   \big( 1 - C\e  \big)^s  
$$
 and proves~(\ref{Z-est}). Now 
 writing
 $$
dZ_{(s ,N)}  :=dz_{s } \dots dz_N \, ,
$$
we compute for $s \leq N$ 
$$
M_{\beta,I,N}^{(s)} (Y,Z_{s-1})  =   {\mathcal Z}_N^{-1}   \tilde {\mathcal Z} _{(s,N)} (Y,Z_{s-1}) \indc_{ {\mathcal D}_\eps^{s}}(Y,Z_{s-1})  \bar M_{\beta, I} (Y)  M_{\beta } ^{\otimes (s-1)}(V_{s-1})  \, ,
$$
where
$$  
 \begin{aligned}
  \tilde {\mathcal Z} _{(s,N)} (Y,Z_{s-1}):=  \int  \Big(  \prod_{s  \leq i \neq j \leq N} \indc_{|x_i - x_j| > \eps} \Big) \Big(  \prod_{i '\leq s-1 <  j'}  \indc_{|x_{i'} - x_{j'}| > \eps} \Big)  \\
  {}\times     \Big(\prod_{\ell = s }^N \indc_{d (x_\ell, X+\frac\eps\alpha R_\Theta  \Sigma_\alpha)>0}\Big) \, 
dX_{(s ,N)} \, .
  \end{aligned}$$
  We deduce that
 $$  \begin{aligned} 
\Big( M_{\beta,I,N}^{(s)} (Y,Z_{s-1})  -  \bar M_{\beta, I}  (Y) M_\beta^{\otimes (s-1)}(V_{s-1})  \Big) 
 \indc_{ {\mathcal D}_\eps^s}(Y,Z_{s-1})
 =  \bar M_{\beta, I}  (Y)  M_\beta^{\otimes (s-1)}(V_{s-1})   \\
{}\times  \indc_{ {\mathcal D}_\eps^s}(Y,Z_{s-1})\Big(
{\mathcal Z}_N^{-1}  \tilde {\mathcal Z} _{(s,N)} (Y,Z_{s-1}) -1  \Big)\, . \end{aligned}
$$
Next defining
 $$
\begin{aligned} 
 \bar{\mathcal Z}_{N-s+1}& :=   \int  \prod_{s  \leq i\neq j \leq N}\indc_{|x_i- x_j| > \eps}    \, 
dX_{(s ,N)} \quad \mbox{and} \\
 \bar{\mathcal Z}^\flat_{(s ,N)} (Y,Z_{s-1})&:=\sum_{\ell = s }^N \int  \Big( \indc_{ x_\ell \in X+\frac\eps\alpha R_\Theta  \Sigma_\alpha } \Big)
    \prod_{i' \leq s-1 < j'} \indc_{|x_{i'} - x_{j'}| > \eps }  \prod_{s  \leq i\neq j \leq N} \indc_{|x_i- x_j| > \eps} \,dX_{(s,N)} \\
&+ \sum_{i' \leq s-1 < j'}  \int  \indc_{|x_{i'} - x_{j'}| \leq \eps }   \prod_{s  \leq i\neq j \leq N} \indc_{|x_i- x_j| > \eps} \prod_{\ell = s }^N \indc_{d (x_\ell, X+\frac\eps\alpha R_\Theta  \Sigma_\alpha)>0} \,dX_{(s,N)} \, ,
 \end{aligned}
$$
we have
$$ 
 \bar{\mathcal Z}_{N-s+1} \geq 
\tilde{\mathcal Z}_{(s ,N)}(Y,Z_{s-1}) \geq  \bar{\mathcal Z}_{N-s+1} - \bar{\mathcal Z}^\flat_{(s ,N)} (Y,Z_{s-1})\,.
 $$
It remains to prove that~${\mathcal Z}_N^{-1}  \bar{\mathcal Z}_{N-s+1} = 1+O(\eps/\alpha^2)$ and~${\mathcal Z}_N^{-1} \bar{\mathcal Z}^\flat_{(s ,N)} (Y,Z_{s-1}) = O(\eps/\alpha^2)$.
We recall that as proved in~\cite{GSRT},   in the scaling $N \e \equiv 1,$  there holds
$$
1\geq \bar {\mathcal Z}_{s+1}   \geq
\bar  {\mathcal Z}_{s}   
(1-C \eps    )  \, ,
$$
So on the one hand~$ {\mathcal Z}_{N-s+1} \leq \bar {\mathcal Z}_{N-s+1} $, and
$$
 \begin{aligned} 
 {\mathcal Z}_{N-s+1}&  \geq  \int    \prod_{s  \leq i\neq j \leq N} \indc_{|x_i- x_j| > \eps} \,dX_{(s,N)} dX  \\
&\qquad\qquad\qquad\qquad {} - \sum_{\ell = s }^N \int \Big( \indc_{ x_\ell \in X+\frac\eps\alpha R_\Theta  \Sigma_\alpha } \Big)  \prod_{s  \leq i\neq j \leq N} \indc_{|x_i- x_j| > \eps} \,dX_{(s,N)} \\
&\geq  \bar{\mathcal Z}_{N-s+1} +  O\Big ((N-s+1)  \frac{\eps^2}{\alpha^2} \bar{\mathcal Z}_{N-s} \Big ) \\
 & \geq      \bar{\mathcal Z}_{N-s+1} \Big (  1 + O  (  \frac{\eps }{\alpha^2}   ) \Big ) \,.
 \end{aligned}
$$
Thus
$$
  {\mathcal Z}_{N-s+1} 
=    \bar{\mathcal Z}_{N-s+1} \Big (  1 + O  (  \frac{\eps }{\alpha^2}   ) \Big) \,.
$$
Then
$$
{\mathcal Z}_N^{-1}   \bar{\mathcal Z}_{N-s+1} ={\mathcal Z}_N^{-1} {\mathcal Z}_{N-s+1} \Big( 1+  O  (  \frac{\eps }{\alpha^2}   )\Big)
$$
so thanks to~(\ref{Z-est}) we find
$$
{\mathcal Z}_N^{-1}   \bar{\mathcal Z}_{N-s+1} =  1 +  O  (  \frac{\eps }{\alpha^2}  + C^s \eps )\,.
$$
Finally to conclude the proof we notice that 
 $$
 \begin{aligned} 
 0\leq  \bar{\mathcal Z}^\flat_{(s ,N)}(Y,Z_{s-1})&\leq 
 C\Big(  (N-s+1) \frac{\eps^2}{\alpha^2} +(s-1)(N-s)   \eps^2\Big)  \bar{\mathcal Z}_{N-s} 
  \end{aligned}
 $$
 so again
 $$
 {\mathcal Z}_N^{-1}    \tilde{\mathcal Z}^\flat_{(s ,N)}(Y,Z_{s-1})  =  
 O  (  \frac{\eps }{\alpha^2}  + C^s \eps )\, .
 $$
The result follows.
\end{proof}

\section{Technical geometrical results}

In this section, we provide, in Paragraph~\ref{geometricalresults}, a few  geometrical computations useful for the study of recollisions. 
We also justify   in Paragraph~\ref{prooflawsmotion} the collision laws stated in \eqref{eq: collision laws 0}.

\subsection{Pre-images of rectangles by   scattering}
\label{geometricalresults}
 
 In the proof of Proposition~\ref{geometric-prop}, there is a need to translate the condition that outgoing velocities   belong to some given set (typically a rectangle) into a condition on the incoming velocity and deflection angle (which are the integration parameters).

We first consider the case of a collision involving two atoms and start by recalling 
Carleman's parametrization which
relies on the following representation of the scattering:
\begin{equation}
\label{defcarleman}
(v^*,\nu^* )\in \R^2 \times {\mathbb S} \mapsto 
\begin{cases}
 v' _* := v ^*- (v ^* - \bar  v) \cdot \nu ^* \nu  ^*\\
 v' := \bar v + (v ^*-  \bar v) \cdot \nu ^* \nu ^* \, ,
\end{cases}
\end{equation}
where $(v',v'_*)$ belong to the set ${\mathcal C}$ defined by 
$${\mathcal C}:=\Big \{(v',v'_*)\in \R^2\times \R^2 \,/\, (v'-\bar v)\cdot (v'_*-\bar v) =0\Big\} \, .$$
This map sends the measure $ |(v ^* - \bar  v) \cdot \nu ^* |  \,  dv^*d\nu^*$ on the 
measure $dv'dS(v'_*)$, where $dS$ is the Lebesgue measure  on the line
orthogonal to $(v'-\bar v)$ passing through $\bar v$. 

Using Carleman's parametrization the following control on the scattering has been derived in \cite{BGSR2}.

\begin{Lem}
\label{aa-scattering}
Let $\cR$ be a rectangle with sides of length $\delta, \delta' $, then 
\begin{align}
\int  _{B_R \times \mathbb{S}}  \indc_{\{ v'  \in \cR \hbox{ or } v'_* \in \cR \} }  \; \big| (v^* -\bar v  ) 
\cdot \nu^* \big|  \,  dv^* d\nu^*
\leq C R^2 \min (\delta, \delta')     \big( |\log \delta| + |\log \delta' | + 1\big) \,.
\label{rectangle0}
\end{align} 
\end{Lem}

We refer to \cite{BGSR2} for a proof. 
We are now going to extend this result to the case of a collision between an atom and the rigid body (see Lemma \ref{aa-scattering}).
We use the following notation: the collision takes place at a point
of arc-length~$\nu^*$  and we denote by~$n_\Theta = R_\Theta n$ the corresponding unit outward normal at that point (on the rotated rigid body) and by~$r_\Theta = R_\Theta r$ the vector joining the center of mass~$G$ to the impact point~$P$ after rotation and rescaling. Finally the velocities at collision are given by
\begin{equation}\label{scatteringformulasappendix}
\begin{aligned}
 v'-\alpha V &= v^* - \alpha V   + {2 \over A+1} (\alpha V_P- v^* ) \cdot  n_{ \Theta } \,  n_{ \Theta } \\
  V' -   V  &=   {2\alpha \over A+1}  (\alpha V_P - v^* ) \cdot n_{ \Theta } \,  n_{ \Theta } \,,
\end{aligned}
\end{equation}
with~$V_P:= V+  \Omega \,  r_{ \Theta }^\perp$.
\begin{Lem}
\label{am-scattering}
Let $\cR$ be a rectangle with sides of length $\delta, \delta'$, then 
$$
\begin{aligned}
&\int  _{B_R \times [0,L_\alpha]}  \indc_{\{ V'  \in \cR \hbox{ or } v' \in \cR \} }  \; \big| (v^* -\alpha V_P ) \cdot  n_{\Theta } \big|  \,  dv^* d\nu^* \\
& \qquad \qquad \qquad \qquad \qquad   \leq   \frac {CR^2 }{\kappa_{min} }  \left(\frac{A+1}{2\alpha   } \right)^2  \min (\delta, \delta')     \big( |\log \delta| + |\log \delta' | + 1\big) \,.
\end{aligned}
$$
\end{Lem}
\begin{proof}
The first step of the proof consists in  deriving  a counterpart of Carleman's parametrization~(\ref{defcarleman}) by finding a change of measure from~$(v^*, \nu^* )$ to~$(V', v')$. 
We start by projecting~$v^*$ onto~$n^\perp_{ \Theta } $ and~$n_{ \Theta } $: this gives
$$  d v^* = 
d (v^* \cdot n_{ \Theta }  ) d (v^* \cdot n^\perp_ { \Theta }  ) 
$$
and then by translation invariance we can write
$$
d v^* = 
d\big((v^*-\alpha V_P)\cdot n_{ \Theta } \big) d\big((v^*-\alpha V)\cdot n^\perp_ { \Theta } \big) \,.
$$
The   formulas~(\ref{scatteringformulasappendix}) then give 
$$
d v^* = \frac{A+1}{2\alpha  }d (  |V'-V| ) d\big((v'-\alpha V)\cdot n^\perp_ { \Theta } \big) \,.
$$
Then we notice that
$$
 |V'-V| d (  |V'-V| ) dn_\Theta = d (V'-V) = dV'
$$
and
$$
 |V'-V| =  {2\alpha \over A+1} | (\alpha V_P - v^* ) \cdot  n_{ \Theta } |
$$
so finally there holds
$$
\begin{aligned}
 | (\alpha V_P - v^* ) \cdot  n_{ \Theta } | dn_{ \Theta }   d v^*  =  \left(\frac{A+1}{2\alpha  } \right)^2dV'  d \mu_{\alpha V} ( v') 
\end{aligned}$$
where~$ d \mu_{\alpha V}   $ is the Lebesgue measure on   the line  orthogonal to~$ n_{  \Theta }$  passing through  $\alpha V$.
It follows that~$ {4\alpha^2 \over (A+1)^2}| (\alpha V_P - v^*) \cdot  n_{ \Theta } | dv^* d n_{ \Theta }$ is mapped to~$dV' d \mu_{\alpha V} ( v')  $. Noticing that    $$
  dn_{ \Theta }  = - \kappa (\nu^* )   d\nu^*  \, ,
  $$
    where we recall that~$  \kappa (\nu )$ is the   curvature of the boundary of~$  \Sigma_\alpha$ at the point determined by the arc-length~$\nu $, the change of measure from~$ dn_{ \Theta } $ to~$ d\nu ^*$ has therefore  jacobian~$ \kappa (\nu^* ) ^{-1}$.
So finally we obtain
\begin{equation}
\label{formulacarlemanrigid body}
 \kappa (\nu^* )   | (\alpha V_P - v^* ) \cdot  n_{ \Theta } |  d\nu^*  d v^*  =  \left(\frac{A+1}{2\alpha  } \right)^2dV'  d \mu_{\alpha V} ( v') \,.
\end{equation}
Now let us turn to the proof of the lemma, following the proof of Lemma~\ref{aa-scattering} which can be found in~\cite{BGSR2}. 
Suppose $\delta' > \delta$.
Estimating the measure of the event $\{ V' \in \cR\}$ is straightforward by the change of variable 
\eqref{formulacarlemanrigid body}. Thus we are going to focus on the event $\{ v' \in \cR\}$ and first start by 
estimating the measure that $v'$ belongs to a small ball of given center, say~$w$ and radius~$\delta>0$.
This estimate will be used by covering the rectangle $\cR$ by such small balls. 
We distinguish two cases.


If  $|w - \alpha V| \leq \delta $, meaning that $\alpha V$ is itself in the same ball, then for any $V' \in B_R$, the intersection between the small ball and the line $ \alpha V+ \R n_{ \Theta } ^\perp$ is  a segment, the length of which is at most 
$\delta$. We therefore find 
$$ \int \indc_{|v'   -w|\leq \delta}  \,\big| (v^* -\alpha V_P ) \cdot  n_{\Theta } \big| \,    dv ^* d\nu ^* \leq  \frac C{\kappa_{min}  }  \left(\frac{A+1}{2\alpha  } \right)^2R^2 \delta \,.$$

\medskip

 If $|w -  \alpha V|  >\delta$, in order for  the intersection between the ball and the line $\alpha V + \R n_{ \Theta } ^\perp$ to be non empty, we have the additional condition that $\alpha V'-\alpha V$ has to be in an angular sector of size $\delta /| w-\alpha V|$. We therefore have 
$$ \int \indc_{|v'   -w|\leq \delta}  \, \big| (v^* -\alpha V_P ) \cdot  n_{\Theta } \big|  \,   dv ^* d\nu ^* \leq  
 \frac C{\kappa_{min}  }  \left(\frac{A+1}{2\alpha  } \right)^2R^2   {\delta^2\over  |w-\alpha V|}  \,\cdotp$$

Combining both estimates, we get finally
\begin{equation}
\label{preimage-sphere1}
\int \indc_{|v'  -w|\leq \delta}  \, \big| (v^* -\alpha V_P ) \cdot  n_{\Theta } \big|  \,    dv ^* d\nu ^* \leq  \frac C{\kappa_{min} }  \left(\frac{A+1}{2\alpha  } \right)^2 R^2 \delta \min \Big (1, {\delta\over |w-\alpha V|}\Big) \,.
\end{equation}

\medskip

Now let us prove Lemma~\ref{am-scattering}. We suppose to simplify that $\delta \leq \delta' \leq 1$.
We cover the rectangle $\cR$ into~$\lfloor \delta'/\delta \rfloor$ balls of radius~$2 \delta$.
Let $\omega$ be the axis of the rectangle  $ \cR$ and denote by $w_k = w_0 +  \delta k \, \omega$ the centers of the balls which are indexed by the integer $k \in \{ 0,  \dots,  \lfloor \delta'/\delta \rfloor \}$.
Applying \eqref{preimage-sphere1} to each ball, we get 
\begin{eqnarray}
 \int    \indc_{ v'  \in \cR} \; \big| (v^* -\alpha V_P ) \cdot  n_{\Theta } \big|  \,    dv ^* d\nu ^* 
&\leq&  \sum_{k = 0}^{ \lfloor \delta'/\delta \rfloor}
\int    \indc_{ |v'  - w_k| \leq 2 \delta}  \; \big| (v^* -\alpha V_P ) \cdot  n_{\Theta } \big|  \,    dv ^* d\nu ^* \nonumber
\\
&\leq&    \frac C{\kappa_{min} }  \left(\frac{A+1}{2\alpha  } \right)^2  R^2 \sum_{k = 0}^{ \lfloor \delta'/\delta \rfloor} \delta \min \Big( {\delta \over |w_k -\alpha V|}, 1 \Big) ,
\nonumber\\
&\leq&     \frac C{\kappa_{min}  }  \left(\frac{A+1}{2\alpha  } \right)^2   R^2 \delta \sum_{k = 0}^{ \lfloor \delta'/\delta \rfloor}   {\delta \over |w_k -\alpha V| + \delta }\nonumber\\
&\leq&     \frac C{\kappa_{min}  }  \left(\frac{A+1}{2\alpha  } \right)^2    R^2 \delta \Big( \log (\frac{\delta'}{\delta}) +1 \Big)\, , \nonumber
\end{eqnarray}  
where the log divergence in the last inequality follows by summing over $k$.
This completes the proof of the lemma.
\end{proof}

\subsection{Collision laws}
\label{prooflawsmotion}

In this section we  recall how relation \eqref{eq: collision laws 0} can be derived from the collision invariants.
First note that  the collision produces a force in the normal direction $n_\Theta$. Since this force is a Dirac mass in time  this produces jump conditions. 
The momenta after the collision become
\begin{align}
\label{eq: momenta change}
M V' - M V = - f n_\Theta
\quad \text{and} \quad
m \hat v' - m \hat v =  f n_\Theta \,,
\end{align}
where~$f$,  the amplitude of the force, is an unknown.
When the impact is at the point $\frac\eps\alpha r_\Theta$ of~$\frac\eps\alpha \Sigma$,
the angular momentum   changes as
\begin{align}
\label{eq: momenta angular}
\hat  I  \hat\Omega'  - \hat  I  \hat\Omega = 
  -   \frac\eps\alpha  f  \, n_\Theta \cdot r^\perp \quad \Rightarrow \quad
\hat\Omega'  -  \hat\Omega =  -  \frac\eps\alpha f   \hat  I^{-1} \, n_\Theta \cdot r^\perp   \,.
\end{align}
As the atom is a sphere, its angular momentum is unchanged (the force $f \, n_\Theta$ is collinear to the direction between the
center of the sphere and the collision point).
Finally, the conservation of the total energy provides a last equation
\begin{align}
\label{eq: energy}
\frac12 \big(m|\hat v'|^2 +M |V'|^2 \big) +\frac12 \hat  I  \hat\Omega '^2
=
\frac12 \big( m|\hat v|^2 +M |V|^2 \big) +\frac12 \hat I  \hat\Omega^2 \, .
\end{align}
Since the angular momentum of the atoms is constant, it is not taken into account in the
energy conservation.

In order to determine $(V',\Omega',f)$ from the previous equations, we first plug
\eqref{eq: momenta change} and \eqref{eq: momenta angular} in \eqref{eq: energy}
to identify $f$
\begin{align*}
- f (V \cdot n_\Theta) +  \frac{f^2}{2 M} + f (\hat v \cdot n_\Theta) +  \frac{f^2}{2 m}
- f  \frac\eps\alpha \, \hat \Omega    ( n_\Theta \cdot r^\perp) +
\frac{f^2}{2}   \left(\frac\eps\alpha  \right)^2 \, \hat  I^{-1} ( n_\Theta \cdot r^\perp   )^2 
= 0 \,.
\end{align*}
As $f \neq 0$, the solution is
\begin{align*}
f
=  \frac{2 m} {A +1}\big(V  + \frac\eps\alpha \hat\Omega  r_\Theta^\perp    -\hat  v \big)  \cdot  n_\Theta \, ,
\end{align*}
where $V  + \frac\eps\alpha \hat\Omega  r_\Theta^\perp$ is the velocity of the impact point in $\Sigma$ as defined in~\eqref{vitesse locale} and
$A$ is given in \eqref{eq: A 0}.
Since the force is in the normal direction, we get from \eqref{eq: momenta change} that the tangential components 
are constant.
 The normal component and the angular momentum
are deduced by the value of $f$.
This completes the derivation of the collision laws  \eqref{eq: collision laws 0}.

\bigskip

\noindent
{\bf Acknowledgements.} We would like to thank F. Alouges for enlightening discussions on solid reflection laws. T.B. would like to thank the support of ANR-15-CE40-0020-02 grant LSD.

\end{document}